\documentclass[]{article}
\usepackage[a4paper,left=25mm,right=30mm,top=30mm,bottom=35mm]{geometry}
\usepackage{hyperref}
\usepackage{titlesec}
\usepackage{graphicx}
\usepackage{nomencl}
\makenomenclature
\usepackage{amsmath}%
\allowdisplaybreaks
\usepackage{color}
\usepackage{amsthm}
\usepackage{amsfonts}%
\usepackage{amssymb}%
\usepackage{graphicx}

\usepackage{paralist}
\usepackage[onehalfspacing]{setspace}
\usepackage{tikz}
\usepackage{color}
\usepackage{esint}
\usepackage{lmodern}

\theoremstyle{plain}

\newtheorem{definition}{Definition}[section]

\newtheorem{lemma}[definition]{Lemma}

\newtheorem{proposition}[definition]{Proposition}
\newtheorem{remark}[definition]{Remark}

\newtheorem{theorem}[definition]{Theorem}
\newtheorem{corollary}[definition]{Corollary}

\normalsize

\renewcommand{\Re}{\textnormal{Re}}

\newcommand{\gast}{\nabla_{\ast}}

\newcommand{\gasta}[1]{\nabla_{\Gamma_{\ast}^{#1}}}

\newcommand{\Dast}{\Delta_{\ast}}
\newcommand{\Dasta}[1]{\Delta_{\Gamma_{\ast}^{#1}}}
\newcommand{\Dasti}{\Delta_{\Gamma_{\ast}^i}}
\newcommand{\partialnuasta}[1]{\partial_{\nu_{\ast}^{#1}}}
\newcommand{\nuast}{\nu_{\ast}}
\newcommand{\trsigma}[1]{\boldsymbol{#1}\big|_{\Sigma_{\ast}}}

\newcommand{\brho}{\boldsymbol{\rho}}
\newcommand{\bmu}{\boldsymbol{\mu}}
\newcommand{\bw}{\boldsymbol{w}}
\newcommand{\bu}{\boldsymbol{u}}
\newcommand{\bv}{\boldsymbol{v}}

\newcommand{\ddeps}{\frac{d}{d\varepsilon}}
\newcommand{\Stricheps}{\big|_{\varepsilon=0}}

\DeclareMathOperator{\dist}{dist}

\DeclareMathOperator{\divergenz}{div}

\newcommand{\dH}{d\mathcal{H}}
\newcommand{\R}{\mathbb{R}}

\newcommand{\N}{\mathbb{N}}
\newcommand{\C}{\mathbb{C}}
\newcommand{\mb}[1]{\mathbb{#1} }

\newcommand{\pr}{\text{pr}}
\newcommand{\supp}{\text{supp}}

\newcommand{\Gammaast}{\Gamma_{\ast}}
\newcommand{\Gammaastdelta}{\Gamma_{\ast,\delta}}
\newcommand{\GammaastT}{\Gamma_{\ast,T}}
\newcommand{\Sigmaast}{\Sigma_{\ast}}
\newcommand{\Sigmaastdelta}{\Sigma_{\ast,\delta}}

\newcommand{\Gastzwei}{\Gamma_{\ast}^2}

\newcommand{\Nast}{N_{\ast}}
\newcommand{\tauast}{\tau_{\ast}}

\newcommand{\dalpha}{\delta^{\bar{\alpha}}}
\definecolor{light-gray}{rgb}{0.23, 0.27, 0.29}

\title{On the Surface Diffusion Flow with Triple Junctions\\ in Higher Space Dimensions}
\author{H. Garcke\footnote{University of Regensburg, 93040 Regensburg,
		harald.garcke@mathematik.uni-regensburg.de}, M. G\"o\ss{}wein\footnote{University of Regensburg, 93040 Regensburg,  michael.goesswein@mathematik.uni-regensburg.de}}
\begin{document}
\maketitle
\begin{abstract}
We show short time existence for the evolution of triple junction clusters driven by the surface diffusion flow. On the triple line we use the boundary conditions derived by Garcke and Novick-Cohen as the singular limit of a Cahn-Hilliard equation with degenerated mobility. These conditions are concurrency of the triple junction, angle conditions between the hypersurfaces, continuity of the chemical potentials and a flux-balance. For the existence analysis we first write the geometric problem over a fixed reference surface and then use for the resulting analytic problem an approach in a parabolic H\"older setting.
\end{abstract}

\section{Introduction}
\noindent
Motion by surface diffusion flows was firstly proposed by Mullins \cite{mullins1957theory} to describe the development of thermal grooves at grain boundaries of heated polycrystalls. This process depends mainly on two physical effects, which are evaporation and surface diffusion, i.e., molecular motion on the surface of heated, solid substances. In the case that surface diffusion is the more involved process Mullins derived that the profile of the surface $\Gamma$ evolves due to
\begin{align}\label{EquationMotionDuetoSurfaceDiffusion}
V_{\Gamma}=-\Delta_{\Gamma}H_{\Gamma}.
\end{align} 
Hereby, $V_{\Gamma}$ denotes the normal velocity, $\Delta_{\Gamma}$ the Laplace-Beltrami operator and $H_{\Gamma}$ the mean curvature operator. For closed hypersurfaces a lot of research was done in the end of the last century. Cahn and Taylor identified the surface diffusion flow as formal $\mathcal{H}^{-1}$-gradient flow of the surface energy, cf. \cite{taylor1994linking}. The evolution is also connected to the Cahn Hilliard equations as Cahn, Elliott and Novick-Cohen proved via formal asymptotics that the surface diffusion flow is its singular limit, cf. \cite{cahn1996cahn}. In the case of closed curves, short time existence was proven by Elliott and Garcke in \cite{elliottgarcke1994existence}. The result was generalized to closed hypersurfaces of arbitrary dimensions by Escher, Mayer and Simonett, cf. \cite{eschermayersimonnett1998surface}. Unlike other geometric flows, e.g., the mean curvature flow, it is very typical for the surface flow to have a loss of convexity and to develop singularities during the evolution, cf. \cite{gigaito1997pinching}, \cite{gigaito1999loss}.\\
In this article we are interested in the evolution of triple junction clusters with respect to surface diffusion flow. A triple junction cluster consists of a set of hypersurfaces such that each connected component of the boundary of a hypersurface is also boundary of two other hpyersurfaces. We will restrict to the case of three connected hypersurfaces $\Gamma^1, \Gamma^2, \Gamma^3$ in $\R^n$ meeting in one triple junction $\Sigma$. The result transfers to more general situations. Additionally, we need that the $\Gamma^i$ are embedded, oriented, compact and do not intersect with each other. We will consider evolutions of such objects with respect to \eqref{EquationMotionDuetoSurfaceDiffusion} such that at every time $t$ we have a decomposition
\begin{align*}
\Gamma(t):=\Gamma^1(t)\cup \Gamma^2(t)\cup\Gamma^3(t)\cup \Sigma(t)
\end{align*} 
into the three hypersurfaces and the triple junction, such that the following boundary conditions are fulfilled:
\begin{align}
\partial\Gamma^1(t)=\partial\Gamma^2(t)&=\partial\Gamma^3(t)=\Sigma(t), \tag{CC}\\
\angle(\nu_{\Gamma^i(t)},\nu_{\Gamma^j(t)})&=\theta^k,\quad (i,j,k)\in \{(1,2,3),(2,3,1),(3,1,2)\},\tag{AC} \\
\gamma^1H_{\Gamma^1(t)}+\gamma^2H_{\Gamma^2(t)}+\gamma^3H_{\Gamma^3(t)}&=0, \tag{CCP}\\
\nabla_{\Gamma^1(t)}H_{\Gamma^1(t)}\cdot\nu_{\Gamma^1(t)}=\nabla_{\Gamma^2(t)}H_{\Gamma^2(t)}\cdot\nu_{\Gamma^2(t)}&=\nabla_{\Gamma^3(t)}H_{\Gamma^3(t)}\cdot\nu_{\Gamma^3(t)}. \tag{FB}
\end{align}
Here, $\nu_{\Gamma^i(t)}$ denotes the outer conormal of $\Gamma^i(t)$, $\gamma^1,\gamma^2,\gamma^3$ constants determining the energy density on the hypersurfaces $\Gamma^i(t)$ and $\theta^1,\theta^2,\theta^3\in(0,2\pi)$ given angles, that are related to the $\gamma^i$. Indeed, the condition (AC) is equivalent to Young's law
\begin{align}
\frac{\sin(\theta^1)}{\gamma^1}=\frac{\sin(\theta^2)}{\gamma^2}=\frac{\sin(\theta^3)}{\gamma^3}.
\end{align}
The condition (CCP) results from continuity of the chemical potentials at the triple junction and (FB) is equivalent to the flux balances. (CC) gives the concurrency of the triple junction during the flow.
\begin{figure}[h]
	\centering
	\includegraphics[height=6cm]{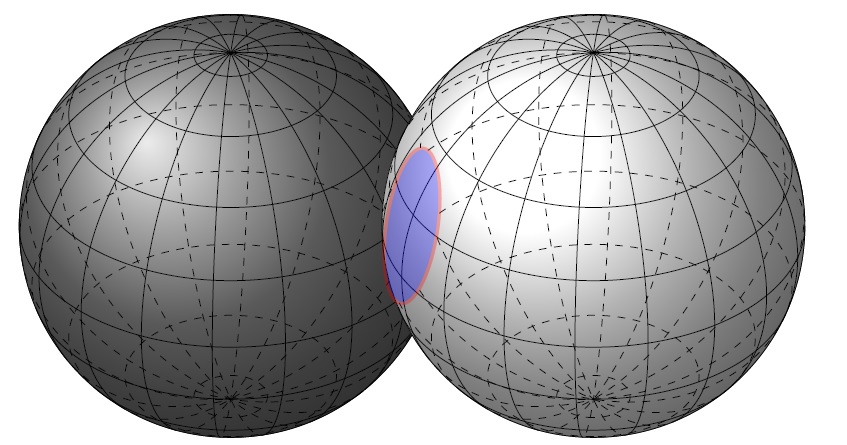}
	\caption{The picture shows the considered kind of triple junction cluster. In total, there are three hypersurfaces. In this illustration these are the two spherical caps and the flat blue area. The red line marks the triple junction, which is the boundary of all three hypersurfaces. Note that if the two enclosed volumes are unequal the blue surfaces will normally bend into direction of the larger volume.
	The shown geometry is the famous solution of the double-bubble conjecture, that is the geometric object with two fixed enclosed volumes and least surface area.}
\end{figure}\\
In this paper, we will show short time existence of the surface diffusion flow with triple junctions. Roughly, we will prove the following.
\begin{theorem}[Short time existence for the surface diffusion flow of triple junction clusters]\ \\ Let $\Gamma_{\ast}\subset\R^n$ be a triple junction cluster of class $C^{5+\alpha}$, such that the three hypersurfaces $\Gamma^i_{\ast}$ are compact, embedded, orientable and connected. Then, for all triple junction clusters $\Gamma_0$ of class $C^{4+\alpha}$, which are close enough to $\Gamma_{\ast}$, there exists a solution to the evolution with respect to surface diffusion flow \eqref{EquationMotionDuetoSurfaceDiffusion}, the boundary conditions (CC), (AC), (CCP), (FB), and initial data $\Gamma(0)=\Gamma_0$. Additionally, the existence time is uniform in $\Gamma_0$ and there is a bound for the distance of $\Gamma(t)$ to $\Gammaast$, which is uniform in $t$ and $\Gamma_0$.
\end{theorem}
We want to note that our results refers to an analytic framework over a reference geometry, which we will introduce later. Due to this reason we call our existence result analytic existence, i.e., existence in the used analytic framework. Whereas we speak about geometric existence and uniqueness when one can actually prove these results for the original geometric problem. In the literature, most existence results for geometric flows of higher dimensional surfaces only give analytic results. For example, Depner, Garcke and Kohsaka proved in \cite{depner2014mean} analytic existence for triple junction clusters evolving due to mean curvature flow. Indeed, the strategy of our proof is based on this work. Abels, Garcke and M\"uller proved analytic existence of volume-preserving mean curvature flow and Wilmore flow with line tension of connected hypersurfaces in $\R^3$ with a free boundary, cf. \cite{abelsgarckemuller2016localwellposednesswilmoreflowlinetension}. In \cite{abelsbutzexistencecurvefidd}, Abels and Butz proved analytic existence for the evolution of open curves with respect to curve diffusion flow, which is the one dimensional pendant to the surface diffusion flow, with a contact angle and rough initial data. Indeed, for curves there are already methods to prove geometric existence and uniqueness. E.g., Spener proved in \cite{spener2017shorttimeexistenceclampedcurves} such an result for clamped curves evolving due to the elastic flow. Garcke and Novick-Cohen proved the same for the one dimensional version of our problem, cf. \cite{garcke2000singular}. Using their ideas, Garcke, Menzel and Pluda proved in \cite{garckemenzelpluda2019willmore} geometric existence and uniqueness for classes of planar networks evolving with respect to Willmore flow. The last two works both work in a setting with parabolic H\"older spaces and use partly the same methods like we do. Due to the higher space dimensions we have to include global weak solution theory and a localization procedure.\\
The article is organized as follows. In Section \ref{SectionPrelim} we will sum up the notation we use and then give an overview of important results for parabolic H\"older and Sobolev spaces on manifolds, especially taking the triple junction cluster into account. In Section \ref{SectionParametrizationCom} we will parametrize our problem over a fixed reference configuration and derive a corresponding analytic problem from that. Additionally, we will state the needed compatibility conditions and the main result of this article. In Section \ref{SectionLinearization} we linearize our equations in the reference configuration and then show in Section \ref{SectionLinearAnalysis} existence of continuous inverse of the linear operator. In order to do so we first find a weak formulation and derive existence of weak solution (Section \ref{SubsectionWeakTheory}), then prove H\"older regularity for the localized problem and connect the localization with the original problem using compactness arguments (Section \ref{SubsectionSchauderEstimates}) and at the end (Section \ref{SubsectionCompleteLinearProblem}) include lower order terms, inhomogeneities and initial data, which we had to omit in the beginning. Finally, we use a fixed point argument to derive our existence result for the original problem from the linear theory.

\section{Preliminaries}\label{SectionPrelim}
\noindent
\subsection{Notation}

In the whole paper $\Gamma(t)$ will refer to the evolving triple junction cluster at time $t$. The triple junction cluster consist of three compact, embedded, connected, orientable hypersurfaces $\Gamma^1(t), \Gamma^2(t), \Gamma^3(t)$ in $\R^{n+1}$ with a common boundary, the triple junction $\Sigma(t)$. Two of the hypersurfaces will always form a volume containing the third hypersurface, which we will choose to be $\Gamma^1$. By $\Omega_{12}$ and $\Omega_{13}$ we denote the volume enclosed by $\Gamma^1$ and $\Gamma^2$ resp. $\Gamma^1$ and $\Gamma^3$. We choose the unit normal fields $N^i$ of the $\Gamma^i$ such that $N^1$ points in the interior of $\Omega_{12}$, $N^2$ outside of $\Omega_{12}$ and $N^3$ into the inside of $\Omega_{13}$. The outer unit conormal of $\Gamma^i$ will be denoted by $\nu^i$. \\
Furthermore, we use the standard notation for quantities of differential geometry, see for example \cite[Chapter 6]{kuehneldifferentialgeometry2015}. That includes the canonical basis $\{\partial_i\}_{i=1,...,n}$ of the tangent space $T_p\Gamma$ at a point $p\in\Gamma$ induced by a parametrization $\varphi$, the entries $g_{ij}$ of the metrical tensor $g$, the entries $g^{ij}$ of the inverse metric tensor $g^{-1}$, the Christoffel symbols $\Gamma^{i}_{jk}$, the second fundamental form $II$, its squared norm $|II|^2$ and the entries $h_{ij}$ of the shape operator. We use the usual differential operators on a manifold $\Gamma$, which are the surface gradient $\nabla_{\Gamma}$, the surface divergence $\divergenz_{\Gamma}$ and the Laplace-Beltrami operator $\Delta_{\Gamma}$.\\
By $\brho=(\rho^1,\rho^2, \rho^3)$ we will denote the evolution in normal direction and by $\bmu=(\mu^1, \mu^2, \mu^2)$ the evolution in tangential direction, which we will use to track the evolution of $\Gamma(t)$ over a reference frame $\Gammaast$ via a direct mapping approach. $\Gamma_{\brho,\bmu}$ will denote the triple junction cluster that is given as graph over $\Gammaast$. We will normally write $\Gamma_{\brho}$ instead of $\Gamma_{\brho,\bmu}$ as $\bmu$ will always be given as function in $\brho$. Details about this parametrization will be explained in Section \ref{SectionParametrizationCom}. Sub- and superscripts $\brho$ resp. $\bmu$ on a quantity will indicate that the quantity refers to the manifold $\Gamma_{\brho}$ resp. $\Gamma_{\brho,\bmu}$. An asterisk will denote an evaluation in the reference geometry. Both conventions are also used for differential operators. For example, we will write $\nabla_{\rho}$ for $\nabla_{\Gamma_{\rho}}$ and $\gast$ for $\nabla_{\Gammaast}$. We will denote by $J_{\rho}$ the transformation of the surface measure, that is,
\begin{align*}
\dH^n(\Gamma_{\rho})=J_{\rho}\dH^n(\Gammaast).
\end{align*} 
If we index a domain or a submanifold in $\R^n$ with a $T$ or $\delta$ in the subscript, this indicated that the corresponding parabolic set, e.g., $\Gamma_T=\Gamma\times[0,T]$. With an abuse of notation, in most parts of the work we will not differ between quantities on $\Gamma_{\rho}$ and the pullback of them on $\Gammaast$. An index $i$ will be used to indicate that a quantity refers to the hypersurface $\Gamma^i(t)$ resp. $\Gammaast^i$. A quantity in bold characters will refer to the vector consisting of the quantity on the three hypersurfaces of a triple junction cluster.\\
The subscript TJ in a function space will indicate that the function space has to be read as product space on each hypersurface. For example, we write
\begin{align*}
L^2_{TJ}(\Gamma):=L^2(\Gamma^1)\times L^2(\Gamma^2)\times L^2(\Gamma^3).
\end{align*} 
We will denote by $\sigma_i$ the order of the $i$-th boundary condition of our problem, i.e.,
\begin{align}\label{EquationDefinitionofSigma}
\sigma_1=0,\quad \sigma_2=\sigma_3=1,\quad \sigma_4=2,\quad \sigma_5=\sigma_6=3.
\end{align}
Finally, we will always use the convention of dynamical constants. This will also be used for dynamical coefficient functions of lower order terms, which will be introduced in Section \ref{SectionLinearization} 

\subsection{Function Spaces}

In this section we want to introduce the two most important function spaces on manifolds we will use, which are Sobolev and parabolic H\"older spaces. Here, $(\Gamma,\mathcal{A})$ will always be a compact, orientable, embedded submanifold $\Gamma$ of $\R^{n+1}$ together with a maximal atlas $\mathcal{A}$. 
\begin{definition}[Sobolev spaces on manifolds]
	Let $\Gamma$ be of class $C^j, j\in\N$. Then we define for $k\in\N, k<j, 1\le p\le \infty$ the Sobolev space $W^{k,p}(\Gamma)$ as the set of all functions $f: \Gamma\to\R$, such that for any chart $\varphi\in\mathcal{A}, \varphi: V\to U$ with $V\subset\Gamma, U\subset \R^n$ the map $f\circ \varphi^{-1}$ is in $W^{k,p}(U)$. Hereby, $W^{k,p}(U)$ denotes the usual Sobolev space. We define a norm on $W^{k,p}(\Gamma)$ by
	\begin{align}\label{EquationNormSobolevSpace}
	\|f\|_{W^{k,p}(\Gamma)}:=\sum_{i=1}^s \|f\circ \varphi_i^{-1}\|_{W^{k,p}(U_i)},
	\end{align}
	where $\{\varphi_i: V_i\to U_i\}_{i=1,...,s}\subset\mathcal{A}$ is a family of charts that covers $\Gamma$. 
\end{definition}
\begin{remark}[Equivalent norms on $W^{k,p}(\Gamma)$]\ \\\vspace{-0,6cm}
	\begin{itemize}
		\item[i.)] The norm on $W^{k,p}(\Gamma)$ depends on the choice of the $\varphi_i$ but for a different choice we will get an equivalent norm as the transitions maps are $C^j$. \vspace{-0,2cm}
		\item[ii.)] For the space $W^{1,p}(\Gamma)$ we will use the norm
		\begin{align}
		\|f\|_{W^{1,p}(\Gamma)}=\left(\int_{\Gamma}|\nabla_{\Gamma}f|^p+|f|^p\dH^n\right)^{\frac{1}{p}}.
		\end{align}
		Equivalence to (\ref{EquationNormSobolevSpace}) follows directly from the representation of the surface gradient in local coordinates. \vspace{-0,2cm}
		\item[iii.)] As usual, we will write $H^k(\Gamma)$ for $W^{k,2}(\Gamma)$.
	\end{itemize}
\end{remark}
As we want to use these spaces in our weak analysis, we will need a version of the Ehrling lemma for triple junction clusters.
\begin{proposition}[Ehrling-type lemma on triple junction]\label{LemmaHilfslemmafurEnergieabschatzungschwacheLt}\ \\
	For every $\varepsilon>0$ there exists a $C_{\varepsilon}>0$ only depending on $\varepsilon$ such that for all $\boldsymbol{u}\in H^{1}_{TJ}(\Gamma)$ it holds
	\begin{align}
	\|\boldsymbol{u}\|_{L^2(\Sigma)^3}\le \varepsilon \|\nabla_{\Gamma}\boldsymbol{u}\|_{L^2_{TJ}(\Gamma)}+C_{\varepsilon}\|\boldsymbol{u}\|_{L^2_{TJ}(\Gamma)}.
	\end{align}
\end{proposition}
\begin{proof}
See \cite[Proposition 2.11]{goesswein2019Dissertation}.	
\end{proof}

The other important class of function spaces on manifolds are parabolic H\"older spaces. We will first define them on a domain $\Omega\subset\R^{n}$ with smooth boundary $\partial\Omega$. For $\alpha\in(0,1), a,b\in\R$ consider the two semi-norms for a function $f: \bar{\Omega}\times[a,b]\to\R$ given by
\begin{align*}
\langle f\rangle_{x,\alpha}&:=\sup_{x_1,x_2\in\bar{\Omega},t\in[a,b]}\frac{|f(x_1,t)-f(x_2,t)|}{|x_1-x_2|^{\alpha}},\\
\langle f\rangle_{t,\alpha}&:=\sup_{x\in\bar{\Omega},t_1,t_2\in[a,b]}\frac{|f(x,t_1)-f(x,t_2)|}{|t_1-t_2|^{\alpha}}.
\end{align*}
Now, we define for $k, k'\in\N, \alpha\in (0,1), m\in\N$ the spaces
\begin{align*}
C^{\alpha,0}(\bar{\Omega}\times[a,b])&:=\{f\in C(\bar{\Omega}\times[a,b])|\langle f\rangle_{x,\alpha}<\infty\},\\
\|f\|_{C^{\alpha,0}(\bar{\Omega}\times[a,b])}&:=\|f\|_{\infty}+\langle f\rangle_{x,\alpha},\\
C^{0,\alpha}(\bar{\Omega}\times[a,b])&:=\{f\in C(\bar{\Omega}\times[a,b])| \langle f\rangle_{t,\alpha}<\infty\},\\
\|f\|_{C^{0,\alpha}(\bar{\Omega}\times[a,b])}&:=\|f\|_{\infty}+\langle f\rangle_{t,\alpha}, \\
C^{k+\alpha, 0}(\bar{\Omega}\times[a,b])&:=\{f\in C(\bar{\Omega}\times[a,b])|\forall t\in[a,b]: f\in C^k(\bar{\Omega}),\\
&\phantom{:=\{ }  \forall\beta\in \N^n_0, |\beta|\le k: \partial^x_{\beta}f\in C^{\alpha,0}(\bar{\Omega}\times[a,b])\},\\
\|f\|_{C^{k+\alpha, 0}(\bar{\Omega}\times[a,b])}&:=\sum_{|\beta|\le k}\|\partial_{\beta}^xf\|_{\infty}+\sum_{|\beta|=k}\langle \partial_{\beta}^xf \rangle_{x,\alpha}, \\
C^{k+\alpha,\frac{k+\alpha}{m}}(\bar{\Omega}\times[a,b])&:=\{f\in C(\bar{\Omega}\times[a,b])|\forall\beta\in\N_0^n,i\in\N_0,mi+|\beta|\le k:\ \\ &\phantom{:=\{ }\partial_t^i\partial_{\beta}^xf\in C^{\alpha,0}(\bar{\Omega}\times[a,b])\cap C^{0,\frac{k+\alpha-mi-|\beta|}{m}}(\bar{\Omega}\times[a,b])  \},\\
\|f\|_{C^{k+\alpha,\frac{k+\alpha}{m}}(\bar{\Omega}\times[a,b])}&:=\sum_{0\le mi+|\beta|\le k}\left(\|\partial_t^i\partial_{\beta}^xf\|_{\infty}+\|\partial_t^i\partial_{\beta}^xf\|_{C^{0,\frac{k+\alpha-mi-|\beta|}{m}}(\bar{\Omega}\times[a,b])}\right)\\
&+\sum_{mi+|\beta|=k}\|\partial_t^i\partial_{\beta}^xf\|_{C^{\alpha,0}(\bar{\Omega}\times[a,b])}.
\end{align*}
Hereby, we denote by $\partial^x_{\beta}$ a partial derivative in space with respect to the multi-index $\beta$ and $\partial^i_t$ the $i$-th partial derivative in time. The parameter $m$ corresponds to the order of the differential equation one is considering and in our work it will always be four. Parabolic H\"older spaces on submanifolds can now be defined in local coordinates.
\begin{definition}[Parabolic H\"older spaces on submanifolds]\label{DefinitionParabolichHspacesmfd}\ \\
	Let $\Gamma$ be a $C^{r}$-submanifold of $\R^n$. Then we define for $k\in\N_0, k<r, \alpha\in(0,1), a,b\in\R,m \in\N$ the space $C^{k+\alpha,\frac{k+\alpha}{m}}(\Gamma\times[a,b])$ as the set of all functions $f:\Gamma\to\R$ such that for any parametrization $\varphi:\Omega\to V\subset\Gamma$ we have that $f\circ\varphi\in C^{k+\alpha,\frac{k+\alpha}{m}}(\bar{\Omega}\times[a,b])$. 
\end{definition}
\begin{remark}[Traces of parabolic H\"older spaces]\label{RemarkTracesParabolicHspaces}\ \\
	On the boundary $\Sigma$ of $\Gamma$ we may choose $\varphi$ to be a parametrization that flattens the boundary. From this we see that
	\begin{align*}
	f\in C^{k+\alpha,k'+\alpha'}(\Gamma\times[a,b])\Rightarrow f\big|_{\Sigma\times[a,b]}\in C^{k+\alpha,k'+\alpha'}(\Sigma\times[a,b]).
	\end{align*}
\end{remark}
For the non-linear analysis we will the need the following product estimate and the contractivity property of lower order terms.
\begin{lemma}[Product estimates in parabolic H\"older spaces]\label{LemmaProductEstimatesParabolichHspace}\ \\
	Let $k,m\in\N,\alpha\in(0,1)$ and $f,g\in C^{k+\alpha,\frac{k+\alpha}{m}}(\overline{\Omega}\times[0,T])$. Then we have 
	\begin{align}\label{EquationProductareClosedinHolderSpaces}
	fg\in C^{k+\alpha,\frac{k+\alpha}{m}}(\overline{\Omega}\times[0,T]),
	\end{align}
	and furthermore we have that
	\begin{align}
	\|fg\|_{C^{k+\alpha,\frac{k+\alpha}{m}}(\overline{\Omega}\times[0,T])}&\le C \|f\|_{C^{k+\alpha,\frac{k+\alpha}{m}}(\overline{\Omega}\times[0,T])}\|g\|_{C^{k+\alpha,\frac{k+\alpha}{m}}(\overline{\Omega}\times[0,T])},\label{EquationProductEstimatesHolderSpaces1}\\
	\|fg\|_{C^{k+\alpha,\frac{k+\alpha}{m}}(\overline{\Omega}\times[0,T])}&\le C\left(\|f\|_{C^{k+\alpha,\frac{k+\alpha}{m}}(\overline{\Omega}\times[0,T])}\|g\|_{C^{k,0}}+\|f\|_{C^{k,0}}\|g\|_{C^{k+\alpha,\frac{k+\alpha}{m}}(\overline{\Omega}\times[0,T])}\right).\label{EquationProductEstimatesHolderSpaces2}
	\end{align}
\end{lemma}
\begin{proof}
See \cite[Lemma 2.16]{goesswein2019Dissertation}
\end{proof}
\begin{lemma}[Contractivity property of lower order terms in parabolic H\"older spaces]\label{LemmaContractivityLowerOrderTerms}\ \\
	Let $\Omega\subset\R^n$ be a bounded domain with smooth boundary and $$k,k'\in\{0,1,2,3,4\},k'<k,\alpha\in(0,1),a,b\in\R.$$ Then, we have for any $f\in C^{k+\alpha,\frac{k+\alpha}{4}}(\bar{\Omega}\times[a,b])$ that
	\begin{align}
	\|f\|_{C^{k'+\alpha,\frac{k'+\alpha}{4}}(\bar{\Omega}\times[a,b])}\le \|f\big|_{t=a}\|_{C^{k'+\alpha}(\bar{\Omega})}+C(b-a)^{\bar{\alpha}}\|f\|_{C^{k+\alpha,\frac{k+\alpha}{4}}(\bar{\Omega}\times[a,b])}.
	\end{align}
	Hereby, the constants $C$ and $\bar{\alpha}$ depend on $\alpha,k,k'$ and $\bar{\Omega}$. Especially, if $f\big|_{t=a}\equiv 0$, we have
	\begin{align}\label{EquationContractivityofLowerOrderTerms}
	\|f\|_{C^{k'+\alpha,\frac{k'+\alpha}{4}}(\bar{\Omega}\times[a,b])}\le C(b-a)^{\bar{\alpha}}\|f\|_{C^{k+\alpha,\frac{k+\alpha}{4}}(\bar{\Omega}\times[a,b])}.
	\end{align}
\end{lemma}
\begin{proof}
See \cite[Lemma 2.17]{goesswein2019Dissertation}
\end{proof}

As a final remark we want to mention that we sometimes identify  Sobolev and H\"older spaces (in local coordinates) with Besov spaces, especially to use interpolation and composition results. As we do not need them on manifolds we will not introduce them here but they can be found, e.g., in \cite[Section 2.3]{Triebel1994TheoryofFunctionSpace}.

\section{Parametrization, Compatibility Conditions, Main Result}\label{SectionParametrizationCom}
Similar to the classic approach for geometric flows of closed hypersurfaces with higher space dimensions, we want to write the evolution of the triple junction clusters as normal graphs over a fixed reference triple junction cluster $\Gamma_{\ast}:=\Gamma^1_{\ast}\cup\Gamma^2_{\ast}\cup\Gamma^3_{\ast}\cup\Sigma_{\ast}$. The classical approach is to write the evolution as a graph in normal direction over $\Gammaast$. Additionally, we need to allow a tangential part in a neighborhood of the triple junction to describe general motions of the geometric object. So, we want to describe the evolving hypersurface $\Gamma$ as image of the diffeomorphism
\begin{align}
\Phi^{i}_{\boldsymbol{\rho},\boldsymbol{\mu}}: \Gamma^i_{\ast}\times [0,T]&\to \R^{n+1},\notag\\
(\sigma,t)&\mapsto \sigma+\rho^i(\sigma,  t)\Nast^i(\sigma)+\mu^i(\sigma,t)\tauast^i(\sigma), \label{EquationNormalGraphTJ}
\end{align}
where $\tau^i_{\ast}$ are fixed, smooth tangential vector fields on $\Gamma^i_{\ast}$ that equal $\nuast^i$ on $\Sigmaast$ and have a support in a neighborhood of $\Sigmaast$ in $\Gamma^i_{\ast}$. The tuple $(\boldsymbol{\rho},\boldsymbol{\mu})$ consists of the unknown functions for which we want to derive a PDE system. Hereby, $\bmu$ has to be given as function in the normal part as otherwise the resulting PDE problem will be degenerated. We know from the work of \cite{depner2014mean} that the condition
\begin{align}\label{EquationConditonForConcurrencyofTJ}
\Phi^1_{\boldsymbol{\rho},\boldsymbol{\mu}}(\sigma,t)=\Phi^2_{\boldsymbol{\rho},\boldsymbol{\mu}}(\sigma,t)=\Phi^3_{\boldsymbol{\rho},\boldsymbol{\mu}}(\sigma,t) & &\text{for }\sigma\in\Sigma_{\ast}, t\ge 0,
\end{align}
which guarantees concurrency of the triple junction, is equivalent to
\begin{align}\label{EquationEquivalentCondtionforTripleJunctionConservation}
\begin{cases}
\gamma^1\rho^1+\gamma^2\rho^2+\gamma^3\rho^3=0 &\text{on }\Sigma_{\ast},\\
\boldsymbol{\mu}=\mathcal{T}\boldsymbol{\rho} &\text{on }\Sigma_{\ast}.
\end{cases}
\end{align}
Hereby, the matrix $\mathcal{T}$ is given by
\begin{align*}
\mathcal{T}=\begin{pmatrix}0 & \frac{c^2}{s^1} & -\frac{c^3}{s^1} \\
-\frac{c^1}{s^2} & 0 & \frac{c^3}{s^2} \\ \frac{c^1}{s^3} & -\frac{c^2}{s^3} & 0 \end{pmatrix},
\end{align*}
with $s^i=\sin(\theta^i)$ and $c^i=\cos(\theta^i)$. The second line in (\ref{EquationEquivalentCondtionforTripleJunctionConservation}) implies that the tangential part $\boldsymbol{\mu}$ is uniquely determined on $\Sigmaast$ by the values of $\boldsymbol{\rho}$. This motivates to get rid of the degenerated degrees of freedom of $\boldsymbol{\mu}$ by setting 
\begin{align}
\mu^i(\sigma):=\mu(\pr^i_{\Sigma}(\sigma)),
\end{align} where $\pr^i_{\Sigma}$ denotes the projection from a point on $\Gamma^i_{\ast}$ to the nearest point on $\Sigma_{\ast}$. Note that this map is only well-defined on a neighborhood of $\Sigmaast$ in $\Gammaast^i$. But we only need $\mu^i$ to be defined on the support of $\tau^i_{\ast}$, which can be assumed to be arbitrary small. Therefore, this tangential part can be used to find a well-defined PDE formulation.
\begin{remark}[Analytic aspects of the non-local term]\ \\
	One might expect that the included non-local term will have a huge impact on the existence analysis. Indeed, as we linearize our problem only in the reference geometry, the non-localities will be of lower order in the linearization and can thus be treated with perturbation theory. For the contraction estimates we can use the same arguments as for all other quasi-linear terms. In total, for the goal of this paper the non-localities will not cause any problems.
\end{remark}

Now, we want to find a suitable PDE-setting for $\boldsymbol{\rho}$. To begin with, we retract the equations from $\Gamma^i(t)$ on $\Gamma_{\ast}^i$. From here on, we will write $\Gamma_{\boldsymbol{\rho}}$ resp. $\Sigma_{\boldsymbol{\rho}}$ when referring to the triple junction cluster and the triple junction given as image of $\Phi_{\boldsymbol{\rho},\boldsymbol{\mu}}$. Also, we will use sub- and superscripts $i$ and $\boldsymbol{\rho}$ to denote pull-backs of quantities of the hypersurface $\Gamma_{\boldsymbol{\rho}}^i$ or of the triple junction $\Sigma_{\boldsymbol{\rho}}$.
This will also be applied on differential operators. So, for example, we will write for $(\sigma,t)\in\Gammaast^i\times[0,T], i=1,2,3,$
\begin{align*}
H^i_{\boldsymbol{\rho}}(\sigma,t)&:=H_{\Gamma^i_{\boldsymbol{\rho}}}(\Phi^i_{\boldsymbol{\rho},\boldsymbol{\mu}}(\sigma,t)),\\
\Delta_{\boldsymbol{\rho}}H_{\boldsymbol{\rho}}(\sigma,t)&:=\left(\Delta_{\Gamma^i_{\boldsymbol{\rho}}}H_{\Gamma^i_{\boldsymbol{\rho}}}\right)(\Phi^i_{\boldsymbol{\rho},\boldsymbol{\mu}}(\sigma,t)).
\end{align*}
Later, we will also sometimes use this notation when we consider the quantities on $\Gamma_{\boldsymbol{\rho}}$ but it will always be clear what is meant.\\
The evolution law \eqref{EquationMotionDuetoSurfaceDiffusion} together with the boundary conditions (CC), (AC), (CCP) and (FB) writes now as the following problem on $\Gammaast$.
\begin{align}(SDFTJ)\label{EquationSurfaceDiffusionTripleJunctionGeometricVersiononRrenceFrame}
\begin{cases}
V^i_{\boldsymbol{\rho}}=-\Delta_{\brho}H^i_{\brho} &\text{on }\Gamma^i_{\ast}, t\in[0,T], i=1,2,3,\\
\gamma^1\rho^1+\gamma^2\rho^2+\gamma^3\rho^3=0 &\text{on }\Sigma_{\ast}, t\in[0,T],\\
\langle N^1_{\brho},N_{\brho}^2\rangle=\cos(\theta^3) &\text{on }\Sigma_{\ast}, t\in [0,T],\\
\langle N^2_{\brho}, N^3_{\brho}\rangle=\cos(\theta^1) &\text{on }\Sigma_{\ast}, t\in[0,T], \\
\gamma^1H^1_{\brho}+\gamma^2H^2_{\brho}+\gamma^3H^3_{\brho}=0 &\text{on }\Sigma_{\ast}, t\in[0,T],\\
\nabla_{\brho}H^1_{\brho}\cdot\nu_{\brho}^1=\nabla_{\brho}H^2_{\brho}\cdot\nu^2_{\brho} &\text{on }\Sigma_{\ast}, t\in[0,T], \\
\nabla_{\brho}H^2_{\brho}\cdot\nu^2_{\brho}=\nabla_{\brho}H^3_{\brho}\cdot\nu_{\brho}^3 &\text{on }\Sigma_{\ast}, t\in[0,T],\\
(\rho^i(\sigma,0),\mu^i(\sigma,0))=(\rho_0^i, \mu_0^i) &\text{on }\Gamma_{\ast}^i\times \Sigma_{\ast}, i=1,2,3.
\end{cases}
\end{align}
Here, we assume that the initial surfaces are given as $\Gamma^i_0=\Gamma^i_{\rho_0^i,\mu_0^i}, i=1,2,3$ for $\boldsymbol{\rho}_0$ small enough in the $C^{4+\alpha}$-norm and $\boldsymbol{\mu}_0=\mathcal{T}\boldsymbol{\rho}_0$. This will then guarantee that the $\Gamma_0^i$ are indeed embedded hypersurfaces, cf. \cite[Remark 1]{depner2014mean} .\\
We want to mention that one can rewrite $\eqref{EquationSurfaceDiffusionTripleJunctionGeometricVersiononRrenceFrame}_1$ as a parabolic equation
\begin{align}
\partial_t\rho^i=\mathcal{K}^i(\rho^i, \boldsymbol{\rho}|_{\Sigma_{\ast}}),
\end{align}
where the operator $$\mathcal{K}^i:C^{4+\alpha,1+\alpha/4}(\Gamma^i_{\ast}\times[0,T])\times C^{4+\alpha,1+\alpha/4}(\Sigma_{\ast}\times[0,T])\to C^{\alpha,\alpha/4}(\Gammaast\times[0,T]) $$ contains both derivatives of $\rho^i$ and $\boldsymbol{\rho}|_{\Sigma_{\ast}}$ of up to order four. A detailed calculation can be found in \cite[Page 38f.]{goesswein2019Dissertation}.\\
The boundary condition can be rewritten as operators $$\mathcal{G}^i: C^{4+\alpha,1+\alpha/4}_{TJ}(\Gammaast\times[0,T])\to C^{4-\sigma_i+\alpha,\frac{4-\sigma_i+\alpha}{4}}(\Sigmaast\times[0,T])$$ given by
\begin{align*}
\mathcal{G}^1(\boldsymbol{\rho})&:=\gamma^1\rho^1+\gamma^2\rho^2+\gamma^3\rho^3=0   &\text{ on }\Sigma_{\ast}\times[0,T],\\
\mathcal{G}^2(\boldsymbol{\rho})&:=\langle N^1_{\brho}, N^2_{\brho}\rangle-\cos(\theta^3)=0  &\text{ on }\Sigma_{\ast}\times[0,T],\\
\mathcal{G}^3(\boldsymbol{\rho})&:=\langle N^2_{\brho},N^3_{\brho}\rangle-\cos(\theta^1)=0  &\text{ on }\Sigma_{\ast}\times[0,T],\\
\mathcal{G}^4(\boldsymbol{\rho})&:=\gamma^1H^1_{\brho}+\gamma^2H^2_{\brho}+\gamma^3H^3_{\brho}=0   &\text{ on }\Sigma_{\ast}\times[0,T],\\
\mathcal{G}^5(\boldsymbol{\rho})&:=\nabla_{\brho}H^1_{\brho}\cdot\nu^1_{\brho}-\nabla_{\brho}H^2_{\brho}\cdot \nu^2_{\brho}=0    &\text{ on }\Sigma_{\ast}\times[0,T],\\
\mathcal{G}^6(\boldsymbol{\rho})&:=\nabla_{\brho}H^2_{\brho}\cdot\nu^2_{\brho}-\nabla_{\brho}H^3_{\brho}\cdot \nu^4_{\brho}=0    &\text{ on }\Sigma_{\ast}\times[0,T].
\end{align*}
With $\mathcal{G}(\boldsymbol{\rho}):=\big(\mathcal{G}^i(\boldsymbol{\rho})\big)_{i=1,...,6}$, \eqref{EquationSurfaceDiffusionTripleJunctionGeometricVersiononRrenceFrame} rewrites to the following  problem for $(\rho^1, \rho^2, \rho^3)$:
\begin{align}\label{EquationAnalyticFormulationofSDFTJ}
\begin{cases}
\partial_t\rho^i=\mathcal{K}^i(\rho^i, \boldsymbol{\rho}|_{\Sigma_{\ast}}) &\text{ on }\Gamma^i_{\ast}\times[0,T], i=1,2,3,\\
\mathcal{G}(\boldsymbol{\rho})=0 &\text{ on }\Sigma_{\ast}\times[0,T],\\
\rho^i(\cdot,0)=\rho^i_0 &\text{ on }\Sigma_{\ast}.
\end{cases}
\end{align}

As we seek for solutions smooth up to $t=0$, we will require some compatibility conditions for our initial data. Clearly, the boundary conditions given by $\mathcal{G}$ have to be fulfilled by the initial data. Additionally, for a smooth solution the conditions $\eqref{EquationSurfaceDiffusionTripleJunctionGeometricVersiononRrenceFrame}_1$ resp $\mathcal{G}^1(\brho)=0$ are differentiable in time at $t=0$ and so will also require a compatibility condition. Together, we get by considering \eqref{EquationSurfaceDiffusionTripleJunctionGeometricVersiononRrenceFrame} resp. \eqref{EquationAnalyticFormulationofSDFTJ} the corresponding compatibility conditions

 \begin{align}
 \label{EquationGeometricCompabilityConditionforSDFTJ}
 (GCC)&\begin{cases}
 \Gamma_0\text{ fulfils }(CC), (AC), (CCP), (FB) &\text{on } \Sigma_{0},\\
 \sum_{i=1}^3\gamma^i\Delta_{\Gamma_0^i} H_{\Gamma_0^i}=0 &\text{on }\Sigma_0.
 \end{cases}\\
 \label{EquationCompabilityConditionsAnalyticVersionofSFDTJ}
(ACC)&\begin{cases}\mathcal{G}(\boldsymbol{\rho}_0)=0 & \text{on }\Sigma_{\ast}, \\ \mathcal{G}_0(\boldsymbol{\rho}_0):=\sum_{i=1}^3\gamma^i\mathcal{K}^i(\rho_0^i, \boldsymbol{\rho}_0\big|_{\Sigma_{\ast}})=0 & \text{on }\Sigma_{\ast}.
\end{cases} 
\end{align}
Indeed, one can prove that $(ACC)$ and $(GCC)$ are equivalent provided that $\brho_0$ is small enough in the $C^{4+\alpha}(\Gammaast)$-norm, cf. \cite[Lemma 4.1]{goesswein2019Dissertation}. With these we can now state the main results of this article. 
\begin{theorem}[Short time existence for surface diffusion flow with triple junctions] \label{TheoremSTETripleJunctions}\ \\
	Let $\Gammaast$ be a $C^{5+\alpha}$-reference cluster. Then there exists an $\varepsilon_0>0$ and a $T>0$ such that for all initial data $\brho_0\in C^{4+\alpha}_{TJ}(\Gammaast)$ with $\|\brho_0\|_{C^{4+\alpha}_{TJ}(\Gammaast)}<\varepsilon_0$, which fulfill the analytic compatibility conditions (\ref{EquationCompabilityConditionsAnalyticVersionofSFDTJ}), there exists a unique solution $\boldsymbol{\rho}\in C^{4+\alpha,1+\frac{\alpha}{4}}_{TJ}(\GammaastT)$ of (\ref{EquationAnalyticFormulationofSDFTJ}). Additionally, the norm of $\brho$ is bounded by a constant $R$ uniformly in $\brho_0$.
\end{theorem}
\begin{remark}[Results for the geometric problem]\ \\ If the initial data are of class $C^{5+\alpha}$ we can use them itself as a reference cluster. In particular, Theorem \ref{TheoremSTETripleJunctions} implies existence of the original geometric problem given by \ref{EquationMotionDuetoSurfaceDiffusion} together with $(CC), (AC), (CCP), (FB)$ for such initial data. For initial data of class $C^{4+\alpha}$ one needs to verify that one can construct close enough reference frames. For the closed case, this was proven in \cite{prusssimmonett2012manifoldofclosedhypersurfaces}. The uniqueness of solutions of the geometric problem remains an open question.
\end{remark}
\section{Linearization} 
\label{SectionLinearization}
Now we want to linearize problem \eqref{EquationSurfaceDiffusionTripleJunctionGeometricVersiononRrenceFrame} in an arbitrary $C^{5+\alpha}$-reference frame $\Gammaast$. This is done pointwise in every $p\in\Gamma_{\ast}\cup\Sigma_{\ast}$. We want to explain what we mean precisely with our linearization procedure. For any fixed point $\sigma\in\Gamma^i_{\ast},i=1,2,3$ or $\sigma\in\Sigma_{\ast}$, any term in $\eqref{EquationSurfaceDiffusionTripleJunctionGeometricVersiononRrenceFrame}$, which are all given as functions in $\sigma, \boldsymbol{\rho}$ and $\boldsymbol{\mu}$, and any tuple 
\begin{align*}
(\boldsymbol{u}, \boldsymbol{\phi})\in \left(C^{4+\alpha}(\Gamma^1_{\ast})\times C^{4+\alpha}(\Gamma^2_{\ast})\times C^{4+\alpha}(\Gamma^3_{\ast})\right)\times \left(C^{4+\alpha}(\Sigma_{\ast})\right)^3,
\end{align*}
fulfilling (\ref{EquationEquivalentCondtionforTripleJunctionConservation}), we replace $\rho^i$ with $\varepsilon u^i$ and $\mu^i$ with $\varepsilon\phi^i$, differentiate the new expression in $\varepsilon$ and evaluate this for $\varepsilon=0$. Hereby, we observe that (\ref{EquationEquivalentCondtionforTripleJunctionConservation}) is also fulfilled for $(\varepsilon \boldsymbol{u}, \varepsilon\boldsymbol{\phi})$ due to its linear structure. Additionally, all terms are well defined for $(\boldsymbol{u},\boldsymbol{\phi})$ small enough in the $C^{4+\alpha}$-norm for every $t$, so at least for $\varepsilon$ small enough.
\begin{remark}[Independence of local coordinates]\label{RemarkOnLinearisationofSDFTJ}\ \\
	The goal of this process is to derive a linear equation on $\Gamma_{\ast}$. Our procedure, though, will lead to equations in local coordinates. But as we linearize geometric quantities that are itself independent of local coordinates we conclude that the equations we derive have to represent a global equation on $\Gammaast$. We will get this form for the terms of highest order and all other terms will be dealt with by using perturbation arguments.
\end{remark}
In the following,  we will index a geometric quantity with $\varepsilon$ to denote the quantity on $\Gamma_{\varepsilon u,\varepsilon \phi}$ at the point $\Phi_{\varepsilon u,\varepsilon \phi}(\sigma,t)$. We will omit the fixed time and space variable $(\sigma,t)$ and also the projection in the $\boldsymbol{\phi}$-terms. For the analysis we will do later it is not important to know the lower order terms precisely. Thus, we will denote them only in qualitative form using dynamical coefficient functions $a^{k+s}$, which denotes some function on the corresponding hypersurface $\Gamma^i_{\ast}$ that has $C^{k+s}$-regularity on this surface. Also, like dynamical constants the $a^{k+s}$ may adapt from line to line. \\
Luckily for us the most critical part was already done in \cite[Section 3]{depner2013linearized} and \cite[Section 3]{depner2014mean}, where the authors proved that
\begin{align}
\frac{d}{d\varepsilon}V^i_{\varepsilon}\big|_{\varepsilon=0}&=\partial_t u^i,\label{EquationLinearizV}\\
\frac{d}{d\varepsilon}H^i_{\varepsilon}\big|_{\varepsilon=0}&=\Dast u^i+|\Pi_{\ast}^i|^2u^i+\langle\gast H^i,\tauast^i\rangle\phi^i, \label{EquationLinearizH}\\
\frac{d}{d\varepsilon}\langle N^i_{\varepsilon}, N^j_{\varepsilon}\rangle\big|_{\varepsilon=0}&=\partialnuasta{i}u^i+\Pi^i_{\ast}(\nuast^i,\nuast^i)\phi^i-\partialnuasta{j}u^j-\Pi_{\ast}^j(\nuast^j,\nuast^j)\phi^j.  \label{EquationLinearizAC}
\end{align} 
Indeed, with \eqref{EquationLinearizV}-\eqref{EquationLinearizAC} we can already determine all terms of highest order in the linearization. As all others terms will be dealt with using perturbation arguments, we will only need qualitative results for their structure, especially the regularity of their coefficients. For this we need the following result on the linearization of the basic geometric quantities. 
\begin{lemma}[Linearization of basic geometric quantities]\label{LemmaLinearisierungSDFTJErstesHilfslemma}\label{LemmaSurfaceDiffusionTriplePointsLinearisationGeometricQuantities}\vspace{-0,2cm}
	\begin{align}
	\frac{d}{d\varepsilon}\partial_k^{i,\varepsilon}\big|_{\varepsilon=0}&=a^{3+\alpha}u^i+a^{3+\alpha}\phi^i+a^{4+\alpha}\partial_k u^i+a^{4+\alpha}\partial_k \phi^i, \\
	\frac{d}{d\varepsilon}g_{kl}^{i,\varepsilon}\big|_{\varepsilon=0}&=a^{3+\alpha}u^i+a^{3+\alpha}\phi^i+a^{4+\alpha}\partial_k \phi^i+a^{4+\alpha}\partial_l \phi^i,\\
	\frac{d}{d\varepsilon}g^{kl}_{i,\varepsilon}\big|_{\varepsilon=0}&=a^{3+\alpha}u^i+a^{3+\alpha}\phi^i+\sum_{j=1}^na^{4+\alpha}\partial_j\phi^i,\\
	\frac{d}{d\varepsilon}(\Gamma_{kl}^{m})^{i,\varepsilon}\big|_{\varepsilon=0}&=a^{2+\alpha}u^i+a^{2+\alpha}\phi^i+\sum_ja^{3+\alpha}\partial_j u^i+\sum_{j}a^{3+\alpha}\partial_j \phi^i+\sum_{j,j'}a^{4+\alpha}\partial_{jj'}\phi^i. \\
		\frac{d}{d\varepsilon}N^i_{\varepsilon}\big|_{\varepsilon=0}   &=\sum_{l=1}^n\big(a^{4+\alpha}\partial_l u^i+a^{4+\alpha}\partial_l\phi^i\big)+a^{3+\alpha}u^i+a^{3+\alpha}\phi^i,\\
	\frac{d}{d\varepsilon}\nu^i_{\varepsilon}\big|_{\varepsilon=0} &=\sum_{l=1}^n\left(a^{4+\alpha}\partial_lu^i+a^{4+\alpha}\partial_{l}\phi^i\right)+a^{3+\alpha}u^i+a^{3+\alpha}\phi^i.
	\end{align}
\end{lemma} 
\begin{proof}
Cf. \cite[Lemma 4.5 und 4.6]{goesswein2019Dissertation}.	
\end{proof}
Now we can use that for our linearization procedure we have a product rule both for the geometric quantities and the differential operators, which can be seen by writing the latter ones in local coordinates. Thus, we conclude for the linearization of the right-hand side of $\eqref{EquationSurfaceDiffusionTripleJunctionGeometricVersiononRrenceFrame}_1$ that
\begin{align*}
\frac{d}{d\varepsilon}\left(\Delta_{\Gamma^i_{\varepsilon}}H^i_{\varepsilon}\right)\big|_{\varepsilon=0}&=\ddeps\left(\Delta_{\Gamma^i_{\varepsilon}}\right)\big|_{\epsilon=0}H^i_{\ast}+\Dast\ddeps\left(H^i_{\varepsilon}\right)\Stricheps=-\Dast\Dast u^i+loT,
\end{align*}  
where we used in the last equality \eqref{EquationLinearizH} and the fact that $H^i_{\ast}$ does not contain any derivative of $\bu$ and consequently the first summand is of lower order. Combined with \eqref{EquationLinearizV} we get for the linearization of $\eqref{EquationSurfaceDiffusionTripleJunctionGeometricVersiononRrenceFrame}_1$
\begin{align*}
\partial_t u^i=-\Dast\Dast u^i+\widetilde{\mathcal{A}}^i_{LOT}u^i+\bar{\mathcal{A}}^i_{LOT}\phi^i. 
\end{align*}
Hereby, the $\widetilde{\mathcal{A}}^i_{LOT}u^i$ denote the lower order terms in $u^i$ and $\bar{\mathcal{A}}^i_{LOT}\phi^i$ the lower order terms in $\phi^i$, so the non-local part in the lower order terms. Observe here that the terms of highest order are purely local expressions. Also, in the following linearizations of the boundary conditions no non-local terms will arise as the tangential part is only non-local away from the triple junction.\\
Condition $\eqref{EquationSurfaceDiffusionTripleJunctionGeometricVersiononRrenceFrame}_2$ is already linear. The result for the linearization of the angle conditions can be found in \cite[Section 3]{depner2014mean} and follows directly from \eqref{EquationLinearizAC}. For $\eqref{EquationSurfaceDiffusionTripleJunctionGeometricVersiononRrenceFrame}_5$ we apply \eqref{EquationLinearizH} on all three mean curvature operators to get
\begin{align}
\gamma^1\Dast u^1+\gamma^2\Dast u^2+\gamma^3\Dast u^3+\mathcal{B}_{LOT}^{4}\boldsymbol{u}=0,
\end{align}
where $\mathcal{B}_{LOT}^4\boldsymbol{u}$ denotes the lower order terms. Finally, for the linearization of $\eqref{EquationSurfaceDiffusionTripleJunctionGeometricVersiononRrenceFrame}_5$ and $\eqref{EquationSurfaceDiffusionTripleJunctionGeometricVersiononRrenceFrame}_6$ we first note that we have
\begin{align}
\ddeps\left(\nabla_{\Gamma^i_{\varepsilon}}H^i_{\varepsilon}\right)\Stricheps=\ddeps\left(\nabla_{\Gamma^i_{\varepsilon}}\right)\Stricheps H^i_{\ast}+\gast \ddeps\left(H^i_{\varepsilon}\right)\Stricheps=\gast\left(\Dast u^i\right)+loT, \label{EquationLinearizationAux}
\end{align}
where we again used that $H^i_{\ast}$ does not contain any derivatives of $\bu$ and so the first summand only contributes lower order terms. With \eqref{EquationLinearizationAux} and the same argumentation we get that
\begin{align}
\ddeps\left(\nabla_{\Gamma^i_{\varepsilon}}H^i_{\varepsilon}\cdot \nu^i_{\varepsilon}\right)\Stricheps=\ddeps\left(\nabla_{\Gamma^i_{\varepsilon}}H^i_{\varepsilon} \right)\Stricheps\cdot\nuast^i+\gast H^i_{\ast}\cdot\ddeps\left(\nu^i_{\varepsilon}\right)\Stricheps=\gast\left(\Dast u^i\right)+loT.
\end{align}
Summing up we get for the linearizations of $\eqref{EquationSurfaceDiffusionTripleJunctionGeometricVersiononRrenceFrame}_5$ and $\eqref{EquationSurfaceDiffusionTripleJunctionGeometricVersiononRrenceFrame}_6$ 
\begin{align*}
\gast(\Dast u^1)\cdot\nu^1_{\ast}-\gast
(\Dast u^2)\cdot\nu^2_{\ast}+\mathcal{B}^{5}_{LOT}\boldsymbol{u}&=0,\\
\gast(\Dast u^2)\cdot\nu^2_{\ast}-\gast
(\Dast u^3)\cdot\nu^3_{\ast}+\mathcal{B}^{6}_{LOT}\boldsymbol{u}&=0.
\end{align*}
where $\mathcal{B}^5_{LOT}\bu$ and $\mathcal{B}^6_{LOT}\bu$ denote the lower order terms and we directly rewrote the tangential terms using \eqref{EquationEquivalentCondtionforTripleJunctionConservation}. Summing up all our results we get the following linearized problem.
\begin{align}
(LSDFTJ)\begin{cases}
\partial_t u^i=-\Dast\Dast u^i+\widetilde{\mathcal{A}}^i_{LOT}u^i+\bar{\mathcal{A}}^i_{LOT}\trsigma{\bu} &\text{on }\Gamma^i_{\ast}\times[0,T],\ i=1,2,3, \\
\gamma^iu^i+\gamma^2u^2+\gamma^3u^3=0 &\text{on }\Sigma_{\ast}\times[0,T], \\
\partial_{\nu^1_{\ast}}u^1-\partial_{\nu^2_{\ast}}u^2+\mathcal{B}^2_{LOT}\boldsymbol{u}=0&\text{on }\Sigma_{\ast}\times[0,T],\\
\partial_{\nu^2_{\ast}}u^2-\partial_{\nu^3_{\ast}}u^3+\mathcal{B}^3_{LOT}\boldsymbol{u}=0 &\text{on }\Sigma_{\ast}\times[0,T],\\
\gamma^1\Delta_{\ast}u^1+\gamma^2\Delta_{\ast}u^2+\gamma^3\Delta_{\ast}u^3+\mathcal{B}^4_{LOT}\bu=0 &\text{on }\Sigma^{\ast}\times[0,T],\\
\gast(\Dast u^1)\cdot\nu^1_{\ast}-\gast
(\Dast u^2)\cdot\nu^2_{\ast}+\mathcal{B}^{5}_{LOT}\boldsymbol{u}=0 &\text{on }\Sigma^{\ast}\times[0,T],\\
\gast(\Dast u^2)\cdot\nu^2_{\ast}-\gast
(\Dast u^3)\cdot\nu^3_{\ast}+\mathcal{B}^{6}_{LOT}\boldsymbol{u}=0 &\text{on }\Sigma_{\ast}\times[0,T],\\
u^i\big|_{t=0}=\rho^i_{0}  &\text{on }\Gamma^i_{\ast},\ i=1,2,3.
\end{cases}
\end{align}
Here, $\mathcal{B}^2_{LOT}\bu$ and $\mathcal{B}^3_{LOT}\bu$ refer to the arising lower order terms. In an analytic version this writes now as
\begin{align}\label{EquationLinearisedSurfaceDiffusionTripleJunctionAbstractFormulaton}
\begin{cases}
\partial_t u^i = \mathcal{A}^i(u^i, \boldsymbol{u}\big|_{\Sigma_{\ast}})\big)+\mathfrak{f}^i &\text{on }\Gamma^i_{\ast}\times[0,T], i=1,2,3,\\
\mathcal{B}\boldsymbol{u}=\mathfrak{b} &\text{on }\Sigma_{\ast}\times[0,T],\\
u^i\big|_{t=0}=u_0^i &\text{on }\Gamma_{\ast}^i, i=1,2,3,
\end{cases}
\end{align}
where we used the notation
\begin{align*}
\mathcal{A}^i(u^i,\boldsymbol{u}\big|_{\Sigma_{\ast}})&=\mathcal{A}^i_{HOT}u^i+\mathcal{A}^i_{LOT}(u^i,\trsigma{\bu}),\\
\mathcal{A}_{HOT}^iu^i&=-\Dast\Dast u^i,\\
\mathcal{A}_{LOT}^iu^i&=\widetilde{\mathcal{A}}^i_{LOT}u^i+\bar{\mathcal{A}}^i_{LOT}\trsigma{\bu},\\
\mathcal{B}\bu&=\mathcal{B}_{HOT}\bu+\mathcal{B}_{LOT}\bu,\\
\mathcal{B}^1_{HOT}\boldsymbol{u}&=\gamma^1u^1+\gamma^2u^2+\gamma^3u^3,\\
\mathcal{B}^2_{HOT}\boldsymbol{u}&=\partialnuasta{1}u^1-\partialnuasta{2}u^2,\\
\mathcal{B}^3_{HOT}\boldsymbol{u}&=\partialnuasta{2} u^2-\partialnuasta{3}u^3,\\
\mathcal{B}^4_{HOT}\boldsymbol{u}&=\gamma^1\Dast u^1+\gamma^2\Dast u^2+\gamma^3\Dast u^3,\\
\mathcal{B}^5_{HOT}\boldsymbol{u}&=\partial_{\nuast^1}\big(\Dast u^1\big)-\partial_{\nuast^2}\big(\Dast u^2\big),\\
\mathcal{B}^6_{HOT}\boldsymbol{u}&=\partial_{\nuast^2}\big(\Dast u^2\big)-\partial_{\nuast^3}\big(\Dast u^3\big).
\end{align*}
Additionally, the $\mathfrak{f}^i$ and $\mathfrak{b}^i$ denote possible inhomogeneities, which we will will need for the fixed-point argument in the non-linear analysis.

\section{Linear Analysis}\label{SectionLinearAnalysis}
In this section we want to derive an existence result for the linearized problem \eqref{EquationLinearisedSurfaceDiffusionTripleJunctionAbstractFormulaton}. We will first study the system with the terms of highest order, zero initial data and without inhomogeneities in the lower order boundary conditions. That is, we are considering
\begin{align}\label{EquationReducedLinearProblem}
\begin{cases}
\partial_t u^i =-\Dast\Dast u^i+\mathfrak{f}^i &\text{on }\Gamma^i_{\ast}\times[0,T], i=1,2,3,\\
\mathcal{B}_{HOT}\boldsymbol{u}=\mathfrak{b} &\text{on }\Sigma_{\ast}\times[0,T],\\
u^i\big|_{t=0}=0 &\text{on }\Gamma_{\ast}^i, i=1,2,3.
\end{cases},
\end{align} 
with $\mathfrak{b}^1=\mathfrak{b}^2=\mathfrak{b}^3=\mathfrak{b}^4\equiv 0$. For this problem, we will show existence of weak solutions in Section \ref{SubsectionWeakTheory} and then H\"older-regularity for the found weak solutions in Section \ref{SubsectionSchauderEstimates}. In Section \ref{SubsectionCompleteLinearProblem} we will return to the original problem by including the missing terms via perturbations arguments. Altogether, we will prove the following result for \eqref{EquationLinearisedSurfaceDiffusionTripleJunctionAbstractFormulaton}.
\begin{theorem}[Short time existence for (LSDFTJ)]\label{theoremShorttimeexistenceLinearisedSDFTJGeneralInitialdata}\ \\
	Let $\sigma_i, i=1,...,6 $ be defined as in \eqref{EquationDefinitionofSigma}. For any $\alpha\in(0,1)$ and initial data $\bu_0\in C^{4+\alpha}_{TJ}(\Gamma_{\ast})$ there exists a $\delta_0>0$ such that for every $$\mathfrak{f}\in C^{\alpha,\frac{\alpha}{4}}_{TJ}(\Gamma_{\ast,\delta_0}),\quad \mathfrak{b}^i\in C^{4+\alpha-\sigma_i,\frac{4+\alpha-\sigma_i}{4}}(\Sigma_{\ast,\delta_0}),\ i=1,...,6,$$ fulfilling the inhomogeneous compatibility conditions
	\begin{align}\label{EquationCompatibilityConditionsLinearisedTJ}
	(CLP)\begin{cases}
	(\gamma^1\mathfrak{f}^1+\gamma^2\mathfrak{f}^2+\gamma^3\mathfrak{f}^3)\big|_{t=0}=\mathcal{B}_0(\boldsymbol{u}_0):=-\sum_{i=1}^3\gamma^i \mathcal{A}^i_{\text{all}}u_0^i &\text{on }\Sigmaast,\\
	\mathfrak{b}^i\big|_{t=0}=-\mathcal{B}^iu_0^i &\text{on }\Sigmaast, i=1,...,6,
	\end{cases}
	\end{align}
	 the problem (\ref{EquationLinearisedSurfaceDiffusionTripleJunctionAbstractFormulaton}) with initial data $\boldsymbol{u}_0$ has a unique solution $\boldsymbol{u}\in C^{4+\alpha,1+\frac{\alpha}{4}}(\Gamma_{\ast,\delta_0})$. Moreover, there exists a $C>0$ with 
	\begin{align}\label{EquationEnergieabschatzungenLineareGleichungVollTJ}
	\|\boldsymbol{u}\|_{C^{4+\alpha, \frac{1+\alpha}{4}}_{TJ}(\Gamma_{\ast,\delta_0})}\le C\left( \|\boldsymbol{\mathfrak{f}}\|_{C^{\alpha,\frac{\alpha}{4}}_{TJ}(\Gamma_{\ast,\delta_0})}+\|\boldsymbol{u}_0\|_{C^{4+\alpha}_{TJ}(\Gamma_{\ast})}+\sum_{i=1}^6\|\mathfrak{b}^i\|_{C^{4+\alpha-\sigma_i,\frac{4+\alpha-\sigma_i}{4}}(\Sigma_{\ast,\delta_0})}\right).
	\end{align} 
\end{theorem}
\subsection{Existence of Weak Solutions}
\label{SubsectionWeakTheory}
In this section we will discuss the theory of weak solution for \eqref{EquationReducedLinearProblem}. Unfortunately, we cannot directly use \eqref{EquationReducedLinearProblem} for a weak formulation as the sum condition for the $u^i$ on the triple junction reduces one degree of freedom in the testing procedure. Thus, we will not be able to fit in all boundary conditions into a weak formulation. To overcome this obstacle we need to split the fourth order parabolic problem in a coupled system of a second order parabolic and a second order elliptic equation. For this we introduce the auxiliary function
\begin{align}\label{EquationAuxilaryVariable}
v^i:=-\Dast u^i+C_vu^i,\quad i=1,2,3,
\end{align}
which equals the linearization of the mean curvature operator up to lower order terms. Using \eqref{EquationAuxilaryVariable} we can rewrite \eqref{EquationReducedLinearProblem} to
\begin{align}\label{EquationLinearisedSDFTJWeakPart}
(LSDFTJ)_P\begin{cases}
\partial_t u^i= \Dast v^i-C_uv^i+\mathfrak{f}^i &\text{on }\Gamma_{\ast}^i\times[0,T], i=1,2,3,\\
v_i=-\Delta_{\ast}u_i+C_vu^i
&\text{on }\Gamma_{\ast}^i\times[0,T], i=1,2,3,\\
\gamma^1u^1+\gamma^2u^2+\gamma^3u^3=0 &\text{on }\Sigma_{\ast}\times[0,T],\\
\partial_{\nu_{\ast}^1}u^1-\partial_{\nu^{\ast}}u^2=0 &\text{on }\Sigma_{\ast}\times[0,T],\\
\partial_{\nu_{\ast}^2}u^2-\partial_{\nu_{\ast}^3}u^3=0 &\text{on }\Sigma_{\ast}\times[0,T],\\
\gamma^1v^1+\gamma^2v^2+\gamma^3v^3=0 &\text{on }\Sigma_{\ast}\times[0,T], \\
\partial_{\nu^1_{\ast}}v^1-\partial_{\nu^2_{\ast}}v^2=\mathfrak{b}^5 &\text{on }\Sigma_{\ast}\times[0,T],\\
\partial_{\nu_{\ast}^2}v^2-\partial_{\nu_{\ast}^3}v^3=\mathfrak{b}^6 &\text{on }\Sigma_{\ast}\times[0,T],\\
u\big|_{t=0}=0 &\text{on }\Gamma^i_{\ast}, i=1,2,3.
\end{cases} 
\end{align}
Hereby, the terms $-C_uv^i$ and $C_vu^i$ are included to guarantee coercivity for the operators induced by the weak formulation. The constants $C_u$ and $C_v$ are chosen large enough such that these hold. During the proof of existence of weak solutions we will see that they only depend on the system. Once we haven proven H\"older regularity for this system they will disappear in the perturbation argument. \\
On this we apply the usual testing procedure to get the following weak formulation. We consider the function spaces
\begin{align*}
\mathcal{L}&:=L^2(\Gamma^1_{\ast})\times L^2(\Gamma^2_{\ast})\times L^2(\Gamma^3_{\ast}),\\
\mathcal{L}_b&:=L^2(\Sigma_{\ast})^3,\\
\mathcal{H}^1&:=H^{1}(\Gamma^1_{\ast})\times H^1(\Gamma^2_{\ast})\times H^1(\Gamma^3_{\ast}),\\
\mathcal{H}^{-1}&:=H^{-1}(\Gamma_{\ast}^1)\times H^{-1}(\Gamma_{\ast}^2)\times H^{-1}(\Gamma_{\ast}^3),\\
\mathcal{E}&:=\{\bu\in\mathcal{H}^1|\gamma^1u^1+\gamma^2u^2+\gamma^3u^3=0\text{ a.e. on }\Sigma^{\ast}\},
\end{align*}
and the continuous operators given by 
\begin{align*}
\langle\partial_t \boldsymbol{u},\boldsymbol{\zeta}\rangle_{dual}&:=\sum_{i=1}^3\gamma^i\langle\partial_t u^i, \zeta^i\rangle_{dual} & &\partial_t\boldsymbol{u}\in\mathcal{E}^{-1}, \boldsymbol{\zeta}\in\mathcal{E},\\
B_u[\boldsymbol{v},\boldsymbol{\zeta}]&:=-\sum_{i=1}^3\gamma^i\int_{\Gamma_{\ast}^i}\gast v^i\cdot\gast\zeta^i+C_uv^i\zeta^i\dH^n & &\boldsymbol{v},\boldsymbol{\zeta}\in\mathcal{E},\\
B_v[\boldsymbol{u},\boldsymbol{\psi}]&:=\sum_{i=1}^3\gamma^i\int_{\Gamma_{\ast}^i}\gast u^i\cdot \gast\psi^i+C_vu^i\psi^i\dH^n & &\boldsymbol{u},\boldsymbol{\psi}\in\mathcal{E},\\
b_u(\boldsymbol{\zeta};t)&:=\int_{\Sigma_{\ast}}\gamma^1\mathfrak{b}^5\zeta^1-\gamma^3\mathfrak{b}^6\zeta^3\dH^n& &\boldsymbol{\zeta}\in\mathcal{E},\\
f_u(\boldsymbol{\zeta}; t)&=\sum_{i=1}^3\int_{\Gamma_{\ast}^i}\gamma^i\mathfrak{f}^i\zeta^i\dH^n & &\boldsymbol{\zeta}\in\mathcal{E}.
\end{align*}
For clarification we note that $\langle\cdot,\cdot\rangle_{\text{dual}}$ in the first line on the right hand side is the usual duality pairing between $\mathcal{E}^{-1}$ and $\mathcal{E}$. With this we can introduce our weak solution concept.
\begin{definition}[Weak Solution of $(LSDFTJ)_P$]\ \\
	We call a tuple $(\boldsymbol{u},\boldsymbol{v})\in L^2(0,T;\mathcal{E})\times L^2(0,T;\mathcal{E})$ with $\partial_t \bu\in L^2(0,T; \mathcal{E}^{-1})$ a weak solution of $(\ref{EquationLinearisedSDFTJWeakPart})$ if for all $\boldsymbol{\zeta},\boldsymbol{\psi}\in \mathcal{E}$ and almost all $t\in[0,T]$ it holds
	\begin{align}\label{EquationWeakformulationofLinearSDFTJ}
	\begin{cases}\langle\partial_t \boldsymbol{u},\boldsymbol{\zeta}\rangle_{\text{dual}}=B_u[\boldsymbol{v},\boldsymbol{\zeta}]+b_u(\boldsymbol{\zeta};t)+f_u(\boldsymbol{\zeta};t),\\
	(\boldsymbol{\gamma}\boldsymbol{v},\boldsymbol{\psi})_{\mathcal{L}}=B_v[\boldsymbol{u},\boldsymbol{\psi}],
	\end{cases}
	\end{align}
	and additionally it holds 
	\begin{align}
	\bu\big|_{t=0}=0.
	\end{align}
\end{definition}
Reversing the test procedure one can show that weak solutions with $C^{4,1}$-regularityare indeed classical solutions of $(LSDFTJ)_P$, cf. \cite[Lemma 4.11]{goesswein2019Dissertation}. Now we show existence of a unique weak solution of $(LSDFTJ)_P$.
\begin{proposition}[Existence of weak solutions of $(LSDFTJ)_P$]\label{PropostionsExistenceofWeakSolutionsSurfaceDiffsuionTripleJunction}\ \\ 
	For all $\mathfrak{f}\in L^2(0,T; \mathcal{L})$ and $\mathfrak{b}^5, \mathfrak{b}^6\in L^2(0,T; L^2(\Sigma_{\ast}))$ there exists a unique weak solution $(\boldsymbol{u}, \boldsymbol{v})$ of (\ref{EquationLinearisedSDFTJWeakPart}) and we have the energy estimates
	\begin{align}
	\max_{0\le t\le T}\|\boldsymbol{u}(t) \|_{\mathcal{E}}+\|\boldsymbol{u}\|_{L^2(0,T;\mathcal{E})}&+\|\boldsymbol{u'}\|_{L^2(0,T;\mathcal{E}^{-1})}+\|\boldsymbol{v}\|_{L^2(0,T; \mathcal{E})}\label{EquationEnergyEstiamteWeakSolutionTJ}\\
	&\le C\left(\|\mathfrak{f}\|_{L^2(0,T; \mathcal{L})}+\sum_{i=5}^6\|\mathfrak{b}^i\|_{L^2(0,T; L^2(\Sigma_{\ast}))}\right)\notag.
	\end{align}
\end{proposition}
\begin{proof}
	We want to apply the Galerkin scheme. As in \cite[Section 4]{depner2014mean} we can choose
	the solutions $(\boldsymbol{z}_j)_{j\in\N}$ of the eigenvalue problem
	\begin{align}\label{EquationProofofExistenceofWeakSolution}
	\begin{cases}
	-\gamma^i \Dast z^i=\lambda\gamma^iz^i & \text{on }\Gamma^i_{\ast}, i=1,2,3,\\
	\gamma^1z^1+\gamma^2z^2+\gamma^3z^3=0 & \text{on }\Sigma_{\ast},\\
	\partial_{\nu_{\ast}^1}z^1=\partial_{\nu_{\ast}^2}z^2=
	\partial_{\nu_{\ast}^3}z^3 &\text{on }\Sigma_{\ast},
	\end{cases}
	\end{align}  
	as orthonormal basis of $\mathcal{L}$ equipped with the inner product
	\begin{align}\label{EquationInnerProduct}
	(\boldsymbol{f},\boldsymbol{g})_{\gamma\mathcal{L}}:=\sum_{i=1}^3\gamma^i\int_{\Gamma_{\ast}^i}f^ig^i\dH^n.
	\end{align}
	In \cite{depner2014mean} , the authors proved also smoothness of these functions. Also note that due to the fact that the $\boldsymbol{z}_j$ are weak solutions of this eigenvalue problem, they are also orthogonal with respect to the products $B_u$ and $B_v$. We now search for weak solutions $(\boldsymbol{u}_m, \boldsymbol{v}_m)$ of the problem projected to the $m$-dimensional subspace $Z_m:=\langle z_j\rangle_{j=1,...,m}$ of $\mathcal{E}$, i.e. we search for solutions of
	\begin{align}\notag
	\langle\partial_t \boldsymbol{u}_m, \boldsymbol{z}_j\rangle_{dual}&=B_u[\boldsymbol{v}_m,\boldsymbol{z}_j]+b_u(\boldsymbol{z}_j,t)+f_u(\boldsymbol{z}_j, t ) & j&=1,...,m, \\\label{EquationWeakFormulationFiniteDimensionalProblem}
	(\boldsymbol{v}_m, \boldsymbol{z}_j)_{\gamma\mathcal{L}}&=B_v[\boldsymbol{u}_m,\boldsymbol{z}_j]
	 & j&=1,...,m, \\
	u(\cdot, 0)&=0.\notag
	\end{align}
	We will index $(\ref{EquationWeakFormulationFiniteDimensionalProblem})$ with an $m$ to clarify on which space we are projecting. Solutions of $\eqref{EquationWeakFormulationFiniteDimensionalProblem}_m$ are of the form
	\begin{align*}
	\boldsymbol{u}_m(x,t)&=\sum_{j=1}^m a_{mj}(t)\boldsymbol{z}_j(x),\\
	\boldsymbol{v}_m(x,t)&=\sum_{j=1}^mb_{mj}(t)\boldsymbol{z}_j(x).
	\end{align*} 
	Due to the orthonormality of the $\boldsymbol{z}_j$ with respect to the product (\ref{EquationInnerProduct}) we get the following ODE-System for the coefficient functions $a_{mj}$ and $b_{mj}$, where we use the orthogonality properties with respect to $B_u$ and $B_v$ discussed above.
	\begin{align*}
	a_{mj}'(t)&=b_{mj}(t)B_u[\boldsymbol{z}_j,\boldsymbol{z}_j]+b_u(\boldsymbol{z}_j;t)+f_u(\boldsymbol{z}_j;t) & &j=1,..., m,\\
	b_{mj}(t)&= a_{mj}(t)B_u[\boldsymbol{z}_j, \boldsymbol{z}_j] & &j=1,..., m,\\
	a_{mj}(0)&=0 & &j=1,..., m.
	\end{align*}
	Inserting the equations for $b_{mj}(t)$ in those for $a'_{mj}(t)$ we get a standard system of ODEs for which we can apply the Theorem of Caratheodory to show existence of absolute continuous solutions $a_{mj}$. Inserting this solution in the second equation we also get the $b_{mj}$ as functions in $L^2(0,T; \mb{R})$.\\
	The next thing to do is verifying an energy estimate for our problem. We notice that for the solution $(u_m,v_m)$ of $(\ref{EquationWeakFormulationFiniteDimensionalProblem})_m$, the functions $\partial_t u_m$ and $v_m$ are valid testfunctions. Thus, we can test the first equation in $(\ref{EquationWeakFormulationFiniteDimensionalProblem})_m$ with $v_m$ and the second equation with $\partial_t u_m$ to get
	\begin{align}
	\sum_{i=1}^3\gamma^i\int_{\Gamma_{\ast}^i}\partial_t u_m^iv_m^i\dH^n&=-\sum_{i=1}^3\gamma^i\int_{\Gamma_{\ast}^i}|\nabla_{\ast}v_m^i|^2\dH^n-\gamma^iC_u\sum_{i=1}^3\int_{\Gamma_{\ast}^i}|v_m^i|^2\dH^n\notag\\
	&+\sum_{i=1}^3\gamma^i\int_{\Gamma_{\ast}^i}\mathfrak{f}^iv_m^i\dH^n+\int_{\Sigma_{\ast}}\gamma^1\mathfrak{b}^5v_m^1-\gamma^3\mathfrak{b}^6v_m^3\dH^{n-1},\\
	\sum_{i=1}^3\gamma^i\int_{\Gamma_{\ast}^i}v_m^i\partial_t u_m^i\dH^n&=\sum_{i=1}^3\gamma^i\partial_t\left(\frac{1}{2}\int_{\Gamma_{\ast}^i}|\nabla_{\ast}u^i|^2\dH^n+\frac{1}{2}C_v\int_{\Gamma_{\ast}^i}|u_m^i|^2\dH^n\right).\notag
	\end{align}
	Subtracting the first equation from the second and rearranging the terms leads to 
	\begin{align}
	\sum_{i=1}^3\gamma^i\Bigg(\int_{\Gamma_{\ast}^i}|\nabla_{\ast}v_m^i|^2\dH^n&+C_u\int_{\Gamma_{\ast}^i}|v_m^i|^2\dH^n\Bigg)+\gamma^i\partial_t\left(\frac{1}{2}\int_{\Gamma_{\ast}^i}|\nabla_{\ast} u_m^i|^2+C_v|u_m^i|^2\dH^n\right)\label{Equation31028}\\
	&= \sum_{i=1}^3\left(\gamma^i\int_{\Gamma_{\ast}^i}\mathfrak{f}^iv_m^i\dH^n\right)+\int_{\Sigma_{\ast}}\gamma^1\mathfrak{b}^5v_m^1-\gamma^3\mathfrak{b}^6v_m^3\dH^{n-1}.\notag
	\end{align}
	Applying the weighted Young-inequality gives us
	\begin{align}\label{EquationFUCKTHISSHIT}
	\sum_{i=1}^3\gamma^i\int_{\Gamma_{\ast}^i}\mathfrak{f}^iv_m^i\dH^n&\le\sum_{i=1}^3\frac{\gamma^i}{2\varepsilon}\int_{\Gamma_{\ast}^i}|\mathfrak{f}^i|^2\dH^n+\frac{\varepsilon\gamma^i}{2}\int_{\Gamma_{\ast}^i}|v_m^i|^2\dH^n,\\
	\int_{\Sigma_{\ast}}\gamma^1\mathfrak{b}^5v_m^1\dH^{n-1}&\le\frac{1}{2\varepsilon'}\|\gamma^1\mathfrak{b}^5\|^2_{L^2(\Sigma_{\ast})}+\frac{\varepsilon'}{2}\|v_m^1\|^2_{L^2(\Sigmaast)},  \\
	\int_{\Sigmaast}-\gamma^3\mathfrak{b}^6v_m^3\dH^{n-1}&\le \frac{1}{2\varepsilon'}\|\gamma^3\mathfrak{b}^6\|^2_{L^2(\Sigma_{\ast})}+\frac{\varepsilon'}{2}\|v_m^3\|^2_{L^2(\Sigmaast)}.
	\end{align}
	Note now that Lemma \ref{LemmaHilfslemmafurEnergieabschatzungschwacheLt} implies for any $\boldsymbol{g}\in\mathcal{H}^1$ that
	\begin{align*}
	\|g^i\|_{L^2(\Sigmaast)}^2\le \overline{\varepsilon}^2\|\gast\boldsymbol{g}\|_{\mathcal{L}}^2+2\overline{\varepsilon}C_{\overline{\varepsilon}}\|\gast\boldsymbol{g}\|_{\mathcal{L}}\|\boldsymbol{g}\|_{\mathcal{L}}+C_{\overline{\varepsilon}}^2\|\boldsymbol{g}\|_{\mathcal{L}}^2\le 2\overline{\varepsilon}^2\|\gast\boldsymbol{g}\|_{\mathcal{L}}^2+2C_{\overline{\varepsilon}}^2\|\boldsymbol{g}\|_{\mathcal{L}}^2.
	\end{align*}
	Therefore, we conclude 
	\begin{align}\label{EquationFUCKTHISSHIT2}
	\frac{\varepsilon'}{2}\left(\|v_m^1\|_{L^2(\Sigmaast)}^2+\|v_m^3\|_{L^2(\Sigmaast)}^2 \right)\le \overline{\varepsilon}^2\varepsilon'\|\gast\boldsymbol{w}_m\|_{\mathcal{L}}^2+C_{\overline{\varepsilon}}^2\varepsilon'\|\boldsymbol{w}_m\|_{\mathcal{L}}^2.
	\end{align}
	Applying (\ref{EquationFUCKTHISSHIT})-(\ref{EquationFUCKTHISSHIT2}) on (\ref{Equation31028}) we get
	\begin{align}
	\sum_{i=1}^3&\gamma^i\Bigg(\int_{\Gamma_{\ast}^i}|\nabla_{\ast}v_m^i|^2\dH^n+C_u\int_{\Gamma_{\ast}^i}|v_m^i|^2\dH^n\Bigg)+\partial_t\left(\gamma^i\frac{1}{2}\int_{\Gamma_{\ast}^i}|\nabla_{\ast} u_m^i|^2+C_v|u_m^i|^2\dH^n\right)\notag\\
	&\le\sum_{i=1}^3\frac{\gamma^i}{2\varepsilon}\int_{\Gamma_{\ast}^i}|\mathfrak{f}^i|^2\dH^n+\frac{1}{2\varepsilon'}\left(\|\gamma^1\mathfrak{b}^5\|^2_{L^2(\Sigmaast)}+\|\gamma^3\mathfrak{b}^6\|_{L^2(\Sigmaast)}^2\right)\\
	&\phantom{=}+\sum_{i=1}^3\frac{\varepsilon\gamma^i}{2}\int_{\Gamma_{\ast}^i}|v_m^i|^2\dH^n+\overline{\varepsilon}^2\varepsilon'\|\gast\boldsymbol{w}_m\|_{\mathcal{L}}^2+C_{\overline{\varepsilon}}^2\varepsilon'\|\boldsymbol{w}_m\|_{\mathcal{L}}^2.\notag
	\end{align}
	Now choosing $\overline{\varepsilon}=\frac{1}{2}, \varepsilon'=\min(\gamma_1,\gamma_2,\gamma_3), \varepsilon=1$ and  then $C_u$ large enough we can absorb the $\boldsymbol{w}_m$- and $\nabla_{\ast}\boldsymbol{w}_m$-terms on the right-hand side by the ones on the left-hand side to get
	\begin{align}\label{EquationEnergieabschaetzungSchwacheExistenzTripleJunction}
	&\partial_t (\|\boldsymbol{u}_m\|_{\gamma\mathcal{L}}^2+\|\nabla_{\ast}\boldsymbol{u}_m\|_{\gamma\mathcal{L}}^2)+C(\|\boldsymbol{w}_m\|_{\mathcal{L}}^2+\|\nabla_{\ast}\boldsymbol{w}_m\|_{\mathcal{L}}^2)\le C\big(\|\mathfrak{b}^5\|_{L^2(\Sigma_{\ast})}^2+\|\mathfrak{b}^6\|_{L^2(\Sigma_{\ast})}^2+\|\mathfrak{f}\|_{\mathcal{L}}^2 \big).
	\end{align}
	In particular, we get for all $t\in [0,T]$ that
	\begin{align}
	\partial_t(\|\boldsymbol{u}_m\|_{\gamma\mathcal{L}}^2+\|\nabla_{\ast}\boldsymbol{u}_m\|_{\gamma\mathcal{L}}^2)\le C\big(\|\mathfrak{b}^5\|_{L^2(\Sigma_{\ast})}^2+\|\mathfrak{b}^6\|_{L^2(\Sigma_{\ast})}^2+\|\mathfrak{f}\|_{\mathcal{L}}^2 \big)
	\end{align}
	Integrating this in time using $u(0)\equiv 0$ leads to
	\begin{align}
	\|\boldsymbol{u}_m(t)\|_{\gamma\mathcal{L}}^2+\|\nabla_{\ast}\boldsymbol{u}_m(t)\|_{\gamma\mathcal{L}}^2&\le C\int_0^t\left(\|\mathfrak{b}^5(t)\|_{L^2(\Sigma_{\ast})}^2+\|\mathfrak{b}^6(t)\|_{L^2(\Sigma_{\ast})}^2+\|\mathfrak{f}(t)\|_{\mathcal{L}}^2 \right)dt\\
	&\le C \left(\|\mathfrak{b}^5\|_{L^2(0,T; L^2(\Sigma_{\ast}))}^2+\|\mathfrak{b}^6\|_{L^2(0,T; L^2(\Sigma_{\ast}))}^2+\|\mathfrak{f}\|_{L^2(0,T; \mathcal{L})}^2\right) \notag
	\end{align}
	for all $t\in [0,T]$. 
	Integrating (\ref{EquationEnergieabschaetzungSchwacheExistenzTripleJunction}) from $0$ to $T$ implies for $\boldsymbol{v}_m$ that
	\begin{align}\label{EquationEnergieabschaetzungSchwachesExistenzTripleJunctionBoundforwm}
	\|\boldsymbol{v}_m\|_{L^2(0,T; \mathcal{E})}^2\le C\big(\|\mathfrak{b}^5\|_{L^2(0,T; L^2(\Sigma_{\ast}))}^2+\|\mathfrak{b}^6\|_{L^2(0,T; L^2(\Sigma_{\ast}))}^2+\|\mathfrak{f}\|_{L^2(0,T; \mathcal{L})}^2\big).
	\end{align}
	It remains to study the norm of $\partial_t \boldsymbol{u}_m$ which can be carried out with a standard argument. We first introduce the space $\gamma\mathcal{E}$ that contains all elements of $\mathcal{E}$ and is equipped with the inner product
	\begin{align}
	(\boldsymbol{v}, \boldsymbol{w})_{\gamma\mathcal{E}}:=(\nabla_{\ast}\boldsymbol{v},\gast\boldsymbol{w})_{\gamma\mathcal{L}}+(\boldsymbol{v},\boldsymbol{w})_{\gamma\mathcal{L}}.
	\end{align}
	Note that as the $\boldsymbol{z}_j$ are classical solutions of the eigenvalue problem $\eqref{EquationProofofExistenceofWeakSolution}$ they are also orthogonal sytem in $\gamma\mathcal{E}$. Also, the norm induced by the $\gamma\mathcal{E}$-product is equivalent to the norm on $\mathcal{E}$. Now, choosing any $\boldsymbol{v}\in\mathcal{E}$ with $\|\boldsymbol{v}\|_{\mathcal{E}}\le 1$ we write $\boldsymbol{v}=\boldsymbol{v}_1+\boldsymbol{v}_2$ with 
	\begin{align}\boldsymbol{v}_1\in \text{span}\{\boldsymbol{z}_j\}_{j=1,...,m}, \quad (\boldsymbol{v}_2, \boldsymbol{z}_j)_{\gamma\mathcal{L}}=0, j=1,..., m.
	\end{align} 
	Due to equivalence of norms we get $\|\boldsymbol{v}\|\le C'$ for a constant independent of $\boldsymbol{v}$ and
	\begin{align}
	\|\boldsymbol{v}_1\|_{\gamma\mathcal{E}}\le \|\boldsymbol{v}\|_{\gamma\mathcal{E}}\le C,
	\end{align} due to the orthogonality of the $\boldsymbol{z}_j$ in $\mathcal{E}_{\gamma}$. Then, again using equivalence of the norms on $\mathcal{E}$ and $\gamma\mathcal{E}$, we conclude $\|\boldsymbol{v}_1\|_{\mathcal{E}}\le C$ for a $C$ independent of $\boldsymbol{v}$. With this in mind we get the estimate
	\begin{align*}
	|\langle \partial_t \boldsymbol{u}_m,\boldsymbol{v}\rangle_{dual}|&=|( \partial_t\boldsymbol{u}_m, \boldsymbol{v}_1)_{\gamma\mathcal{L}}|=\big(|B(\boldsymbol{v}_m, \boldsymbol{v}_1)|+|f_u(\boldsymbol{v}_1))|+|b_u(\boldsymbol{v}_1)|\big)\\
	&\le C\big (\|\boldsymbol{v}_m\|_{\mathcal{E}}+\|\mathfrak{f}\|_{\mathcal{L}}+\|\mathfrak{b}^5\|_{L^2(\Sigma_{\ast})}+\|\mathfrak{b}^6\|_{L^2(\Sigma_{\ast})}\big)\\
	&\le C\big (\|\mathfrak{f}\|_{\mathcal{L}}+\|\mathfrak{b}^5\|_{L^2(\Sigma_{\ast})}+\|\mathfrak{b}^6\|_{L^2(\Sigma_{\ast})}\big).
	\end{align*}
	Here, we used (\ref{EquationEnergieabschaetzungSchwacheExistenzTripleJunction}) in the last inequality and for $b_u(\boldsymbol{v}_1;t)$ we used the Cauchy-Schwarz inequality and continuity of the trace operator to derive
	\begin{align*}
	|b_u(\boldsymbol{v}_1;t)|&\le C\big(\|\mathfrak{b}^5(t)\|_{L^2(\Sigma_{\ast})}\|v_1^1\|_{L^2(\Sigma_{\ast})}+\|\mathfrak{b}^6(t)\|_{L^2(\Sigma_{\ast})}\|v^3_1\|_{L^2(\Sigma_{\ast})}\big)\\
	&\le C\big(\|\mathfrak{b}^5(t)\|_{L^2(\Sigma_{\ast})}+\|\mathfrak{b}^6(t)\|_{L^2(\Sigma_{\ast})}\big)\|\boldsymbol{v}_1\|_{\mathcal{E}}\\
	&\le C\big(\|\mathfrak{b}^5(t)\|_{L^2(\Sigma_{\ast})}+\|\mathfrak{b}^6(t)\|_{L^2(\Sigma_{\ast})}\big).
	\end{align*}
	As this holds for all $\boldsymbol{v}\in\mathcal{E}$ with $\|\boldsymbol{v}\|_{\mathcal{E}}\le 1$ we deduce for almost all $t\in[0,T]$ that
	\begin{align}
	\|\partial_t \boldsymbol{u}_m(t)\|_{\mathcal{E}^{-1}}\le C\big(\|\mathfrak{f}\|_{\mathcal{L}}+\|\mathfrak{b}^5(t)\|_{L^2(\Sigma_{\ast})}+\|\mathfrak{b}^6(t)\|_{L^2(\Sigma_{\ast})}\big),
	\end{align}
	and thus
	\begin{align}
	\|\partial_t \boldsymbol{u}_m\|_{L^2(0,T, \mathcal{E}^{-1})}^2\le C\big(\|\mathfrak{f}\|_{L^2(0,T; \mathcal{L})}^2+\|\mathfrak{b}^5\|_{L^2(0,T,L^2(\Sigma_{\ast}))}^2+\|\mathfrak{b}^6\|_{L^2(0,T,L^2(\Sigma_{\ast}))}^2\big).
	\end{align}
	So in total we get the energy estimate
	\begin{align}
	\underset{t\in[0,T]}{\max}\left(\|\boldsymbol{u}_m(t)\|_{\mathcal{E}}^2\right)&+\|\boldsymbol{u}_m\|_{L^2(0,T, \mathcal{E})}^2+\|\boldsymbol{v}_m\|_{L^2(0,T, \mathcal{E})}^2+\|\partial_t \boldsymbol{u}_m\|_{L^2(0,T, \mathcal{E}^{-1})}^2\\ \notag
	&\le C\big(\|\mathfrak{f}\|_{L^2(0,T; \mathcal{L})}^2+\|\mathfrak{b}^5\|_{L^2(0,T,L^2(\Sigma_{\ast}))}^2+\|\mathfrak{b}^6\|_{L^2(0,T,L^2(\Sigma_{\ast}))}^2\big). 
	\end{align}
	These energy estimates imply that there are $\boldsymbol{u}\in L^2(0,T; \mathcal{E}), \partial_t\boldsymbol{u}\in L^2(0,T; \mathcal{E}^{-1})$ and $\boldsymbol{v} \in L^2(0,T; \mathcal{E}))$ together with subsequences (which we will identify with the original sequence) of $u_m, u_m'$ and $v_m$ such that 
	\begin{align*}
	\begin{cases}
	\boldsymbol{u}_m\rightharpoonup \boldsymbol{u} & \text{in }L^2(0,T; \mathcal{E}),\\
	\partial_t \boldsymbol{u_m}\rightharpoonup \partial_t\boldsymbol{u} &\text{in }L^2(0,T; \mathcal{E}^{-1}),\\
	\boldsymbol{v}_m\rightharpoonup \boldsymbol{v}&\text{in }L^2(0,T; \mathcal{E}).
	\end{cases}
	\end{align*}
	We note that like in the standard case we get that the weak limit of $\partial_t \boldsymbol{u}_m$ indeed corresponds with the weak time derivative of $\boldsymbol{u}$ by using the definition of a weak time derivative and weak convergence. \\
	Next we want to see that $(\boldsymbol{u}, \boldsymbol{v})$ is the sought weak solution. For any fixed $N\in\N$ and smooth functions $\{d_k\}_{k=1,..., N}$ we get for any $m>N$ that
	\begin{align*}
	\int_0^T\langle \partial_t \boldsymbol{u}_m, \boldsymbol{\zeta}\rangle_{dual,\gamma}dt&= \int_0^T B_u[\boldsymbol{v}_m, \boldsymbol{\zeta};t]+b_u(\boldsymbol{\zeta}; t)+f_u(\boldsymbol{\zeta},t)dt,\\
	\int_0^T(\boldsymbol{v}_m, \boldsymbol{\zeta})_{\mathcal{L}_{\gamma}}dt&=
	\int_0^TB_v[\boldsymbol{u}_m, \boldsymbol{\zeta}; t],
	\end{align*}
	for $\boldsymbol{\zeta}=\sum_{j=1}^Nd_j\boldsymbol{z}_j$. Note that for fixed $\boldsymbol{\zeta}$ each term gives an element of $\left(L(0,T; \mathcal{E})\right)'$ or $\left(L(0,T, \mathcal{E}^{-1})\right)'$ and so we can - using the definition of weak convergence - pass to the limit to get
	\begin{align*}
	\int_0^T\langle \partial_t \boldsymbol{u}, \boldsymbol{\zeta}\rangle_{dual,\gamma}dt&= \int_0^T B_u[\boldsymbol{v}, \boldsymbol{\zeta};t]+b_u(\boldsymbol{\zeta}; t)+f_u(\boldsymbol{\zeta},t)dt,\\
	\int_0^T(\boldsymbol{v}, \boldsymbol{\zeta})_{\mathcal{L}_{\gamma}}dt&=
	\int_0^TB_v[\boldsymbol{u}, \boldsymbol{\zeta}; t].
	\end{align*}
	As the considered test functions $\boldsymbol{\zeta}$ are dense in $L^2(0,T; \mathcal{E})$ this holds for all such functions implying that 
	(\ref{EquationWeakformulationofLinearSDFTJ}) holds for almost all $t\in [0,T]$. Thus, it remains to show that $u(0)=0$, which follows as in the proof of \cite[Theorem 3, p.378]{evansPDE}. Finally, to derive uniqueness of the weak solution we observe that for two solutions $(\boldsymbol{u}_1, \boldsymbol{v}_1), (\boldsymbol{u}_2, \boldsymbol{v}_2)$ the difference $(\bar{\boldsymbol{u}},\bar{\boldsymbol{v}})$ solves (\ref{EquationWeakformulationofLinearSDFTJ}) with $\mathfrak{f}\equiv 0, \mathfrak{b}^5\equiv\mathfrak{b}^6\equiv 0$ and then the energy estimates (\ref{EquationEnergieabschaetzungSchwacheExistenzTripleJunction}) imply $(\bar{\boldsymbol{u}}, \bar{\boldsymbol{v}})=(0,0)$. This finishes the proof.
\end{proof}
For technical reasons in the localization argument we will need a result on higher Sobolev regularity.
\begin{corollary}[$H^3$-regularity of $\boldsymbol{u}$]\label{CorollaryH3regularityWeakSolutionTJ}\ \\
	The solution $\boldsymbol{u}$ found in Proposition \ref{PropostionsExistenceofWeakSolutionsSurfaceDiffsuionTripleJunction} is in $L^2(0,T; H^3_{TJ}(\Gammaast))$ and we have the estimate
	\begin{align}\label{EquationEnhacedRegularityforweaksolutionTripleJunction}
	\|\boldsymbol{u}\|_{L^2(0,T; H^3_{TJ}(\Gammaast))}\le C\left(\sum_{i=1}^3\|\mathfrak{f}^i\|_{L^2(0,T; \mathcal{L})}+\sum_{i=5}^6\|\mathfrak{b}^i\|_{L^2(0,T; L^2(\Sigma_{\ast}))}\right).
	\end{align}	
\end{corollary}
\begin{proof}
	For every $t\in[0,T]$, the function $\boldsymbol{u}$ is a weak solution of the elliptic problem
	\begin{align*}
	-\Dast u^i+C_vu^i&=v^i &   &\text{on }\Gamma^i_{\ast}, i=1,2,3,\\
	\gamma^1u^1+\gamma^2u^2+\gamma^3u^3&=0   &  &\text{on }\Sigma_{\ast},\\
	\partial_{\nu^1_{\ast}}u^1-\partial_{\nu^2_{\ast}}u^2&=0 & &\text{on }\Sigma_{\ast},\\
	\partial_{\nu^2_{\ast}}u^2-\partial_{\nu^3_{\ast}}u^3&=0 & &\text{on }\Sigma_{\ast}.
	\end{align*}
	As we have $\boldsymbol{v}\in W^{1,2}_{TJ}(\Gamma)$ and it was shown in \cite[Lemma 3]{depner2014mean} that the Lopatinskii-Shapiro conditions for these boundary conditions are satisfied, we may apply elliptic regularity theory from \cite{agmondouglisNi1959estimates} together with a localization argument like in \cite[Section 4]{depner2014mean} to get $\boldsymbol{u}(t)\in W^{3,2}_{TJ}(\Gamma)$ and
	\begin{align*}
	\|\boldsymbol{u}(t)\|_{H^3_{TJ}(\Gammaast)}\le \|\boldsymbol{v}(t)\|_{\mathcal{E}}.
	\end{align*}
	Then, we can use (\ref{EquationEnergyEstiamteWeakSolutionTJ}) to conclude
	\begin{align*}
	\|\boldsymbol{u}\|_{L^2(0,T; H^3(\Gammaast))}\le\|\boldsymbol{v}\|_{L^2(0,T;\mathcal{E})}\le C\left(\sum_{i=1}^3\|\mathfrak{f}^i\|_{L^2(0,T; \mathcal{L})}+\sum_{i=5}^6\|\mathfrak{b}^i\|_{L^2(0,T; L^2(\Sigma_{\ast}))}\right).
	\end{align*}
	This shows the desired estimate.
\end{proof}

\subsection{Local Schauder Estimates}\label{SubsectionSchauderEstimates}
\label{SubsectionLocalSchauderregularity}
Now we want to show that the weak solution found in the last section is actually in $C^{4+\alpha,1+\alpha/4}_{TJ}(\Gammaast\times[0,T])$ for every $T>0$. This splits into two parts. First we will localize the problem and verify that the localization has a unique solution in $C^{4+\alpha,1+\alpha/4}$. Then, we will connect the localization with our weak solution from Proposition \ref{PropostionsExistenceofWeakSolutionsSurfaceDiffsuionTripleJunction} using a compactness argument. We will discuss only the situation locally around points on the triple junction. In the interior of any of the $\Gamma^i$ the problem reduces to an equation solely for $u^i$ with vanishing boundary conditions. Details for this can be found in \cite[Proposition 3.8]{goesswein2019Dissertation}.
\\
Fix any $\sigma\in\Sigmaast$. We construct a special parametrization as follows to make the localized problem more comfortable to deal with. We choose $\delta>0$ sufficiently small such that for $i=1,2,3$ the projection on $\Sigma_{\ast}$ is well defined on $V^i=B_{\delta}(\sigma)\subset \Gamma_{\ast}^i$. For $V_{\Sigma}:=B_{\delta}(\sigma)\subset \Sigma_{\ast}$ we choose $R>0$ together with $U=B_R(0)\cap \R^n_+$ and any parametrization $\varphi: U\cap \{x\in\R^n|x_n=0\}\to V_{\Sigma}$. Now we extend this $\varphi$ to diffeomorphisms $\varphi^i: U\to V^i$ by the distance function. To be precise this induces diffeomorphisms
\begin{align}\label{EquationParmaetrizationnearboundary}
\varphi^i: U\mapsto V_i, (x, d)\mapsto \gamma_{-\nu_{\ast}^i}(\varphi(x), d),
\end{align} 
where $(x,0)\in U\cap \{x\in\R^n|x_n=0\}, (x,d)\in U$ and $\gamma_{-\nu_{\ast}^i}(\sigma, d)$ denotes the evaluation of the geodesic through a point $\sigma\in\Sigmaast$ in direction $-\nu_{\ast}^i$ at distance $d$ . Note that for these parametrizations it holds that
\begin{align}\label{EquationPropertiesofCoordinatesforLopatinski} 
g^i_{nn}=1,\quad g^i_{nl}=0\text{    for } l\neq n,
\end{align}
and the same holds for the inverse metric tensor
due to to the inverse matrix formula for matrices in block form. We now want to study problem (\ref{EquationLinearisedSDFTJWeakPart}) localized on $B_R(0)$. 
Using the notation
\begin{align*}
C:=\partial U\cap\{x_n=0\},\quad S:=\partial U \backslash C,
\end{align*} 
we see that for the localization of (\ref{EquationLinearisedSDFTJWeakPart}) we only have boundary conditions on $C$. In order to get a well-posed problem we have to do a cut-off away from $S$ so we choose $\varepsilon<\frac{R}{2}$ together with a cut-off function $\eta$ with
\begin{align}\label{EquationFirstDefinitionCutOffEta}
\eta\in C^{\infty}(U),\quad \supp(\eta)\subset U\cap B_{R-\varepsilon}(0),\quad \eta\equiv 1\text{ on }U\cap B_{\varepsilon}(0).
\end{align}
 In the following, we will write for this cut-off of the parametrized function $\boldsymbol{u}$ again to keep notation simple. Now, we want to consider the problem induced by $(LSDFTJ)_P$, that is
\begin{align*}
(LLP)\begin{cases}
\partial_t u^i+ \sum\limits_{j,k,l,m=1}^ng_i^{jk}g_i^{lm}\partial_{jklm}u^i+\mathcal{A}^i_L(u_i)=f^i & \text{ on }U\times[0,T],\quad i=1,2,3,\\
\gamma^1u^1+\gamma^2u^2+\gamma^3u^3=0 & \text{ on }C\times[0,T],\\
\partial_n u^1-\partial_n u^2+\mathcal{B}_{2}(u^1,u^2)=0 &\text{ on }C\times [0,T],\\
\partial_n u^2-\partial_n u^3+\mathcal{B}_{3}(u^2,u^3)=0 &\text{ on }C\times [0,T],\\
\sum\limits_{i=1}^3\ \sum\limits_{j,k=1}^n\gamma^ig^{jk}\partial_{jk} u^i+\mathcal{B}_4(u^1, u^2, u^3)=0 & \text{ on }C\times[0,T],\\
\sum\limits_{j,k=1}^n\big(g^{jk}_1\partial_{njk}u^1-g^{jk}_2\partial_{njk}u^2\big)+\mathcal{B}_5(u^1, u^2)=\mathfrak{b}^5 &\text{ on }C\times[0,T],\\
\sum\limits_{j,k=1}^n\big(g^{jk}_2\partial_{njk}u^2-g^{jk}_3\partial_{njk}u^3\big)+\mathcal{B}_6(u^2, u^3)=\mathfrak{b}^6 &\text{ on }C\times[0,T],\\
u^i=0 & \text{ on }S\times[0,T],\quad i=1,2,3,\\
\Delta u^i=0 & \text{ on }S\times[0,T],\quad i=1,2,3,\\
u^i\big|_{t=0}=u_0 &\text{ on }U, i=1,2,3.
\end{cases}
\end{align*}
Here, we used that due to the choice of our coordinates the derivative in direction of $\nu_{\ast}^i$ is given by $-\partial_n$ on $C$. Also, we only need the highest order terms for the following discussion. All other terms are written in $\mathcal{A}_L$ and $\mathcal{B}_j$. We do not allow inhomogeneities for the first four boundary conditions as we also cannot have them in our system yet. Actually, the following analysis for $(LLP)$ would also admit these inhomogeneities. The boundary conditions on $S$ are relatively arbitrary as $\boldsymbol{u}$ vanishes near $S$ and so fulfils any linear boundary condition. Note that the boundary condition on $S$ and $C$ are compatible as $\bu$ solves both near $\partial S\cap\partial C$ due to the cut-off procedure. As a final remark we want to mention that we will need better properties for $\eta$ later to guarantee that no inhomogeneities in the lower order boundary conditions of $(LLP)$ will arise.  But we will discuss this in the second half of this section, when we connect the localization with the weak solution. Now, we want to show that $(LLP)$ fulfills the prerequisites to apply \cite[Theorem 4.9]{ladyzhenskaia1968linear}.
\begin{remark}[Regularity of $\partial U$]\ \\
	To be precise we will need smoothness of $\partial U$ for the analysis now. But as we cut off the problem away from $S$ we may assume w.l.o.g. that the problem is indeed defined on such a domain.
\end{remark} 
Firstly, we want to see that for the system above the basic requirements described on page 8 of \cite{ladyzhenskaia1968linear} hold. Setting $s_i=0$ and $t_i=4$ for $i=1,2,3$, the degree conditions for the differential operators are fulfilled. For the parabolicity condition we see that $b=2$ and
\begin{align*}
\mathcal{L}_0(x,t, \zeta,p)&=\text{diag}(p+ |\zeta|_{g_{1}}^4, p+|\zeta|_{g_2}^4, p+|\zeta|_{g_3}^4),\\
L(x,t,\zeta,p)             &=\prod_{i=1}^3(p+|\zeta|_{g_i}^4),
\end{align*}
for any $p\in\C, \zeta\in\R^n$, where $|\cdot|_{g_i}$ denotes the norm induced by the inverse metric tensor $g_i^{\ast}$. Thus, the roots of $L$ with respect to $p$ are precisely $p_i=-|\zeta|_{g_i}^4$, for which one has
\begin{align*}
p_i\le-c^4|\zeta|^{2b},
\end{align*}
for a suitable $c>0$ fulfilling $c|\zeta|\le|\zeta|_{g_i}$ for $i=1,2,3$, which exists due to the equivalency of $|\cdot|$ and $|\cdot|_{g_i}$. Note that as the inverse metric tensor is bounded, $c$ can be chosen independently of $x$ and thus the equations is even uniformly parabolic with $b=2$.\\
For the boundary conditions we get the following numbers.
\begin{center}
	\begin{tabular}{c|ccc|c}
		$\beta_{qj}$   &    1 &     2    &     3 & $\sigma_q$ \\ \hline
		1            &    0 &     0    &     0 & -4 \\
		2            &    1 &     1    &     0 &  -3  \\
		3            &    0 &     1    &     1 & -3  \\
		4            &    2 &     2    &     2 & -2\\
		5            &    3 &     3    &     0 & -1 \\
		6            &    0 &     3    &     3 & -1  
	\end{tabular}
\end{center}
This means that the number $l$ used in Theorem 4.9 can be an arbitrary number larger than 
\begin{align*}
\max\{0, \sigma_1,\cdots, \sigma_6\}=0,
\end{align*}
which works for us as we will need $l$ to be the H\"older-continuity of the space derivatives, so an $\alpha\in(0,1)$. Hence, we see that (LLP) is indeed of the form covered by the theory of \cite{ladyzhenskaia1968linear} and so we now have to check the complementary and compatibility conditions. \\
The complementary conditions are the typical Lopatinski-Shapiro conditions. We prefer here to 
work with their ODE-version, which can be found in \cite{latushkin2006stable}. A proof, that this version is equivalent to the original version like in \cite{ladyzhenskaia1968linear}, can be found in \cite[Section I.2]{eidelman2012parabolic}. In our situations, the Lopatinski-Shapiro conditions read as follows. We need to verify that for all 
\begin{align*}
\zeta'\in \R^{n-1}, \lambda\in\{z\in\C|\Re(z)\ge 0\}\text{ with }(\lambda,\zeta')\neq(0,0)
\end{align*} the only solution $\boldsymbol{\phi}=(\phi_1, \phi_2, \phi_3)$ in $C_0(\R, \C^3)$ of the ODE-system
\begin{align}
\lambda\phi_i+|\zeta'|^4
_{g_i}\phi_i+\phi_i''''&=0, & &y>0,\quad i=1,2,3,\notag\\
\sum_i^3\gamma^i\phi_i&=0& &y=0,\notag\\
\phi_1'=\phi_2'&=\phi_3' & &y=0,\label{EquationLopatinskiShapiroSystem}\\
\sum_{i=1}^3\gamma^i\big(|\zeta'|_{g_i}^2\phi_i-\phi_i''\big)&=0 & &y=0,\notag\\
\phi_1'''-\phi_2'''-|\zeta'|^2_{g_1}\phi_1'+|\zeta'|^2_{g_2}\phi_2'&=0  & &y=0,\notag\\
\phi_2'''-\phi_3'''-|\zeta'|^2_{g_2}\phi_2'+|\zeta'|^2_{g_3}\phi_3'&=0 & &y=0\notag,
\end{align} 
is $\boldsymbol{\phi}\equiv0$. To prove this we use a straightforward energy method. We first note that due to the structure of the differential equation, all solutions are linear combinations of functions of the form $\exp(\sqrt[4]{-\lambda-|\zeta'|^4}y)$. If such functions converge to $0$ for $y\to\infty$, then all their derivatives converge, too. 
Now, testing the first equation in \eqref{EquationLopatinskiShapiroSystem} with $\gamma^i\overline{\phi_i}$, summing over $i=1,2,3$, integrating by parts and using the boundary conditions for $\phi_i$ and its derivatives gets us to
\begin{align}
\sum_{i=1}^3\gamma^i(\lambda+|\zeta'|^4_{g_i})\int_0^{\infty}|\phi_i|^2dy+\sum_{i=1}^3\gamma^i\int_0^{\infty}|\phi_i''|^2dy+2\mathfrak{Im}\left(\phi_1'\sum_{i=1}^3|\zeta'|^2_{g_i}\gamma^i\overline{\phi}_i\right)=0.
\end{align}
As $\lambda$ has a non-negative real part all terms of the left-hand-side have non-negative real part. For $\zeta'\neq 0$ we get therefore
\begin{align}\label{EquationLopatinskiShapiro4}
\int_0^{\infty}|\phi_i|^2dy=0, \quad i=1,2,3,
\end{align} 
which already implies $\boldsymbol{\phi}\equiv 0$. For $\zeta'=0$ the equation above reduces to
\begin{align}
\sum_{i=1}^3\gamma^i\lambda\int_0^{\infty}|\phi_i|^2dy+\sum_{i=1}^3\int_0^{\infty}|\phi_i''|^2dy=0,
\end{align}
as $|\cdot|_{g_i}$ is also a norm on $\R^n$. If $\mathfrak{Re}(\lambda)\neq 0$ we can argue as before. Otherwise it has to hold that $\mathfrak{Im}(\lambda)\neq 0$ and as the second sum is real this again implies \eqref{EquationLopatinskiShapiro4}  and hence we showed $\boldsymbol{\phi}\equiv 0$. Consequently, the Lopatinskii-Shapiro conditions are fulfilled on $\partial U$.\\
It now remains to consider the compatibility conditions according to \cite[p. 98]{ladyzhenskaia1968linear}. As we choose $l=\alpha<1$, we only need compatibility conditions of order $0$. Due to the values of $\sigma_q$ only for the first boundary equation $i_q$ can take the values $0$ and $1$. For all other $i_q$ we just get the trivial compatibility condition that the boundary condition has to be fulfilled for $\boldsymbol{u}=0$. Therefore, using the first line in $(LLP)$ the only non-trivial compatibility condition is 
\begin{align}
0=\sum_{i=1}^3\gamma^i\partial_t u^i=\sum_{i=1}^3\gamma^i\mathfrak{f}^i.
\end{align}
 With all this checked we can use \cite[Theorem 4.9]{ladyzhenskaia1968linear} to get the following existence result for $(LLP)$.
\begin{proposition}[Schauder Estimates for (LLP)]\label{PropositionSchauderestimatesforlocalizedSurfaceDiffusionnearTripleJunction}\ \\
	The system (LLP) has a unique solution $\boldsymbol{u}\in C^{4+\alpha, 1+\frac{\alpha}{4}}(U\times[0,T])^3$ if and only if the compatibility conditions
	\begin{align}
	\mathfrak{b}^5=\mathfrak{b}^6&=0  &   &\text{ on }C\times\{0\}, \\
	\sum_{i=1}^{3}\gamma^i\mathfrak{f}^i&=0 & &\text{ on }C \times\{0\},
	\end{align}
	are fulfilled.
\end{proposition}
In the next step we want to connect $(LLP)$ with the weak solution of $(LSDFTJ)_P$. For this, we consider the cut-off of solutions of \eqref{EquationWeakformulationofLinearSDFTJ} via the function $\eta$ from \eqref{EquationFirstDefinitionCutOffEta}. This gives us a new problem for which we derive the existence of unique solutions immediately. The solutions coincides with a cut-off of the solution $(\bu,\bw)$ constructed in Proposition \ref{PropostionsExistenceofWeakSolutionsSurfaceDiffsuionTripleJunction}. We then approximate the problem by smoothing the arising inhomogeneities. For this problem we may apply Proposition \ref{PropositionSchauderestimatesforlocalizedSurfaceDiffusionnearTripleJunction}. Taking the limits we derive H\"older-regularity for $\bu$, which finishes the analysis of \eqref{EquationReducedLinearProblem}.\\
Consider now a fixed point $(\sigma, t)\in \Sigma_{\ast}\times [0,T]$. Now we need more prerequisites for $\eta$ from $\eqref{EquationFirstDefinitionCutOffEta}$ as we have to guarantee that our cut-off procedure will not produce inhomogeneities in the linearized angle conditions or the linearization of $\mathcal{B}_4$. We choose neighborhoods $Q_4^i$ of $\sigma\in\Gamma_{\ast}^i$ such that the projections $\pr^i_{\Sigma_{\ast}}$ are well defined and the intersections $Q_4^i\cap\Sigma_{\ast}$ equal for all $i=1,2,3$ a common set $Q_4^{\Sigmaast}\subset\Sigma_{\ast}$. We choose now a cut-off function $\eta_{\Sigmaast}\in C^{\infty}(Q_4^{\Sigmaast})$, such that $\eta_{\Sigmaast}$ has compact support in $Q_4^{\Sigma_{\ast}}$ and $\eta_{\Sigmaast}\equiv 1$ on a neighborhood of $\sigma$ in $Q_4^{\Sigma_{\ast}}$. Now, we extend $\eta_{\Sigmaast}$ via a cut-off of the distance function on $Q_4^i$. Precisely, we choose a cut-off function $\eta_d: [0,1]\to [0,1]$ with $\supp(\eta_d)\subset [0,\varepsilon)$ for some $\varepsilon<1$, $\eta_d\equiv 1$ on a closed interval containing $0$ and such that for all $i=1,2,3$ the function 
\begin{align}\label{EquationDefinitionCutofffunctionTriplejunction}
\eta^i:=\eta_d\eta_{\Sigma}: Q_4^i&\to [0,1],\\
x&\mapsto \eta_d(\dist_{\Sigma_{\ast}}(x))\eta_{\Sigma}(\pr^i_{\Sigma_{\ast}}(x)),  \notag
\end{align}
has compact support $Q_3^i$ in $Q_4^i$. Note that for any point $x\in Q^{\Sigma}$ we have now 
\begin{align}\label{EquationNormalderivativeofcutofffunction}
\partial_{\nu^i_{\ast}}\eta^i(x)&=0 & i&=1,2,3.
\end{align}
 This is exactly what we need for our analysis. We will need some additional notation now and set
\begin{align*}
C_l^i:=\partial Q_l^i\cap \Sigma_{\ast},\quad S_l^i:=\partial Q_l^i\backslash C_l^i,\quad l=1,2,3,4,\ i=1,2,3.
\end{align*}   
Hereby, the sets $Q_2^i$ and $Q_1^i$ will be constructed later and as $Q^i_l\cap\Sigmaast=Q_l^{\Sigmaast}$ for $i=1,2,3$ we have that $C^i_l=C_l$ for $i=1,2,3$. Following our plan, we now want to derive a PDE for $\widetilde{\boldsymbol{u}}=\boldsymbol{\eta u}$. As $\boldsymbol{u}$ is a weak solution of (\ref{EquationLinearisedSDFTJWeakPart}) we calculate formally
\begin{align}
\partial_t\widetilde{u}^i&=\eta^i\partial_t u^i+(\partial_t \eta^i)u^i &&\text{on }Q_4^i\times[0,T],\ &&i=1,2,3,\notag\\
\Dast\widetilde{u}^i&=\eta^i\Dast u^i+2\langle\gast\eta^i, \gast u^i\rangle+(\Dast\eta^i)u^i &&\text{on }Q_4^i\times[0,T],\ &&i=1,2,3,\notag\\
-\Dast\Dast\widetilde{u}^i&=-\eta^i\Dast\Dast u^i-4\Dast\langle \nabla_{\ast}\eta^i, \nabla_{\ast}u^i\rangle\label{EquationIdentitiesfortildeuTripleJunction}\\
&\phantom{=\ }+2\Dast u^i\Dast\eta^i-u^i\Dast\Dast\eta^i  &&\text{on }Q_4^i\times[0,T],\ &&i=1,2,3,\notag\\
\partial_{\nu^i_{\ast}}\widetilde{u}^i&=\eta^i\partial_{\nu_{\ast}^i}u^i+u^i\partial_{\nu_{\ast}^i}\eta^i=\eta^i\partial_{\nu_{\ast}^i}u^i &&\text{on }Q_4^{\Sigmaast}\times[0,T],\ &&i=1,2,3.\notag
\end{align} 
From this we see directly that $\widetilde{\boldsymbol{u}}$ solves formally
\begin{align*}
\partial_t\widetilde{u}^i=-\Dasti\Dasti\widetilde{u}^i+C_v\Dasti\widetilde{u}^i+C_u\Dasti\widetilde{u}^i-C_vC_u\widetilde{u}^i+\widetilde{\mathfrak{f}}^i,
\end{align*}
with the new inhomogeneity
\begin{align*}
\widetilde{\mathfrak{f}}^i=\eta^i\mathfrak{f}^i&+(\partial_t\eta^i)u^i+
4\Dast(\langle\gast\eta^i, \gast u^i\rangle)-2\Dast u^i\Dast\eta^i\\
&+u\Dast\Dast\eta^i-2(C_u+C_v)\langle\gast\eta^i,\gast u^i\rangle-(C_u+C_v)\Dast\eta^i\cdot u^i.
\end{align*}
Observe that due to the $L^2(0,T; H^3)$-regularity of $u^i$ from Corollary \ref{CorollaryH3regularityWeakSolutionTJ} we have that $\widetilde{\mathfrak{f}}^i\in L^2(0,T; L^2)$. To get a formulation for which we can get unique existence of a weak solution we need to do a split again. Therefore, we write the equation as
\begin{align*}
\partial_t \widetilde{u}^i&=\Dasti \widetilde{v}^i-C_u\widetilde{v}^i+\widetilde{\mathfrak{f}}^i,\\
\widetilde{v}^i           &=-\Dasti\widetilde{u}^i+C_v\widetilde{u}^i.
\end{align*}
From (\ref{EquationIdentitiesfortildeuTripleJunction}) we deduce for the first order boundary conditions for $\widetilde{u}^i$ that
\begin{align}
\partialnuasta{i}\widetilde{u}^i-\partialnuasta{j}\widetilde{u}^j=\eta_{\Sigmaast}(\partialnuasta{i}u^i-\partialnuasta{j}u^j)=0,
\end{align}
where we used that $\bu$ solves in a weak sense $\eqref{EquationLinearisedSDFTJWeakPart}_4$ resp. $\eqref{EquationLinearisedSDFTJWeakPart}_5$. Also, $\widetilde{\bu}$ inherits the sum condition $\eqref{EquationLinearisedSDFTJWeakPart}_3$ from $\bu$ as the $\eta^i$ coincide on $Q_4^{\Sigmaast}$.
Hence, we get the boundary conditions 
\begin{align*}
\gamma^1\widetilde{u}^1+\gamma^2\widetilde{u}^2+\gamma^3\widetilde{u}^3&=0 & &\text{on }C_4\times [0,T],\\
\partialnuasta{1}\widetilde{u}^1-\partialnuasta{2}\widetilde{u}^2&=0 & &\text{on }C_4\times[0,T],\\
\partialnuasta{2}\widetilde{u}^2-\partialnuasta{3}\widetilde{u}^3&=0 & &\text{on }C_4\times[0,T].
\end{align*}
Now we have to determine the boundary conditions fulfilled by $\widetilde{\bv}$. First we note that on $C_4\times[0,T]$ it holds that
\begin{align*}
\sum_{i=1}^3\gamma^i\widetilde{v}^i&=\sum_{i=1}^3\gamma^i(-\Dast\widetilde{u}^i+C_v\widetilde{u}^i)\\
&=\sum_{i=1}^3\gamma^i\left(-\eta^i\Dast u^i+\eta^i C_vu^i-2\langle\gast\eta^i,\gast u^i\rangle-(\Dast\eta^i)u^i\right)\\
&=\eta^1\sum_{i=1}^3\gamma^iv^i-\Dast\eta^1\sum_{i=1}^3\gamma^iu^i-2\left\langle\gast\eta^1, \sum_{i=1}^3\gamma^i\gast u^i\right\rangle=0.
\end{align*}
Here, we used in the third identity that the $\eta^i$, $\gast\eta^i$ and $\Dast\eta^i$ equal on $C_4$. For the last identity observe that the first two summands vanish due to the boundary conditions for $\boldsymbol{u}$ and $\boldsymbol{v}$. For the last summand we note that due to \eqref{EquationParmaetrizationnearboundary} we get that
\begin{align}
\gast u^i=\nabla_{\Sigmaast}u^i+(\partialnuasta{i}u^i)\nu_{\ast}^i.
\end{align}
Furthermore, $\sum_{i=1}^3\gamma^iu^i=0$ implies by differentiating also $\sum_{i=1}^3\gamma^i\nabla_{C_4} u^i=0$. Thus, using the boundary conditions for $\boldsymbol{u}$ and the angle conditions for the reference surface we calculate
\begin{align}
\sum_{i=1}^3\gamma^i\gast u^i=\sum_{i=1}^3\gamma^i\nabla_{C_4}u^i+\sum_{i=1}^3\gamma^i(\partialnuasta{i}u^i)\nuast^i=\partialnuasta{1}u^1\sum_{i=1}^3\gamma^i\nuast^i=0.
\end{align}
In total, this proves
\begin{align}
\sum_{i=1}^3\gamma^i\widetilde{v}^i&=0 & &\text{on }C_4\times[0,T].
\end{align}
It remains to study the Neumann-type boundary conditions for $\widetilde{v}^i$. For this observe that \begin{align}
\widetilde{v}^i=\eta^i v^i-(\Dast\eta^i)u^i-2\langle\gast\eta^i, \gast u^i\rangle,
\end{align}
and so
\begin{align*}
\partialnuasta{i}\widetilde{v}^i=(\underbrace{\partialnuasta{i}\eta^i}_{=0})v^i+\eta^i(\partialnuasta{i}v^i)-(\partialnuasta{i}\Dast\eta^i)u^i+(\Dast \eta^i)\partialnuasta{i}u^i
-2\langle\partialnuasta{i}\gast\eta^i, \gast u^i\rangle-2\langle\gast\eta^i, \partialnuasta{i}\gast u^i\rangle.
\end{align*}
With this, we get on $C_4$ the boundary conditions
\begin{align}
\partialnuasta{1}\widetilde{v}^1-\partialnuasta{2}\widetilde{v}^2&=\widetilde{\mathfrak{b}}^5, & &\text{on }C_4\times[0,T],\\
\partialnuasta{2}\widetilde{v}^2-\partialnuasta{3}\widetilde{v}^3&=\widetilde{\mathfrak{b}}^6, & &\text{on }C_4\times[0,T],
\end{align}
with
\begin{align*}
\widetilde{\mathfrak{b}}^5&=\eta\mathfrak{b}^5-(\partialnuasta{1}\Dast \eta^1)u^1-2\langle\partialnuasta{1}\gast\eta^1, \gast u^1\rangle-2\langle\gast\eta^1, \partialnuasta{1}\gasta{1}u^1\rangle\\
&+(\partialnuasta{2}\Dast\eta^2)u^2+2\langle \partialnuasta{2}\gast\eta^2,\gast u^2\rangle+2\langle\gast\eta^2,\partialnuasta{2}\gast u^2\rangle,\\
\widetilde{\mathfrak{b}}^6&=\eta\mathfrak{b}^6-(\partialnuasta{2}\Dast \eta^2)u^2-2\langle\partialnuasta{2}\gast\eta^2, \gast u^2\rangle-2\langle\gast\eta^2, \partialnuasta{2}\gast u^2\rangle\\
&+(\partialnuasta{3}\Dasta{3}\eta^3)u^3+2\langle \partialnuasta{3}\gast\eta^3,\gast u^3\rangle+2\langle\gast\eta^3, \partialnuasta{3}\gast u^3\rangle.
\end{align*}
Here, we used that
\begin{align*}
\left(\Dast\eta\right)(\partialnuasta{1}u^1-\partialnuasta{2}u^2)=\left(\Dast\eta\right)(\partialnuasta{2}u^2-\partialnuasta{3}u^3)=0 & &\text{on }C_4.
\end{align*}
Therefore, we get the following problem for $\widetilde{\boldsymbol{u}}$:
\begin{align}
\partial_t \widetilde{u}^i-\Dast \widetilde{v}^i+C_u\widetilde{v}^i&=\widetilde{\mathfrak{f}}^i & & \text{on }Q_4^i\times[0,T], &&i=1,2,3,\notag\\
\widetilde{v}^i+\Dast\widetilde{u}^i-C_v\widetilde{u}^i&=0 & &\text{on }Q_4^i\times[0,T],  &&i=1,2,3,\notag\\
\gamma^1\widetilde{u}^1+\gamma^2\widetilde{u}^2+\gamma^3\widetilde{u}^3&=0 & &\text{on }C_4\times [0,T],\notag\\
\partialnuasta{1}\widetilde{u}^1-\partialnuasta{2}\widetilde{u}^2&=0 & &\text{on }C_4\times[0,T],\notag\\
\partialnuasta{2}\widetilde{u}^2-\partialnuasta{3}\widetilde{u}^3&=0 & &\text{on }C_4\times[0,T], \notag\\
\gamma^1\widetilde{v}^1+\gamma^2\widetilde{v}^2+\gamma^3\widetilde{v}^3&=0 & &\text{on }C_4\times[0,T],\label{EquationSystemforLoacalSchauderEstimates}\\
\partialnuasta{1}\widetilde{v}^1-\partialnuasta{2}\widetilde{v}^2&=\widetilde{\mathfrak{b}}^5 & &\text{on }C_4\times[0,T],\notag\\
\partialnuasta{2}\widetilde{v}^2-\partialnuasta{3}\widetilde{v}^3&=\widetilde{\mathfrak{b}}^6 & &\text{on }C_4\times[0,T],\notag\\
\widetilde{u}^i&=0 & &\text{on }S_4\times[0,T], &&i=1,2,3,\notag\\
\widetilde{v}^i&=0 & &\text{on }S_4\times[0,T], &&i=1,2,3,\notag\\
\widetilde{u}^i&=0 & &\text{on }Q_4^i\times\{0\}, &&i=1,2,3,\notag
\end{align}
with the inhomogeneities $\widetilde{\mathfrak{f}}^i, \widetilde{\mathfrak{b}}^5$ and $\widetilde{\mathfrak{b}}^6$ chosen as above. Note hereby that due to the chosen support of $\boldsymbol{\eta}$ we have
\begin{align}\label{EquationRegularitatftilde}
\widetilde{\mathfrak{f}}^i\big|_{Q_3^i\times[0,T]}&=\mathfrak{f}^i\in C^{\alpha,\frac{\alpha}{4}}(Q_3^i\times[0,T]),& i&=1,2,3,\\
\widetilde{\mathfrak{b}}^i\big|_{C_3^i\times[0,T]}&= \mathfrak{b}^i\in C^{1+\alpha, (1+\alpha)/4}(C_3\times[0,T]),& i&=5,6,\notag
\end{align}
and furthermore
\begin{align}
\widetilde{\mathfrak{f}}^i\in L^2(Q_4\times[0,T]), \widetilde{\mathfrak{b}}^5\in L^2(C_4\times[0,T]), \widetilde{\mathfrak{b}}^6\in L^2(C_4\times[0,T]),
\end{align}
and additionally we have the estimates
\begin{align}
\|\widetilde{\mathfrak{f}}^i\|_{L^2(Q_4^i\times[0,T])}&\le C\left(\|\mathfrak{f}^i\|_{L^2(Q_4^i\times[0,T])}+\|u^i\|_{L^2(0,T; H^3(Q_4))}\right),\notag\\
\|\widetilde{\mathfrak{b}}^5\|_{L^2(C_4\times[0,T])}&\le C\left(\|\mathfrak{b}^5\|_{L^2(C_4\times[0,T])}+\sum_{i=1}^3\|u^i\|_{L^2(0,T; H^3(Q_4^i))}\right),\label{EquationAbschatzungftilgegegenf}\\
\|\widetilde{\mathfrak{b}}^6\|_{L^2(C_4\times[0,T])}&\le C\left(\|\mathfrak{b}^5\|_{L^2(C_4\times[0,T])}+\sum_{i=1}^3\|u^i\|_{L^2(0,T; H^3(Q_4^i))}\right),\notag
\end{align} 
as all terms in $\widetilde{\mathfrak{f}}^i$ except for $\mathfrak{f}^i$ are products of derivatives of $\eta^i$ and derivatives of $u^i$ of order not larger than three and all terms in $\widetilde{\mathfrak{b}}^i$ except for $\mathfrak{b}^i$ are products of derivatives of $\eta$ and derivatives of $u^i$ of order not larger than $2$ and the latter are due to compactness of the trace operator bounded by the $H^3$-norm of $\boldsymbol{u}$.\\
We observe that $\widetilde{\mathfrak{b}}^1\equiv\widetilde{\mathfrak{b}}^2\equiv\widetilde{\mathfrak{b}}^3\equiv\widetilde{\mathfrak{b}}^4\equiv 0$ and so we can derive the same existence results for \eqref{EquationSystemforLoacalSchauderEstimates} as in Proposition \ref{PropostionsExistenceofWeakSolutionsSurfaceDiffsuionTripleJunction}. The only difference is that we include the boundary conditions on $S_4$ directly in the solution space and thus we do not want to repeat the procedure. Also, by construction \begin{align*}
\widetilde{\boldsymbol{u}}=\boldsymbol{\eta u},\quad \widetilde{\boldsymbol{v}}=\boldsymbol{\eta v}-(\Dast\boldsymbol{\eta})\boldsymbol{u}-2\langle\gast\boldsymbol{\eta},\gast\boldsymbol{u}\rangle,
\end{align*}
is a weak and therefore the unique weak solution of \eqref{EquationSystemforLoacalSchauderEstimates}. \\
It remains to show Schauder estimates on suitable subsets. On the latter we want to apply Proposition \ref{PropositionSchauderestimatesforlocalizedSurfaceDiffusionnearTripleJunction}.
As the $\tilde{\mathfrak{f}}^i, \tilde{\mathfrak{b}}^5, \tilde{\mathfrak{b}}^6$ only have $L^2$-regularity, we take sequences $$\left(\widetilde{\mathfrak{f}}^i_n\right)_{n\in\N}\subset C^{\alpha,\alpha/4}(Q^i_4\times[0,T]),\quad \left(\widetilde{\mathfrak{b}}^5_n\right)_{n\in\N}, \left(\widetilde{\mathfrak{b}}^6_n\right)_{n\in\N}\subset C^{1+\alpha,\frac{1+\alpha}{4}}(C_4\times[0,T])$$ fulfilling
\begin{align}\label{EquationPropertiesofsmoothapproximationofinhomogenitiesforSchauderestimatesTripkeJunctioneins}
\|\widetilde{\mathfrak{f}}^i_n-\widetilde{\mathfrak{f}}^i\|_{L^2(Q_4^i\times[0,T])}\to 0,\quad \|\widetilde{\mathfrak{b}}^5_n-\widetilde{\mathfrak{b}}^6\|_{L^2(C_4\times[0,T])}\to 0,\quad \|\widetilde{\mathfrak{b}}^6_n-\widetilde{\mathfrak{b}}^6\|_{L^2(C_4\times[0,T])}\to 0,
\end{align}
and additionally we need on subsets $Q_2^i\subset Q_3^i$ as $n\to\infty$ that
\begin{align}
\|\widetilde{\mathfrak{f}}^i_n\|_{C^{\alpha,\frac{\alpha}{4}}(Q_2^i\times[0,T])}&\le \|\widetilde{\mathfrak{f}}^i\|_{C^{\alpha,\frac{\alpha}{4}}(Q_3^i\times[0,T])}=\|\mathfrak{f}^i\|_{C^{\alpha,\frac{\alpha}{4}}(Q_3^i\times[0,T])},\notag\\
\|\widetilde{\mathfrak{b}}^5_n\|_{C^{1+\alpha,(1+\alpha)/4}(C_2\times[0,T])}&\le \|\widetilde{\mathfrak{b}}^5\|_{C^{1+\alpha, (1+\alpha)/4}(C_3\times[0,T])}=\|\mathfrak{b}^5\|_{C^{1+\alpha, (1+\alpha)/4}(C_3\times[0,T])},\label{EquationPropertiesofsmoothapproximationofinhomogenitiesforSchauderestimatesTripkeJunctionzwei}\\
\|\widetilde{\mathfrak{b}}^6_n\|_{C^{1+\alpha,(1+\alpha)/4}(C_2\times[0,T])}&\le \|\widetilde{\mathfrak{b}}^6\|_{C^{1+\alpha, (1+\alpha)/4}(C_3\times[0,T])}=\|\mathfrak{b}^6\|_{C^{1+\alpha, (1+\alpha)/4}(C_3\times[0,T])}. \notag
\end{align}
One possibility to guarantee condition $(\ref{EquationPropertiesofsmoothapproximationofinhomogenitiesforSchauderestimatesTripkeJunctionzwei})_1$ is to choose one $C^{\alpha,\alpha/2}$-approximation of $\widetilde{\mathfrak{f}}^i_n$ on $Q_4\backslash Q_3$ and take $\mathfrak{f}^i$, which equals $\widetilde{\mathfrak{f}}^i$ on $Q_3$, as $C^{\alpha,\frac{\alpha}{4}}$-approximation on $Q_3$. We then get a suitable approximation on $Q_4$ by connecting both approximations via partitions of unity. The same procedure can be done for $\widetilde{\mathfrak{b}}^5$ and $\widetilde{\mathfrak{b}}^6$.\\ 
Now, we call problem (\ref{EquationSystemforLoacalSchauderEstimates}) with the approximating inhomogeneities $\widetilde{\mathfrak{f}}^i_n, \widetilde{\mathfrak{b}}^5_n$ and $\widetilde{\mathfrak{b}}^6_n$ problem $(\ref{EquationSystemforLoacalSchauderEstimates})_n$. The parametrization of this problem has a unique solution $\widetilde{\boldsymbol{u}}_n\in C^{4+\alpha, 1+\frac{\alpha}{4}}_{TJ}(Q_4\times[0,T])$ due to Proposition \ref{PropositionSchauderestimatesforlocalizedSurfaceDiffusionnearTripleJunction}. On the other hand, $(\boldsymbol{\widetilde{u}}_n, \widetilde{\boldsymbol{v}}_n)$ with $\widetilde{v}_n^i=-\Dasti\widetilde{u}^i_n+C_v\widetilde{u}^i_n$ is also a weak solution of $(\ref{EquationSystemforLoacalSchauderEstimates})_n$ and so using the estimate (\ref{EquationEnhacedRegularityforweaksolutionTripleJunction}) we get
\begin{align*}
\sum_{i=1}^3\|\widetilde{u}_n^i\|_{L^2(0,T; H^{3}(Q_4^i))}&\le C \left(\sum_{i=1}^3\|\widetilde{\mathfrak{f}}^i_n\|_{L^2(0,T; L^2(Q_4^i))}+\sum_{i=5}^6\|\widetilde{\mathfrak{b}}_n^i\|_{L^2(0,T; L^2(C_4))}\right)\\
&\le C\left(\sum_{i=1}^3\|\widetilde{\mathfrak{f}}^i\|_{L^2(Q_4^i\times[0,T])}+\sum_{i=5}^6\|\widetilde{\mathfrak{b}}^i\|_{L^2(C_4\times[0,T])} \right)\\
&\le C\left(\sum_{i=1}^3\|\mathfrak{f}^i\|_{L^2(Q_4^i\times[0,T])}+\sum_{i=5}^6\|\mathfrak{b}^i\|_{L^2(C_4\times[0,T])}+\sum_{i=1}^3\|u^i\|_{L^2(0,T; H^3(Q_4^i))}\right)\\
&\le C\left(\sum_{i=1}^3\|\mathfrak{f}^i\|_{L^2(Q_4\times[0,T])}+\sum_{i=5}^6\|\mathfrak{b}^i\|_{L^2(C_4\times[0,T])}\right).
\end{align*}
In the third inequality we used (\ref{EquationAbschatzungftilgegegenf}) and  (\ref{EquationEnhacedRegularityforweaksolutionTripleJunction}) in the last inequality. This yields now that both $\boldsymbol{u}_n$ and $\boldsymbol{v}_n$ are bounded in $L^2(0,T; H^1_{TJ}(Q_4))$ and so we can find subsequences $(\boldsymbol{\widetilde{u}}_{n_l})_{n_l}$ and $(\widetilde{\boldsymbol{v}}_{n_l})_{n_l}$ converging weakly to $\bar{\boldsymbol{u}}, \bar{\boldsymbol{v}}\in L^2(0,T; H^1_{TJ}(Q_4))$, which forms a weak solution of (\ref{EquationSystemforLoacalSchauderEstimates}). Due to uniqueness of the weak solution it follows
\begin{align*}
\bar{\boldsymbol{u}}=\widetilde{\boldsymbol{u}}& &\text{in }Q_4\times[0,T].
\end{align*} 
Now we need bounds for the H\"older-norms to get weak convergence in that space, too. For this, we use the local estimates from \cite[Theorem 4.11]{ladyzhenskaia1968linear} for $Q_1\subset Q_2$ to derive
\begin{align}
\sum_{i=1}^3\|\widetilde{u}^i_n\|_{C^{4+\alpha, 1+\frac{\alpha}{4}}(Q_1^i\times[0,T])}&\le C\left(\sum_{i=1}^3\|\widetilde{\mathfrak{f}}^i_n\|_{C^{\alpha,\frac{\alpha}{4}}(Q_2^i\times [0,T])}+\sum_{i=5}^6\|\widetilde{\mathfrak{b}}_n^i\|_{C^{1+\alpha, (1+\alpha)/4}(C_2\times[0,T])}\right)\notag\\
&\phantom{\le}+C\sum_{i=1}^3\|\widetilde{u}_n^i\|_{L^2(Q_2^i\times[0,T])}\notag\\
&\overset{(\ref{EquationPropertiesofsmoothapproximationofinhomogenitiesforSchauderestimatesTripkeJunctionzwei})}{\le} C\left(\sum_{i=1}^3\|\widetilde{\mathfrak{f}}^i\|_{C^{\alpha,\frac{\alpha}{4}}(Q_3^i\times [0,T])}+\sum_{i=5}^6\|\widetilde{\mathfrak{b}}^i\|_{C^{1+\alpha, (1+\alpha)/4}(C_3\times[0,T])}\right)\notag\\
&\phantom{=}+C\sum_{i=1}^3\|\widetilde{u}_n^i\|_{L^2(Q_4^i\times[0,T])}\label{EquationHolderEstimatesforLimit}\\
&\le C\left(\sum_{i=1}^3\|\mathfrak{f}^i\|_{C^{\alpha,\frac{\alpha}{4}}(Q_3^i\times [0,T])}+\sum_{i=5}^6\|\mathfrak{b}^i\|_{C^{1+\alpha, (1+\alpha)/4}(C_3\times[0,T])}\right)\notag\\
&\phantom{=}+C\left(\sum_{i=1}^3\|\mathfrak{f}^i\|_{L^2(Q_4^i\times[0,T])}+\sum_{i=1}^5\|\mathfrak{b}^i\|_{L^2(C_4\times[0,T])}\right)\notag\\
&\le C\left(\sum_{i=1}^3\|\mathfrak{f}^i\|_{C^{\alpha,\frac{\alpha}{4}}(Q_4^i\times[0,T])}+\sum_{i=5}^6\|\mathfrak{b}^i\|_{C^{1+\alpha, (1+\alpha)/4}(C_4\times[0,T])}\right)\notag.
\end{align}
But this implies now that the subsequence $\widetilde{\boldsymbol{u}}_{n_l}$ constructed above is bounded in $C^{4,1}(Q_1\times[0,T])$ and as the H\"older-norms are bounded, it is as well equicontinuous. Thus, the theorem of Arzela-Ascoli applied on every derivative gives us the existence of a subsequence - we again call $\widetilde{\boldsymbol{u}}_{n_l}$ - that converges to some $\widehat{\boldsymbol{u}}\in C^{4,1}_{TJ}(Q_1\times[0,T])$. Indeed, we also have $\widehat{\boldsymbol{u}}\in C^{4+\alpha, 1+\alpha}_{TJ}(Q_1\times[0,T])$ as for any covariant derivative $\nabla^k$ of up to order $4$ in space we have
\begin{align}
|\nabla^k\widehat{u}^i(x,t)-\nabla^k\widehat{u}^i(y,t)|=\underset{n_l\to\infty}{\lim}|\nabla^k\widetilde{u}^i_{n_l}(x,t)-\nabla^k\widetilde{u}^i_{n_l}(y,t)|\le C|x-y|^{\alpha},
\end{align}
for any points $(x,t), (y,t)\in Q_4\times[0,T]$ and we also have
\begin{align}
|\partial_t \widehat{u}^i(x,t_1)-\partial_t\widehat{u}^i(x,t_2)|=\underset{n_l\to\infty}{\lim}|\widetilde{u}^i_{n_l}(x,t_1)-\widetilde{u}^i_{n_l}(x,t_2)|\le C|t_1-t_2|^{\frac{\alpha}{4}}, 
\end{align}
for any points $(x,t_1),(x,t_2)\in Q_4\times[0,T]$. The same argument can used for the H\"older regularity in time of the partial derivatives in space. Also, observe that we have for $\widehat{u}$ the estimate \eqref{EquationHolderEstimatesforLimit}. Uniqueness of limits now implies
\begin{align}
\boldsymbol{\widehat{u}}=\widetilde{\boldsymbol{u}}& &\text{on }Q_1\times [0,T].
\end{align}
As we have $\widehat{u}=u$ on $Q_1\times[0,T]$ this implies now the desired Schauder-estimate for $\boldsymbol{u}$, namely
\begin{align}
\|\boldsymbol{u}\|_{C^{4+\alpha, 1+\frac{\alpha}{4}}_{TJ}(Q_1\times[0,T])}\le C\left(\sum_{i=1}^3\|\mathfrak{f}^i\|_{C^{\alpha,\frac{\alpha}{4}}(Q_1\times[0,T])}+\sum_{i=5}^6\|\mathfrak{b}^i\|_{C^{1+\alpha, (1+\alpha)/4}(Q_1\times[0,T])}\right),
\end{align}
holds. With this done we sum up our results for \eqref{EquationReducedLinearProblem} in the following Proposition.
\begin{proposition}[Existence theory for \eqref{EquationReducedLinearProblem}]\label{PropositionExistenceTheoryReducedLinearProbelm}\ \\
	Problem $\eqref{EquationReducedLinearProblem}$ has for all $T>0$ and all inhomogeneities
	\begin{align}
	\mathfrak{f}\in C^{\alpha,\frac{\alpha}{4}}_{TJ}(\GammaastT),\quad \mathfrak{b}^5,\mathfrak{b}^6\in C^{1+\alpha,\frac{1+\alpha}{4}}(\Sigmaast\times[0,T]),
	\end{align}
	which fulfill the compatibility conditions
	\begin{align}
	(\gamma^1\mathfrak{f}^1+\gamma^2\mathfrak{f}^3+\gamma^3\mathfrak{f}^3)\big|_{t=0}=0\text{ on }\Sigmaast,\quad\quad  \mathfrak{b}^i\big|_{t=0}=0\text{ on }\Sigma_{\ast}, i=5,6,
	\end{align}
	a unique solution in $C^{4+\alpha,\frac{1+\alpha}{4}}_{TJ}(\GammaastT)$ and we have the energy estimate 
	\begin{align}
	\|\boldsymbol{u}\|_{C^{4+\alpha, 1+\frac{\alpha}{4}}_{TJ}(\Gamma_{\ast,T})}\le C\left(\|\mathfrak{f}\|_{C^{\alpha,\frac{\alpha}{4}}_{TJ}(\Gamma_{\ast,T})}+\sum_{i=5}^6\|\mathfrak{b}^i\|_{C^{1+\alpha, \frac{1+\alpha}{4}}(\Sigma_{\ast,T})}\right)
	\end{align}
\end{proposition}

\subsection{The Complete Linear Problem}\label{SubsectionCompleteLinearProblem}
\label{SubsectionTheCompleteLinearProblem}
The system we studied so far differs in three aspects from the original problem \eqref{EquationLinearisedSurfaceDiffusionTripleJunctionAbstractFormulaton}. Firstly, we miss inhomogeneities in the first four boundary conditions. Secondly, both in the parabolic equation itself and the boundary conditions lower order terms are missing. And finally, we have to include general initial data. These final steps are carried out in the next three corollaries.

\begin{corollary}[Inhomogeneities in \eqref{EquationReducedLinearProblem}]\label{CorollaryLinearisedAnalysisFullInhomo}\ \\
	Let $\sigma_i$ be as in \eqref{EquationDefinitionofSigma}. For any 
	\begin{align} 
	\mathfrak{b}^i&\in C^{4-\sigma_i+\alpha, \frac{4-\sigma_i+\alpha}{4}}(\Sigma_{\ast,T}), & &i=1,...,6,
	\end{align}
	fulfilling the compatibility conditions
	\begin{align}
	(\gamma^1\mathfrak{f}^1+\gamma^2\mathfrak{f}^3+\gamma^3\mathfrak{f}^3)\big|_{t=0}=0\text{ on }\Sigmaast\quad  \mathfrak{b}^i\big|_{t=0}=0\text{ on }\Sigma_{\ast}, i=1,...,6,
	\end{align}
	 the system \eqref{EquationReducedLinearProblem}	has a unique solution $\boldsymbol{u}\in C^{4+\alpha, 1+\frac{\alpha}{4}}_{TJ}(\Gamma_{\ast}\times[0,T])$ and we have the energy estimate 
	\begin{align}\label{EquationEnergieabschatzungLinearisiertesProblemmitallenInhomogenitaten}
	\|\boldsymbol{u}\|_{C^{4+\alpha, 1+\frac{\alpha}{4}}_{TJ}(\Gamma_{\ast,T})}\le C\left(\sum_{i=1}^3\|\mathfrak{f}^i\|_{C^{\alpha, \frac{\alpha}{4}}(\Gamma_{\ast, T}^i)}+\sum_{i=1}^6\|\mathfrak{b}^i\|_{C^{4-\sigma_i+\alpha, \frac{4-\sigma_i+\alpha}{4}}(\Sigma_{\ast,T})}\right).
	\end{align}

\end{corollary}
\begin{proof}
	We want to include the missing inhomogeneities one by one, beginning with $\mathfrak{b}^4$. For this we do a shifting procedure where we have to guarantee that the shift will not results in lower order inhomogeneities. Hereby, we follow ideas of \cite[Appendix B]{garckeitokohsaka2008nonlinear}. For any $\mathfrak{b}^4\in C^{2+\alpha, (2+\alpha)/4}(\Sigma_{\ast,T})$ the system 
	\begin{align*}
	\partial_t \bar{b}+\Dasta{2}\Dasta{2}\bar{b}+\bar{b}&=0 & &\text{on }\Gastzwei\times[0,T],\\
	\bar{b}&=0 & &\text{on }\Sigma_{\ast}\times[0,T], \\
	-\Dasta{2}\bar{b}&=\mathfrak{b}^4 & &\text{on }\Sigma_{\ast}\times[0,T],\\
	\bar{b}(x,0)&=0 & & \text{on }\Gamma_{\ast}^2,
	\end{align*}
	has a unique solution $\bar{b}\in C^{4+\alpha, 1+\frac{\alpha}{4}}(\Gamma_{\ast}^2)$ fulfilling the energy estimate
	\begin{align}\label{EquationEnergieabschatzunoverlineb}
	\|\bar{b}\|_{C^{4+\alpha, 1+\frac{\alpha}{4}}(\Gamma_{\ast}^2)}\le C\|\mathfrak{b}^4\|_{C^{3+\alpha, (3+\alpha)/4}(\Sigma_{\ast})}.
	\end{align}
	This can be proven with the same strategy as in Section \ref{SubsectionWeakTheory} and \ref{SubsectionSchauderEstimates}, except that we do not have to split the parabolic equation and so we can include the inhomogeneity $\mathfrak{b}^4$. Now, we define the auxiliary function $\chi^2: \Gamma_{\ast}^2\to\mb{R}$ by
	\begin{align*}
	\chi^2(x)=\frac{1}{2}(\dist_{\Sigma_{\ast}}(x))^2\eta(\dist_{\Sigma_{\ast}}(x))\bar{b}(x),
	\end{align*}
	where $\eta: [0,\infty]\to [0,1]$ is again a suitable cut-off function with $\eta\equiv 1$ on $[0,\varepsilon]$ for some sufficiently small $\varepsilon$. Note that $\dist_{\Sigma_{\ast}}^2(\cdot)\eta\in C^{5+\alpha, \infty}(\Gamma_{\ast,T}^2)$ and so (\ref{EquationEnergieabschatzunoverlineb}) holds also for $\chi^2$. Define finally $\boldsymbol{\chi}$ by setting $\chi^1\equiv 0$ and $\chi^3\equiv 0$ on $\Gamma_{\ast}^1$ resp. $\Gamma_{\ast}^3$. Considering the solution $\bu$ of \eqref{EquationReducedLinearProblem} constructed in Proposition \ref{PropositionExistenceTheoryReducedLinearProbelm} with 
	\begin{align*}
	\widetilde{\mathfrak{f}}^1&=\mathfrak{f}^1-\partial_t\chi-\Dasta{1}\Dasta{1}\chi+(C_u+C_v)\Dasta{1}\chi+C_uC_v\chi,\\
	\widetilde{\mathfrak{b}}^5&=\mathfrak{b}^5+\partial_{\nu_{\ast}^1}(-\Dasta{1}\chi),\\
	\widetilde{\mathfrak{b}}^6&=\mathfrak{b}^6,
	\end{align*}
	the function $\boldsymbol{u}+\chi$ is the wished solution of (\ref{EquationLinearisedSDFTJWeakPart}) with included inhomogeneity $\mathfrak{b}^4$ and (\ref{EquationEnergieabschatzunoverlineb}) implies together with the already known energy estimates the estimate (\ref{EquationEnergieabschatzungLinearisiertesProblemmitallenInhomogenitaten}). 
	To include the inhomogeneities $\mathfrak{b}^1, \mathfrak{b}^2$ and $\mathfrak{b}^3$ we can argue analogously.
\end{proof}
\begin{corollary}[Lower order terms for \eqref{EquationReducedLinearProblem}] \label{CorollaryLOTforLinearizedProbelm}\ \\
	Consider \eqref{EquationLinearisedSurfaceDiffusionTripleJunctionAbstractFormulaton} with $\bu_0\equiv 0$. Then there is a $T_{\ast}>0$ such that for all $T<T_{\ast}$ the result of Corollary \ref{CorollaryLinearisedAnalysisFullInhomo} also holds for this problem.
\end{corollary}
\begin{proof}
	We consider \eqref{EquationLinearisedSurfaceDiffusionTripleJunctionAbstractFormulaton} as a perturbation of \eqref{EquationReducedLinearProblem}. That is, given any 
	\begin{align*}\boldsymbol{u}\in X:=\left\{\boldsymbol{u}\in C^{4+\alpha, 1+\frac{\alpha}{4}}_{TJ}(\Gamma_{\ast,T})\Bigg|\ \boldsymbol{u}\big|_{t=0}\equiv 0\right\}
	\end{align*}
	we want to solve (\ref{EquationReducedLinearProblem}) with the inhomogeneities
	\begin{align*}
	\widetilde{\mathfrak{f}}^i(\boldsymbol{u})=\mathfrak{f}^i-\widetilde{\mathcal{A}}^i_{LOT}u^i-\bar{\mathcal{A}}^i_{LOT}\trsigma{\bu},\quad\widetilde{\mathfrak{b}}^i(\boldsymbol{v})=\mathfrak{b}^i-\mathcal{B}_{LOT}^i\boldsymbol{u}.
	\end{align*}
	Note that $\widetilde{\mathfrak{f}}^i$ and $\widetilde{\mathfrak{b}}^i$ fulfill the compatibility conditions (CCP) as $\mathcal{A}^i_{LOT}$ and $\mathcal{B}^i_{LOT}$ vanish on $X$ at $t=0$. This implies that the solution $\Lambda(\boldsymbol{u})$ exists due to Corollary \ref{CorollaryLinearisedAnalysisFullInhomo}. We claim that for $T$ small enough the map $\Lambda: X\to X$ is a contraction mapping. Then, Banach's fixed point theorem gives us the unique solution. \\
	For $\boldsymbol{u},\boldsymbol{v}\in X$ we note that the difference $\Lambda(\boldsymbol{u})-\Lambda(\boldsymbol{v})$ solves (\ref{EquationReducedLinearProblem}) with inhomogeneities $$\widetilde{\mathfrak{f}}^i=\widetilde{\mathcal{A}}_{LOT}^i(v^i-u^i)+\bar{\mathcal{A}}^i_{LOT}(\trsigma{\bv}-\trsigma{\bu}), \quad \widetilde{\mathfrak{b}}^i=\mathcal{B}^i_{LOT}(\boldsymbol{v}-\boldsymbol{u}),$$ and due to estimate (\ref{EquationEnergieabschatzungLinearisiertesProblemmitallenInhomogenitaten}) we have
	\begin{align*}
	\|\Lambda(\boldsymbol{u})-\Lambda(\boldsymbol{v})\|_X&\le C_1\sum_{i=1}^3\|\mathcal{A}^i_{LOT}(\boldsymbol{v}-\boldsymbol{u},\trsigma{(\bv-\bu)})\|_{C^{\alpha,\frac{\alpha}{4}}(\Gamma^i_{\ast,T})}\\ &+C_2\sum_{i=1}^6\|\mathcal{B}^i_{LOT}(\boldsymbol{v}-\boldsymbol{u})\|_{C^{4-\sigma_i+\alpha,\frac{4-\sigma_i+\alpha}{4}}(\Sigma_{\ast,T})}.
	\end{align*}
	There are two different kind of terms that appear in the perturbation operators. One are lower order partial derivatives of $\boldsymbol{u-v}$ with $C^{\alpha}$-coefficients only depending on the reference geometry, which are given by $\widetilde{\mathcal{A}}^i_{LOT}(v^i-u^i)$ and $\mathcal{B}^i_{LOT}(\bv-\bu)$. For these we may directly apply Lemma \ref{LemmaContractivityLowerOrderTerms} to get the sought contractivity property if $T$ is sufficiently small. Note that the regularity of the coefficient functions is not a problem due to Lemma \ref{LemmaProductEstimatesParabolichHspace}. The other terms are the non-local terms $\bar{\mathcal{A}}^i_{LOT}(\trsigma{(\bv-\bu)}$. But here we see that they are also of lower order (second order terms are the highest order arising) and by the chain rule all space and time derivatives are bounded by space and time derivatives on the boundary and so by the H\"older-norm of $\boldsymbol{u-v}$ itself. So, this shows us that for $T$ sufficiently small $\Lambda$ is a $\frac{1}{2}$-contraction, which finishes the proof.  
\end{proof}
\begin{corollary}[General Initial Data for \eqref{EquationReducedLinearProblem}]
	The result of Corollary \ref{CorollaryLOTforLinearizedProbelm} stays true for all initial data $\bu_0$ fulfilling the compatibility conditions \eqref{EquationCompatibilityConditionsLinearisedTJ}, whereby the energy estimate is replaced by \eqref{EquationEnergieabschatzungenLineareGleichungVollTJ}. In particular, this proves Theorem \ref{theoremShorttimeexistenceLinearisedSDFTJGeneralInitialdata}.
\end{corollary}
\begin{proof}
	We can do the same procedure as in \cite{depner2014mean} by shifting the equation by $\bu_0$. Condition (CLP) will guarantee that the compatibility condition for Corollary \ref{CorollaryLOTforLinearizedProbelm} are fulfilled and the additional inhomogeneities give us the sought energy estimate.
\end{proof}
\section{Non-linear Analysis}\label{SectionNonLinearAnalysis}
In this section we will use Theorem \ref{theoremShorttimeexistenceLinearisedSDFTJGeneralInitialdata} to prove short time existence for our original problem \eqref{EquationSurfaceDiffusionTripleJunctionGeometricVersiononRrenceFrame}. 
We want to note that unlike the authors of \cite{depner2014mean} we prefer not to work with the parabolic version \eqref{EquationAnalyticFormulationofSDFTJ}. By choosing the more geometric formulation \eqref{EquationSurfaceDiffusionTripleJunctionGeometricVersiononRrenceFrame} it is more convenient to exploit the quasi-linear structure of our system to derive contraction estimates. This includes also the non-local term, which indeed can be treated like the other quasi-linear terms. Actually, we will have to put the most work into the angle conditions as these are fully non-linear.\\
Our main strategy is now to write (\ref{EquationSurfaceDiffusionTripleJunctionGeometricVersiononRrenceFrame}) as a fixed-point problem which we do in the following way. For $\boldsymbol{\rho}_0\in C^{4+\alpha}_{TJ}(\Gammaast)$, $\sigma_i$ as in \eqref{EquationDefinitionofSigma}, $R,\varepsilon>0$ and $\delta>0$, which we always assume to be smaller than the existence time from Theorem \ref{theoremShorttimeexistenceLinearisedSDFTJGeneralInitialdata}, we consider the sets\footnote{Note that in $X_{R,\delta}^{\brho_0}$ we include the sum condition for $\brho$ on the parabolic boundary to guarantee compatibility conditions, which we will see in the proof of Lemma \ref{LemmaWelldefinednessofLambdaTripleJunction}.}
\begin{align*}
X_{R,\delta}^{\varepsilon}&:=\left\{\boldsymbol{\rho}\in C^{4+\alpha, 1+\frac{\alpha}{4}}_{TJ}(\Gammaastdelta)\ \Big|\ \|\boldsymbol{\rho}(0)\|_{C^{4+\alpha}(\Gammaast)}\le\varepsilon,
\|\boldsymbol{\rho}-\boldsymbol{\rho}(0)\|_{X_{R,\delta}}\le R\right\},\\
X_{R,\delta}^{\brho_0}&:=\left\{\boldsymbol{\rho}\in C^{4+\alpha, 1+\frac{\alpha}{4}}_{TJ}(\Gammaastdelta)\ \Big|\ \brho(0)=\brho_0, \sum_{i=1}^3\gamma^i\rho^i=0\text{ on }\Sigmaastdelta,
\|\boldsymbol{\rho}-\brho_0\|_{X_{R,\delta}}\le R\right\},\\
Y_{\delta}&:=C^{\alpha, \frac{\alpha}{4}}_{TJ}(\Gammaastdelta)\times\left(\prod_{i=1}^{6}C^{4-\sigma_i+\alpha, \frac{4-\sigma_i+\alpha}{4}}(\Sigmaastdelta)\right).
\end{align*}
We will denote by $\|\cdot\|_{X_{R,\delta}}$ and $\|\cdot\|_{Y_{\delta}}$ the canonical norms on $X_{R,\delta}^{\varepsilon}$ (resp. $X_{R,\delta}^{\brho_0}$) and $Y_{\delta}$. Observe that we have $\|\brho_0\|_{X_{R,\delta}}=\|\brho_0\|_{C^{4+\alpha}_{TJ}(\Gammaast)}$ and therefore 
\begin{align}\label{EquationEstimateforelementsofXRdelta}
\forall \brho\in X_{R,\delta}: \|\brho\|_{X_{R,\delta}}\le R+\|\brho_0\|_{C^{4+\alpha}_{TJ}(\Gammaast)}.
\end{align}
On these sets we consider the inhomogeneities operator $S:=(\mathfrak{f}, \mathfrak{b}): X_{R,\delta}^{\brho_0}\to Y_{\delta}$ given by
\begin{align}
\mathfrak{f}^i(\boldsymbol{\rho})&:=\partial_t\rho^i-V_{\boldsymbol{\rho}}^i+\Delta_{\brho}H_{\brho}^i-\mathcal{A}^i(\rho^i, \boldsymbol{\rho}\big|_{\Sigma_{\ast}}), & &i=1,2,3,\label{EquationDefinitionInhomeneityf}\\
\mathfrak{b}^i(\boldsymbol{\rho})&:=\mathcal{B}^i(\boldsymbol{\rho})-G^i(\boldsymbol{\rho}), & &i=1,...,6, \label{EquationInhomogeneityOperatorBoundary}
\end{align}
where we used the notation from Section \ref{SectionParametrizationCom} and \ref{SectionLinearization}. Furthermore, we define $L: Y_{\delta}\to X_{R,\delta}^{\brho_0}$ as the solution operator from Theorem \ref{theoremShorttimeexistenceLinearisedSDFTJGeneralInitialdata} and $\Lambda:=L\circ S: X_{R,\delta}^{\brho_0}\to X_{R,\delta}^{\brho_0}$. The main result of this section will now be the following.
\begin{proposition}[Existence of a fixed-point of $\Lambda$]\label{TheoremSTESurfaceDiffTripleJunction}\ \\ 
	There exists $\varepsilon_0, R_0>0$ with the following property. For all $R>R_0$ and $\varepsilon<\varepsilon_0$ there exists a $\delta>0$ such that for all $\boldsymbol{\rho}_0\in  C^{4+\alpha}(\Gamma_{\ast})$, fulfilling $\|\brho_0\|\le\varepsilon_0$ and the geometric compatibility conditions (\ref{EquationGeometricCompabilityConditionforSDFTJ}), the map $\Lambda: X_{R,\delta}^{\brho_0}\to X_{R,\delta}^{\brho_0}$ is well-defined and there exists a unique fixed-point of $\Lambda$ in $X_{R,\delta}$.
\end{proposition}
The proof splits into three main parts. We will first verify that if we choose $\varepsilon$ sufficiently small we can guarantee that $\Lambda$ is well-defined as long as $\delta(R)$ is also sufficiently small. Then, we will check that for $\varepsilon$ small and $R$ large we can find a $\delta(R,\varepsilon)$ such that $\Lambda$ is a $\frac{1}{2}$-contraction. Finally, we will see that with this choice of $\delta$ we can choose $R$ large enough such that $\Lambda$ is also a self-mapping on $X_{R,\delta}^{\brho_0}$.  
\begin{lemma}[Well-definedness of $\Lambda$]\label{LemmaWelldefinednessofLambdaTripleJunction}\ \vspace{-0,2cm}
	\begin{itemize}
		\item[i.)] There is a $\varepsilon_W>0$ such that for any $R>0$ and $\varepsilon<\varepsilon_W$ there is a $\delta_W(\varepsilon, R)>0$ such that $S$ is a well-defined map $X_{R,\delta}^{\varepsilon}\to Y_{\delta}$.\vspace{-0,2cm}
		\item[ii.)] For all initial data $\brho_0$ fulfilling the geometric compatibility condition (\ref{EquationGeometricCompabilityConditionforSDFTJ}) and the bound from i.) we have that $S(\brho)$ fulfills the linear compatibility condition $(\ref{EquationCompatibilityConditionsLinearisedTJ})$ for all $\brho\in X_{R,\delta}^{\brho_0}$.
		\vspace{-0,2cm}
		\item[iii.)] Choosing $\varepsilon,R,\delta$ and $\brho_0$ as in i.),ii.) the map $\Lambda: X_{R,\delta}^{\brho_0}\to X_{R,\delta}^{\brho_0}$ is well-defined.
	\end{itemize}
\end{lemma}
\begin{proof}
	For i.) we have to check both well-definedness of the geometric quantities in $\mathfrak{f}$ and $\mathfrak{b}$ and the correct regularity properties. For the first part recall that due to the H\"older-regularity in time for space derivatives we have for a multi-index $\beta$ with $1\le|\beta|\le 4$ and any $\brho\in X_{R,\delta}^{\varepsilon}$ with $\brho(0)=\brho_0\in C^{4+\alpha}_{TJ}(\Gammaast)$ that
	\begin{align}
	\|\partial_{\beta}^x\boldsymbol{\rho}(t)\|_{\infty}\le t^{\frac{4-|\beta|+\alpha}{4}}\langle\partial_{\beta}^x\boldsymbol{\rho}\rangle_{t, (4-|\beta|+\alpha)/4}+\|\partial_{\beta}^x\boldsymbol{\rho}(0)\|_{\infty}\le  t^{\frac{4-|\beta|+\alpha}{4}}(R+\varepsilon)+\|\brho_0\|_{C^{4+\alpha}_{TJ}(\Gammaast)}.
	\end{align}
	Additionally, we have that
	\begin{align}
	\|\brho(t)\|_{\infty}\le \delta\|\partial_t\brho\|_{\infty}+\|\brho(0)\|_{\infty}\le \delta (R+\varepsilon)+\|\brho_0\|_{C^{4+\alpha}_{TJ}(\Gammaast)}.
	\end{align}
	Together, this implies for $\delta<1$ and all $t\in[0,\delta]$ that
	\begin{align}\label{EquationBittekeineLabelmehr}
	\|\boldsymbol{\rho}(t)\|_{C^4_{TJ}(\Gamma_{\ast})}\le C\left(\delta (R+\varepsilon)+\|\brho_0\|_{C^{4+\alpha}_{TJ}(\Gammaast)}\right)\le C(\delta(R+\varepsilon)+\varepsilon).
	\end{align}	
	Now, for any $C'>0$ we get for $\varepsilon\le \frac{C'}{2C}$ and $\delta\le\frac{C'}{2C(R+\varepsilon)}$ that
	\begin{align}
	\|\brho(t)\|_{C^{4}_{TJ}(\Gammaast)}\le C'.
	\end{align}
	Thus, we can get a bound for the $C^2$-norm of $\boldsymbol{\rho}$ sufficiently small such that \cite[p. 326]{depner2014mean} implies that all geometric quantities - in particular the normal, the conormal and the inverse metric tensor - are well defined. It remains to show that these objects have the required regularity. For the normal and conormal this follows directly from writing these quantities in local coordinates as normalized crossproducts of tangent vectors. For the inverse metric tensor we observe that matrix inversion is a smooth operator $GL_n(\mb{R})\to GL_n(\mb{R})$. This follows by applying the inverse function theorem on the map
	\begin{align*}
	GL_n(\mb{R})\times GL_n(\mb{R})\to GL_n(\mb{R}), (A,B)\mapsto A\cdot B-\mathbb{E}_n.
	\end{align*}
	Applying now composition operator theory for $g^{-1}$ first as a function in time and then as function in space we get that $g^{-1}$ has the same H\"older-regularity as $g$. This implies now $S(\boldsymbol{\rho})\in Y_{\delta}$.\\
	For part ii) we get from the geometric compatibility conditions (\ref{EquationGeometricCompabilityConditionforSDFTJ})
	\begin{align}
	\sum_{i=1}^3\gamma^i\mathfrak{f}^i\big|_{t=0}=\sum_{i=1}^3\gamma^i\left(\partial_t\rho^i(0)- V^i_{\brho}(0) \right)-\sum_{i=1}^3\gamma^i\mathcal{A}^i_{all}(\rho^i_0) \quad\quad\text{on }\Sigma_{\ast}.
	\end{align}
	So, it remains to see that the first sum vanishes. But as we included the sum condition for $\brho$ on the triple junction on in $X_{R,\delta}^{\brho_0}$ we get immediately 
	\begin{align}
	\sum_{i=1}^3\gamma^i\partial_t\rho^i(0)=0.
	\end{align}
	For the sum of the $V^i_{\brho}$ we can argue like in the derivation of (\ref{EquationGeometricCompabilityConditionforSDFTJ}). The compatibility conditions for $\mathfrak{b}^i$ follow directly from $\mathcal{G}(\brho_0)\equiv 0$, which in total shows ii.).\\
	The last part follows from i.) and ii.) as $L$ is well-defined as long as $S(\brho)$ fulfills the linear compatibility conditions (\ref{EquationCompatibilityConditionsLinearisedTJ}).
\end{proof}
In the following we will always assume that $\varepsilon$ and $R$ are chosen such that Lemma \ref{LemmaWelldefinednessofLambdaTripleJunction} is fulfilled. Now we want to derive suitable contraction estimates for the operator $S$. Hereby, we will use the norm on $Y_{\delta}$ also when dealing with components of $S$. In the following, we will need two technical lemmas. The first one states the Lipschitz continuities of some geometrical quantities, which we will need when exploiting the quasi-linear structure of most terms. The second lemma is a representation of the scalar product of the normals on the boundary as function in the values of $\brho$ and its first order derivatives. This will be essential to use methods for fully non-linear equations from \cite{lunardi2012analytic}. For the proofs of these two properties we refer to \cite[Lemma 4.22]{goesswein2019Dissertation} and \cite[Lemma 4.24]{goesswein2019Dissertation}.
\begin{lemma}[Lipschitz continuity of geometrical quantities]\label{LemmaLipschitzcontinuitiesTJ}\ \\ 
	Suppose that $\delta,\varepsilon<1$ and $R>1$. \vspace{-0.3cm}
	\begin{itemize}
		\item[i.)] The mapping 
		\begin{align*}
		X_{R,\delta}^{\varepsilon}\to C^{4+\alpha,1+\frac{\alpha}{4}}_{TJ}(\Sigmaastdelta), \brho\mapsto \bmu(\brho),
		\end{align*} 
		is linear and Lipschitz-continuous. \vspace{-0,2cm}
		\item[ii.)] For any local parametrization $\varphi: U\to V\subset\Gamma^i_{\ast}, i=1,2,3$ the  $g_{jk}^{\brho}, g^{jk}_{\brho}$ are Lipschitz continuous as maps
		\begin{align*}
		X_{R,\delta}^{\varepsilon}\to C^{3+\alpha,\frac{3+\alpha}{4}}(U_{\delta}).
		\end{align*}
		and $N_{\brho}$ is a Lipschitz continuous function in $\brho$ as map 
		\begin{align*}
		X_{R,\delta}^{\varepsilon}\to C^{3+\alpha,\frac{3+\alpha}{4}}(\Gammaastdelta, \R^n).
		\end{align*}
		\item[iii.)] For any local parametrization $\boldsymbol{\varphi}: U\to V\subset \Sigmaast$ the $g_{jk}^{\brho}, g^{jk}_{\brho}$ and $N_{\brho}$ are Lipschitz continuous functions in $\brho$ as maps
		\begin{align*}
		X_{R,\delta}^{\varepsilon}\to C^{3+\alpha,\frac{3+\alpha}{4}}(U_{\delta}).
		\end{align*}
	\end{itemize}
	Furthermore, all arising Lipschitz constants are independent of $\delta$.
\end{lemma}  
	\begin{lemma}[Representation of $N^i_{\brho}\cdot N^j_{\brho}$]\label{LemmaRepresentationN12}\ \\
	There is a function 
	\begin{align*} 
	N^{12}: C_{\delta}\times \R\times\R\times \R^n\times\R^n\to \R
	\end{align*}
	such that we have for all $\boldsymbol{\rho}\in X_{R,\delta}$ and $(\sigma,t)\in \Sigma_{\delta}$
	\begin{align*}
	\left(N^1_{\brho}\cdot N^2_{\brho}\right)(\sigma,t)=N^{1,2}(x, t, \widehat{\rho}^1(x,t), \widehat{\rho}^2(x,t), \nabla \widehat{\rho}^1(x,t), \nabla \widehat{\rho}^2(x,t)).
	\end{align*}
	Hereby, we denote by $\widehat{\rho}^1$ resp. $\widehat{\rho}^2$ the functions $\rho^1$ resp. $\rho^2$ in local coordinates with respect to a parametrisation $\boldsymbol{\varphi}=(\varphi^1,\varphi^2,\varphi^3)$, where $\varphi^i$ is a local parametrisation of $\Gamma^i_{\ast}$ locally around $\sigma$, and $x=\varphi^{-1}(\sigma)$. Additionally, all partial derivatives of $N^{12}$ with respect to the values of $\widehat{\rho}^i$ and $\partial_j \widehat{\rho}^i$ with $i=1,2$ and $j=1,...,n$, which we denote by $\partial_{[\rho^i]}$ resp. $\partial_{[\partial_j\rho^i]}$, are in $C^{3+\alpha, \frac{3+\alpha}{4}}$ and we have
	\begin{align}\label{EquationBoundPartialDerivativesN1N2inC3+alpha}
	\|\partial N^{12}\|_{C^{3+\alpha, \frac{3+\alpha}{4}}}\le C(\Gammaast,R), 
	\end{align}
	for $\partial\in \{\partial_{[\rho^i]}, \partial_{[\partial_j\rho^i]}|i=1,2, j=1,...,n\}$.
\end{lemma}
With these auxiliary results we can now derive contraction estimates for $S$. 
\begin{lemma}[Contraction estimates for $S$]\label{LemmaContractionEstimatesTripleJunction}\ \\
	Suppose that $\delta,\varepsilon<1$ and $R>1$. Then, for all $\boldsymbol{u},\boldsymbol{w}\in X_{R,\delta}^{\brho_0}$ there is an $\bar{\alpha}\in(0,1)$ such that the following contraction estimates hold:
	\begin{align}
	\|\partial_t\boldsymbol{u}-V_{\boldsymbol{u}}-\partial_t\boldsymbol{w}+V_{\boldsymbol{w}}\|_{Y_{\delta}}&\le C(\Gammaast)(\varepsilon+R\dalpha)\|\boldsymbol{u}-\boldsymbol{w}\|_{X_{R,\delta}},\label{EquationContractionEstimate1}\\
	\|\Delta_{\boldsymbol{u}}H_{\boldsymbol{u}}-\mathcal{A}_{\text{all}}(\boldsymbol{u})-\Delta_{\boldsymbol{w}}H_{\boldsymbol{w}}+\mathcal{A}_{\text{all}}(\boldsymbol{w})\|_{Y_{\delta}}&\le C(\Gammaast)(\varepsilon+R\dalpha)\|\boldsymbol{u}-\boldsymbol{w}\|_{X_{R,\delta}}, \label{EquationContractionEstimate2}\\
	\|\mathfrak{f}(\boldsymbol{u})-\mathfrak{f}(\boldsymbol{w})\|_{Y_{\delta}}&\le C(\Gammaast)(\varepsilon+R\dalpha)\|\boldsymbol{u}-\boldsymbol{w}\|_{X_{R,\delta}}, \label{EquationContractionEstimate3} \\
	\|\mathfrak{b}^1(\boldsymbol{u})-\mathfrak{b}^1(\boldsymbol{w})\|_{Y_{\delta}}&= 0,\\
	\|\mathfrak{b}^2(\boldsymbol{u})-\mathfrak{b}^2(\boldsymbol{w})\|_{Y_{\delta}}&\le \left(C(\Gammaast,R)\dalpha+C(\Gammaast)\varepsilon\right)\|\boldsymbol{u}-\boldsymbol{w}\|_{X_{R,\delta}},\label{EquationContractionEstimate4}\\
	\|\mathfrak{b}^3(\boldsymbol{u})-\mathfrak{b}^3(\boldsymbol{w})\|_{Y_{\delta}}&\le \left(C(\Gammaast,R)\dalpha+C(\Gammaast)\varepsilon\right)\|\boldsymbol{u}-\boldsymbol{w}\|_{X_{R,\delta}},\label{EquationContractionEstimate5}\\
	\|\mathfrak{b}^4(\boldsymbol{u})-\mathfrak{b}^4(\boldsymbol{w})\|_{Y_{\delta}}&\le C(\Gammaast)R\dalpha\|\boldsymbol{u}-\boldsymbol{w}\|_{X_{R,\delta}},\label{EquationContractionEstimate6}\\
	\|\mathfrak{b}^5(\boldsymbol{u})-\mathfrak{b}^5(\boldsymbol{w})\|_{Y_{\delta}}&\le C(\Gammaast)R\dalpha\|\boldsymbol{u}-\boldsymbol{w}\|_{X_{R,\delta}},\label{EquationContractionEstimate7}\\
	\|\mathfrak{b}^6(\boldsymbol{u})-\mathfrak{b}^6(\boldsymbol{w})\|_{Y_{\delta}}&\le C(\Gammaast)R\dalpha\|\boldsymbol{u}-\boldsymbol{w}\|_{X_{R,\delta}},\label{EquationContractionEstimate8}\\
	\|S(\boldsymbol{u})-S(\boldsymbol{w})\|_{Y_{\delta}}&\le \left(C(\Gammaast,R)\dalpha+C(\Gammaast)\varepsilon\right) \|\boldsymbol{u}-\boldsymbol{w}
	\|_{X_{R,\delta}}.\label{EquationContractionEstimate9}
	\end{align}
\end{lemma}
\begin{proof}
	For the first line we note that the term on the left-hand-side equals
	\begin{align*}
	\underbrace{\partial_t\boldsymbol{u}(1-N_{\boldsymbol{u}}\cdot N_{\ast})-\partial_t\boldsymbol{w}(1-N_{\boldsymbol{w}}\cdot N_{\ast})}_{=(I)}+\underbrace{\partial_t\boldsymbol{\mu}(\boldsymbol{w})\tau_{\ast}\cdot N_{\boldsymbol{w}}-\partial_t\boldsymbol{\mu}(\boldsymbol{u})\tau_{\ast}\cdot N_{\boldsymbol{u}}}_{=(II)}
	\end{align*}
	We will discuss the terms $(I)$ and $(II)$ separately. Firstly, we rewrite (I) as
	\begin{align*}
	\partial_t\boldsymbol{u}(1-N_{\boldsymbol{u}}\cdot N_{\ast})-\partial_t\boldsymbol{u}(1-N_{\boldsymbol{w}}\cdot N_{\ast})+\partial_t\boldsymbol{u}(1-N_{\boldsymbol{w}}\cdot N_{\ast})-\partial_t\boldsymbol{w}(1-N_{\boldsymbol{w}}\cdot N_{\ast})
	\end{align*}
	and observe using Lemma \ref{LemmaProductEstimatesParabolichHspace}, Lemma \ref{LemmaContractivityLowerOrderTerms}, Lemma \ref{LemmaLipschitzcontinuitiesTJ}ii.) and (\ref{EquationEstimateforelementsofXRdelta}) that
	\begin{align*}
	\left\|\partial_t\boldsymbol{u}\left(N_{\boldsymbol{w}}-N_{\bu})\right)\cdot N_{\ast}\right\|_{Y_{\delta}}&\le \left\|\partial_t\boldsymbol{u}\right\|_{Y_{\delta}}\left\|\left(N_{\bw}-N_{\bu}\right)\cdot N_{\ast}\right\|_{Y_{\delta}}\\
	&\le C(\Gamma_{\ast})\|\boldsymbol{u}\|_{X_{R,\delta}}\dalpha\|\left(N_{\bw}-N_{\bu}\right)\cdot N_{\ast}\|_{C^{3+\alpha, \frac{3+\alpha}{4}}_{TJ}(\Gammaastdelta)}\\
	&\le C(\Gammaast)(R+\varepsilon)\dalpha\|\boldsymbol{u}-\boldsymbol{w}\|_{X_{R,\delta}}\\
	&\le C(\Gammaast)2R\dalpha\|\boldsymbol{u}-\boldsymbol{w}\|_{X_{R,\delta}}.\\
	\left\|\left(\partial_t\boldsymbol{u}-\partial_t\boldsymbol{w}\right)(1-N_{\bw}\cdot N_{\ast})\right\|_{Y_{\delta}}&\le \|\partial_t\boldsymbol{u}-\partial_t\boldsymbol{w}\|_{Y_{\delta}}\|1-N_{\bw}\cdot N_{\ast}\|_{Y_{\delta}}\\
	&\le \|\boldsymbol{u}-\boldsymbol{w}\|_{X_{R,\delta}}C(\Gammaast)\dalpha\|1-N_{\bw}\cdot N_{\ast}\|_{C^{3+\alpha,\frac{3+\alpha}{4}}_{TJ}(\Gammaastdelta)}\\
	&\le C(\Gammaast)\dalpha\|\boldsymbol{w}\|_{X_{R,\delta}}\|\boldsymbol{u}-\boldsymbol{w}\|_{X_{R,\delta}}\\
	&\le C(\Gammaast)(R+\varepsilon)\dalpha\|\boldsymbol{u}-\boldsymbol{w}\|_{X_{R,\delta}}\\
	&\le C(\Gammaast)2R\dalpha\|\boldsymbol{u}-\boldsymbol{w}\|_{X_{R,\delta}}.
	\end{align*}
	In total this implies
	\begin{align}\label{EquationProofofContractionEstimatesTJ1}
	\|(I)\|_{Y_{\delta,R}}\le C(\Gammaast)R\dalpha\|\boldsymbol{u}-\boldsymbol{w}\|_{X_{R,\delta}}.
	\end{align}
	Next, we write the term $(II)$ as
	\begin{align*}
	\partial_t\boldsymbol{\mu}(\boldsymbol{w})\tau_{\ast}\cdot N_{\bw}-\partial_t\boldsymbol{\mu}(\boldsymbol{w})\tau_{\ast}\cdot N_{\bu}+\partial_t\boldsymbol{\mu}(\boldsymbol{w})\tau_{\ast}\cdot N_{\bu}-\partial_t\bmu(\boldsymbol{u})\tau_{\ast}\cdot N_{\bu},
	\end{align*}
	and derive using Lemma \ref{LemmaLipschitzcontinuitiesTJ}i.) and ii.), Lemma \ref{LemmaProductEstimatesParabolichHspace}, Lemma \ref{LemmaContractivityLowerOrderTerms} and (\ref{EquationEstimateforelementsofXRdelta}) that
	\begin{align*}
	\|\partial_t\boldsymbol{\mu}(\boldsymbol{w})(N_{\bw}-N_{\bu})\cdot\tau_{\ast}\|_{Y_{\delta}}&\le\|\partial_t\boldsymbol{\mu}(\boldsymbol{w})\|_{Y_{\delta}} \|(N_{\bw}-N_{\bu})\cdot\tau_{\ast}\|_{Y_{\delta}}\\
	&\le \|\boldsymbol{\mu}(\boldsymbol{w})\|_{X_{R,\delta}}C(\Gammaast)\dalpha\|(N_{\bw}-N_{\bu})\cdot\tau_{\ast}\|_{C^{3+\alpha, \frac{3+\alpha}{4}}_{TJ}(\Gammaastdelta)}\\
	&\le C(\Gammaast)\|\boldsymbol{w}\|_{X_{R,\delta}}\dalpha\|\boldsymbol{u}-\boldsymbol{w}\|_{X_{R,\delta}}\\
	&\le C(\Gammaast)(R+\varepsilon)\dalpha\|\boldsymbol{u}-\boldsymbol{w}\|_{X_{R,\delta}}\\
	&\le C(\Gammaast)2R\dalpha\|\boldsymbol{u}-\boldsymbol{w}\|_{X_{R,\delta}},\\
	\|(\partial_t\boldsymbol{\mu}(\boldsymbol{w})-\partial_t\boldsymbol{\mu}(\boldsymbol{u}))N_{\bu}\cdot\tau_{\ast}\|_{Y_{\delta}}&\le \|\partial_t\boldsymbol{\mu}(\boldsymbol{u}-\boldsymbol{w})\|_{Y_{\delta}}\|N_{\bu}\cdot\tau_{\ast}\|_{Y_{\delta}}\\
	&\le C(\Gammaast)\|\boldsymbol{\mu}(\boldsymbol{u}-\boldsymbol{w})\|_{X_{R,\delta}}\left(\varepsilon+\dalpha\|N_{\bu}\cdot\tau_{\ast}\|_{C^{3+\alpha, \frac{3+\alpha}{4}}_{TJ}(\Gammaastdelta)}\right)\\
	&\le C(\Gammaast)\|\boldsymbol{u}-\boldsymbol{w}\|_{X_{R,\delta}}\left(\varepsilon+\dalpha\|\boldsymbol{u}\|_{X_{R,\delta}}\right)\\
	&\le C(\Gammaast)\|\boldsymbol{u}-\boldsymbol{w}\|_{X_{R,\delta}}\left(\varepsilon+\dalpha(R+\varepsilon)\right)\\
	&\le C(\Gammaast)\|\boldsymbol{u}-\boldsymbol{w}\|_{X_{R,\delta}}\left(\varepsilon+2\dalpha R\right).
	\end{align*}
	From this we conclude
	\begin{align}\label{EquationProofofContractionEstimatesTJ2}
	\|(II)\|_{Y_{\delta}}\le C(\Gammaast)(\varepsilon+\dalpha R)\|\boldsymbol{u}-\boldsymbol{w}\|_{X_{R,\delta}}.
	\end{align}
	The estimates \eqref{EquationProofofContractionEstimatesTJ1} and \eqref{EquationProofofContractionEstimatesTJ2} together imply \eqref{EquationContractionEstimate1}.\\
	For \eqref{EquationContractionEstimate2} we note that the highest order terms  on the left-hand side are given in local coordinates by
	\begin{align*}
	&\left(g^{jk}_{\bu}g^{lm}_{\bu}(N_{\bu}\cdot N_{\ast})\right)\partial_{jklm}\boldsymbol{u}-\left(g^{jk}_{\bw}g^{lm}_{\bw}(N_{\bw}\cdot N_{\ast})\right)\partial_{jklm}\boldsymbol{w}-g^{jk}_{\ast}g^{lm}_{\ast}\partial_{jklm}(\boldsymbol{u}-\boldsymbol{w})\\
	+&\left(g^{jk}_{\bu}g^{lm}_{\bu}(N_{\bu}\cdot \tau_{\ast})\right)\partial_{jklm}\boldsymbol{\mu}(\boldsymbol{u})-\left(g^{jk}_{\bw}g^{lm}_{\bw}(N_{\bw}\cdot \tau_{\ast})\right)\partial_{jklm}\boldsymbol{\mu}(\boldsymbol{w}).
	\end{align*}
	We abbreviate the local terms in the first line by $(I)$ and the non-local terms in the second line by $(II)$. First, we rewrite $(I)$ as
	\begin{align*}
	&-g_{\ast}^{jk}g_{\ast}^{lm}\partial_{jklm}(\boldsymbol{u}-\boldsymbol{w})+g^{jk}_{\bu}g^{lm}_{\bu}\left(N_{\bu}\cdot N_{\ast}\right)\partial_{jklm}(\boldsymbol{u}-\boldsymbol{w})
	-g^{jk}_{\bu}g^{lm}_{\bu}\left(N_{\bu}\cdot N_{\ast}\right)\partial_{jklm}(\boldsymbol{u}-\boldsymbol{w})\\
	&+g^{jk}_{\bu}g^{lm}_{\bu}\left(N_{\bu}\cdot N_{\ast}\right)\partial_{jklm}\boldsymbol{u}-g^{jk}_{\bw}g^{lm}_{\bw}\left(N_{\bw}\cdot N_{\ast}\right)\partial_{jklm}\boldsymbol{w}\\
	&=\underbrace{\left(-g_{\ast}^{jk}g_{\ast}^{lm}+g^{jk}_{\bu}g^{lm}_{\bu}(N_{\bu}\cdot N_{\ast})\right)\partial_{jklm}(\boldsymbol{u}-\boldsymbol{w})}_{(A)}+\underbrace{\left(g^{jk}_{\bu}g^{lm}_{\bu}\left(N_{\bu}\cdot N_{\ast}\right)-g^{jk}_{\bw}g^{lm}_{\bw}\left(N_{\bw}\cdot N_{\ast}\right)\right)\partial_{jklm}\boldsymbol{w}}_{(B)}.
	\end{align*}
	Now, we observe that due to Lemma \ref{LemmaLipschitzcontinuitiesTJ}i.) the function 
	\begin{align*} X_{R,\delta}^{\varepsilon}&\to C^{3+\alpha, \frac{3+\alpha}{4}}(\Gammaastdelta),\\
	\brho&\mapsto g^{jk}_{\brho}g^{lm}_{\brho}(N_{\brho}\cdot N_{\ast}),
	\end{align*} is Lipschitz continuous and the evaluation at $\brho\equiv 0$ equals $g_{\ast}^{jk}g_{\ast}^{lm}$. Thus, we get
	\begin{align*}
	\|(A)\|_{Y_{\delta}}&\le C(\Gammaast)\dalpha\|g^{jk}_{\bu}g^{lm}_{\bu}(N_{\bu}\cdot N_{\ast})-g_{\ast}^{jk}g_{\ast}^{lm}\|_{C^{3+\alpha, \frac{3+\alpha}{4}}_{TJ}(\Gammaastdelta)}\|\boldsymbol{u}-\boldsymbol{w}\|_{X_{R,\delta}}\\
	&\le C(\Gammaast)\dalpha\|\boldsymbol{u}\|_{X_{R,\delta}}\|\bu-\bw\|_{X_{R,\delta}}\\
	&\le C(\Gammaast)\dalpha(R+\varepsilon)\|\bu-\bw\|_{X_{R,\delta}}\\
	&\le C(\Gammaast)\dalpha 2R\|\bu-\bw\|_{X_{R,\delta}},\\
	\|(B)\|_{Y_{\delta}}&\le C(\Gammaast)\dalpha\|g^{jk}_{\bu}g^{lm}_{\bu}\left(N_{\bu}\cdot N_{\ast}\right)-g^{jk}_{\bw}g^{lm}_{\bw}\left(N_{\bw}\cdot N_{\ast}\right)\|_{C^{3+\alpha, \frac{3+\alpha}{3}}_{TJ}(\Gammaastdelta)}\|\boldsymbol{w}\|_{X_{R,\delta}}\\
	&\le C(\Gammaast)\dalpha\|\boldsymbol{u}-\boldsymbol{w}\|_{X_{R,\delta}}2R.
	\end{align*}
	Together this implies
	\begin{align}\label{EquationContractionEstimatesTJ3}
	\|(I)\|_{Y_{\delta}}\le C(\Gammaast)\dalpha R\|\boldsymbol{u}-\boldsymbol{w}\|_{X_{R,\delta}}.
	\end{align}
	Now, we write the term $(II)$ as
	\begin{align*}
	&\underbrace{g^{jk}_{\bu}g^{lm}_{\bu}\left(N_{\bu}\cdot\tau_{\ast}\right)\partial_{jklm}\boldsymbol{\mu}(\boldsymbol{u})-g^{jk}_{\bu}g^{lm}_{\bu}\left(N_{\bu}\cdot\tau_{\ast}\right)\partial_{jklm}\boldsymbol{\mu}(\boldsymbol{w})}_{(A)}\\
	+&\underbrace{g^{jk}_{\bu}g^{lm}_{\bu}\left(N_{\bu}\cdot\tau_{\ast}\right)\partial_{jklm}\boldsymbol{\mu}(\boldsymbol{w})-g^{jk}_{\bw}g^{lm}_{\bw}\left(N_{\bw}\cdot\tau_{\ast}\right)\partial_{jklm}\boldsymbol{\mu}(\boldsymbol{w})}_{(B)}.
	\end{align*}
	Using Lemma \ref{LemmaLipschitzcontinuitiesTJ}i.) we see that the function
	\begin{align*}
	X_{R,\delta}&\to C^{3+\alpha,\frac{3+\alpha}{4}}_{TJ}(\Gammaastdelta)\\
	\brho&\mapsto g^{jk}_{\brho}g^{lm}_{\brho}\left(N_{\brho}\cdot\tau_{\ast}\right),
	\end{align*} 
	is Lipschitz continuous and the evaluation at $\brho\equiv 0$ equals $0$ . Thus, we can argue with similar arguments as for the term $(I)$ and so derive
	\begin{align}\label{EquationContractionEstimatesTJ4}
	\|(II)\|_{Y_{\delta}}\le C(\Gammaast)(\varepsilon+R\dalpha)\|\boldsymbol{u}-\boldsymbol{w}\|_{X_{R,\delta}}. 
	\end{align}
	Combining the estimates (\ref{EquationContractionEstimatesTJ3}) and (\ref{EquationContractionEstimatesTJ4}) we get (\ref{EquationContractionEstimate2}) and then together with $(\ref{EquationContractionEstimate1})$ we conclude $(\ref{EquationContractionEstimate3})$.\\
	Next we have to deal with the boundary operator $\mathfrak{b}$. We will only discuss the analysis of $\mathfrak{b}^2$ resp. $\mathfrak{b}^3$. The others therms of $\mathfrak{b}$ have all a quasi-linear structure and so similar arguments as before can be applied. In contrary, this is not possible for $\mathfrak{b}^2$ resp. $\mathfrak{b}^3$ as they are fully non-linear. Instead, we will use ideas from \cite[Chapter 8]{lunardi2012analytic}. Using the function $N^{12}$ from Lemma \ref{LemmaRepresentationN12} we write for any $\boldsymbol{u},\boldsymbol{w}\in X_{R,\delta}^{\brho_0}$ 
	\begin{align*}
	\mathfrak{b}^2(\boldsymbol{u})-\mathfrak{b}^2(\boldsymbol{w})&=\sum_{i=1}^2\Theta^i(u^i-w^i)+\sum_{i=1}^2\sum_{j=1}^n\Theta^i_j\partial_j(u^i-w^i)\\
	&+\sum_{i=1}^2\Theta^i_0(u^i-w^i)+\sum_{i=1}^2\sum_{j=1}^n\Theta_{j;0}^i\partial_j(u^i-w^i),
	\end{align*}
	where we used the abbreviations
	\begin{align*}
	\Theta^i&:=\int_0^1\partial_{[\rho^i]}N^{12}(s\boldsymbol{w}+(1-s)\boldsymbol{u})-\partial_{[\rho^i]}N^{12}(\boldsymbol{\rho_0})ds,\\
	\Theta^i_j&:=\int_0^1\partial_{[\partial_j\rho^i]}N^{12}(s\boldsymbol{w}+(1-s)\boldsymbol{u})-\partial_{[\partial_j\rho^i]}N^{12}(\boldsymbol{\rho_0})ds,\\
	\Theta^i_0&:=\int_0^1\partial_{[\rho^i]}N^{12}(s\boldsymbol{\rho}_0)-\partial_{[\rho^i]}N^{12}(0)ds,\\
	\Theta^i_{j;0}&:=\int_0^1\partial_{[\partial_j\rho^i]}N^{12}(s\boldsymbol{\rho_0})-\partial_{[\partial_j\rho^i]}N^{12}(0)ds.
	\end{align*}
	Hereby, we used that as $\mathcal{B}^2$ is the pointwise linearization of $N^1_{\brho}\cdot N^2_{\brho}$ we can also write it in terms of $\partial N^{12}$. Furthermore, all the arising functions $\Theta$ are also in $C^{3+\alpha, \frac{3+\alpha}{4}}$ due to the theory of parameter integrals. Additionally, their norms are also bound by a constant $C(\Gammaast,R)$. As $$s\boldsymbol{w}+(1-s)\boldsymbol{u}\big|_{t=0}=\boldsymbol{\rho}_0$$ we get $\Theta\big|_{t=0}=0$ for $\Theta\in \{\Theta^i, \Theta^i_j|i=1,2, j=1,...,n\}$ and therefore we conclude 
	\begin{align*}
	\|\Theta\|_{C^{3,0}}\le\delta^{\frac{\alpha}{4}}\langle \Theta\rangle_{t, \frac{\alpha}{4}}\le \delta^{\frac{\alpha}{4}} C(\Gammaast,R),
	\end{align*}
	as all derivatives of $\Theta$ are at least in $C^{0,\frac{\alpha}{4}}$. Additionally, all $\Theta$ inherit the bound (\ref{EquationBoundPartialDerivativesN1N2inC3+alpha}) as it holds uniformly in $s$. From this we deduce
	\begin{align*}
	\|\Theta D(u^i-w^i)\|_{C^{3+\alpha, \frac{3+\alpha}{4}}}&\le \|\Theta^i\|_{C^{3+\alpha,\frac{3+\alpha}{4}}}\|D(u^i-w^i)\|_{C^{3,0}}+\|\Theta^i\|_{C^{3,0}}\|D(u^i-w^i)\|_{C^{3+\alpha, \frac{3+\alpha}{4}}}\\
	&\le C(\Gammaast,R)\|D(u^i-w^i)\|_{C^{3,0}}+C(\Gammaast,R)\delta^{\frac{\alpha}{4}}\|D(u^i-w^i)\|_{C^{3+\alpha,\frac{3+\alpha}{4}}}\\
	&\le C(\Gammaast,R)\delta^{\frac{\alpha}{4}}\|\boldsymbol{u}-\boldsymbol{w}\|_{C^{4+\alpha, 1+\frac{\alpha}{4}}}, 
	\end{align*}
	where $D$ denotes the to $\Theta$ corresponding differential operator, which is the identity for $\Theta^i$ and $\partial_j$ for $\Theta^i_j$. For the $\Theta^i_0$ and $\Theta^i_{j,0}$ we can use that due to the work in the proof of Lemma \ref{LemmaRepresentationN12} we have Lipschitz continuity\footnote{Note that the Lipschitz constant here is independent of $R$ as we need the $\Theta_0$ only as a function in the initial data!} of the $\partial N^{12}$ as maps $C^{4+\alpha, 1+\frac{\alpha}{4}}\to C^{3+\alpha, \frac{3+\alpha}{4}}$ and this implies then that
	\begin{align*}
	\|\Theta_0\|_{C^{3+\alpha, \frac{3+\alpha}{4}}}\le C(\Gammaast)\|\boldsymbol{\rho}_0\|_{C^{4+\alpha}}\le C(\Gammaast)\varepsilon,
	\end{align*}
	for $\Theta_0\in \{\Theta^i_0,\Theta^i_{j,0}|i=1,2,j=1,...,n\}$ and consequently
	\begin{align*}
	\|\Theta_0D(u^i-w^i)\|_{C^{3+\alpha, \frac{3+\alpha}{4}}}\le C(\Gammaast)\varepsilon\|\boldsymbol{u}-\boldsymbol{w}\|_{X_{R,\delta}},
	\end{align*}
	where $D$ denotes again the differential operator matching to the choice of $\Theta_0$. In total we deduce
	\begin{align*}
	\|\mathfrak{b}^2(\boldsymbol{u})-\mathfrak{b}^2(\boldsymbol{w})\|_{Y_{\delta}}\le (C(\Gammaast, R)\delta^{\frac{\alpha}{4}}+C(\Gammaast)\varepsilon)\|\boldsymbol{u}-\boldsymbol{w}\|_{X_{R,\delta}}.
	\end{align*}
	For $\mathfrak{b}^3$ we can argue analogously and thus we conclude \eqref{EquationContractionEstimate4} and \eqref{EquationContractionEstimate5}. This finishes the analysis of the contraction estimates.
\end{proof}
From this we may now conclude that $\Lambda$ is a contraction mapping for suitable $\varepsilon$ and $\delta$.
\begin{corollary}[Contraction property of $\Lambda$]\label{CorollaryContractivityTripleJunction}\ \\
	There is an $\varepsilon_0<\min(1,\varepsilon_W)$ with the following property: for any $R>1$ and $\varepsilon<\varepsilon_0$ there is a $\delta(R,\varepsilon)>0$ such that 
	\begin{align*}
	\Lambda: X_{R,\delta}^{\brho_0}\to C^{4+\alpha,1+\frac{\alpha}{4}}_{TJ}(\Gammaastdelta)
	\end{align*}
	is a $\frac{1}{2}$-contraction.
\end{corollary}
\begin{proof}
	Let $\boldsymbol{v},\bw\in X_{R,\delta}^{\brho_0}$. We observe that $\Lambda(\boldsymbol{v})-\Lambda(\bw)$ solves (\ref{EquationLinearisedSurfaceDiffusionTripleJunctionAbstractFormulaton}) with 
	\begin{align*}
	(\mathfrak{f},\mathfrak{b})=S(\boldsymbol{v}-\bw),\quad \boldsymbol{u}_0\equiv 0.
	\end{align*}
	Suppose now that $\varepsilon,\delta<1$. Then, we can apply the energy estimate (\ref{EquationEnergieabschatzungenLineareGleichungVollTJ}) together with Lemma \ref{LemmaContractionEstimatesTripleJunction} to derive
	\begin{align}
	\|\Lambda(\boldsymbol{v})-\Lambda(\boldsymbol{w})\|_{X_{R,\delta}}\le \left(C_1(\Gammaast,R)\dalpha+C_2(\Gammaast)\varepsilon\right)\|\boldsymbol{v}-\bw\|_{X_{R,\delta}},
	\end{align}
	with suitable constants $C_1(\Gammaast,R)$ and $C_2(\Gammaast)$.
	Now choosing 
	\begin{align}
	\varepsilon\le \frac{1}{4_2C(\Gammaast)},\quad \delta\le \left(\frac{1}{4C_1(\Gammaast,R)}\right)^{\bar{\alpha}^{-1}},
	\end{align}
	we get the sought property for $\Lambda$.
\end{proof}
Finally, we need that $\Lambda$ is also a self-mapping, which we can guarantee as long as $R$ is sufficiently large.
\begin{lemma}[Self-mapping property of $\Lambda$]\label{LemmaSelfMappingofLambdaTripleJunction}\ \\
	For given $R>1$ and $\varepsilon<\varepsilon_0$ let $\delta(R,\varepsilon)$ be chosen as in Corollary \ref{CorollaryContractivityTripleJunction}. There is an $R_0>0$ such that for all $R>R_0$ the map $\Lambda$ is a self-mapping on $X_{R,\delta}^{\brho_0}$.
\end{lemma}
\begin{proof}
	For any $\bu\in X_{R,\delta}^{\brho_0}$ we get
	\begin{align*}
	\|\Lambda(\boldsymbol{u})-\boldsymbol{\rho}_0\|_{C^{4+\alpha,\frac{1+\alpha}{4}}_{TJ}(\Gammaastdelta)}&\le \|\Lambda(\boldsymbol{u})-\Lambda(\boldsymbol{\rho}_0)\|_{C^{4+\alpha,\frac{1+\alpha}{4}}_{TJ}(\Gammaastdelta)}+\|\Lambda(\boldsymbol{\rho}_0)-\boldsymbol{\rho}_0\|_{C^{4+\alpha,\frac{1+\alpha}{4}}_{TJ}(\Gammaastdelta)}\\
	&\le \frac{R}{2}+\|\Lambda(\boldsymbol{\rho}_0)-\boldsymbol{\rho}_0\|_{C^{4+\alpha,\frac{1+\alpha}{4}}_{TJ}(\Gammaastdelta)},
	\end{align*}
	where we used that $\Lambda$ is a $\frac{1}{2}$-contraction on $X_{R,\delta}^{\brho_0}$. As we want $R$ to be independent of $\brho_0$ we want to find an estimate for the second summand that is uniformly in $\varepsilon$. For this, we note that the function $\boldsymbol{w}:=\Lambda(\boldsymbol{\rho}_0)-\boldsymbol{\rho}_0$ solves the system
	\begin{align}
	\partial_t w^i &=\mathcal{A}_{all}^iw^i+\mathfrak{f}^i_{0}(\boldsymbol{\rho}_0) &  &\text{on }\Gamma_{\ast,\delta}^i, i=1,2,3,\notag\\
	\mathcal{B}\boldsymbol{w}&=0 & &\text{on }\Sigma_{\ast,\delta}, \label{EquationSystemForSelfMappingTJ}\\
	w^i\big|_{t=0}&=0 & &\text{on }\Gamma_{\ast}^i, i=1,2,3,\notag
	\end{align}
	with the inhomogeneity 
	\begin{align}
	\mathfrak{f}^i_0(\boldsymbol{\rho}_0):=\Delta_{\brho_0}H_{\brho_0}^i.
	\end{align}
	We want to give a short explanation of this. As $\Lambda(\boldsymbol{\rho}_0)$ and $\boldsymbol{\rho}_0$ have the same initial data their difference vanishes at $t=0$. The boundary inhomogeneity operator from (\ref{EquationInhomogeneityOperatorBoundary}) does not explicitly depend on the time and so we get
	\begin{align}\label{Equation81020181}
	\mathcal{B}(\Lambda(\brho_0))(t)=\mathcal{B}(\Lambda(\brho_0))(0) \quad \forall t\in[0,\delta].
	\end{align}
	Furthermore, due to the compatibility conditions (\ref{EquationGeometricCompabilityConditionforSDFTJ}) for $\brho_0$ we get
	\begin{align}\label{Equation81020182}
	\mathcal{B}(\Lambda(\brho_0))\big|_{t=0}=\mathcal{B}(\brho_0)\big|_{t=0}.
	\end{align}
	Combining (\ref{Equation81020181}) and (\ref{Equation81020182}) we derive $(\ref{EquationSystemForSelfMappingTJ})_2$.\\
	Finally, recalling (\ref{EquationDefinitionInhomeneityf}) we see that
	\begin{align}
	\mathfrak{f}^i(\brho_0)=\Delta_{\brho_0}H^i_{\brho_0}-\mathcal{A}^i_{all}(\brho_0^i,\brho_0\big|_{\Sigmaast})
	\end{align}
	Additionally, $\brho_0$ solves $(\ref{EquationLinearisedSurfaceDiffusionTripleJunctionAbstractFormulaton})_1$ with $\mathfrak{f}^i=\mathcal{A}^i_{all}(\brho_0^i,\brho\big|_{\Sigmaast})$. Together this shows now also $(\ref{EquationSystemForSelfMappingTJ})_1$.\\
Now, we can apply Theorem \ref{theoremShorttimeexistenceLinearisedSDFTJGeneralInitialdata} as due to the geometric compatibility conditions \eqref{EquationGeometricCompabilityConditionforSDFTJ} condition (\ref{EquationCompatibilityConditionsLinearisedTJ}) is fulfilled for $\mathfrak{f}^i_0(\brho_0)$ and therefore we get using  (\ref{EquationEnergieabschatzungenLineareGleichungVollTJ}) that
	\begin{align*}
	\|\boldsymbol{w}\|_{C^{4+\alpha, 1+\frac{\alpha}{4}}_{TJ}(\Gamma_{\ast,\delta})}\le C\sum_{i=1}^3\|\mathfrak{f}^i_0(\boldsymbol{\rho})\|_{C^{4+\alpha, 1+\frac{\alpha}{4}}(\Gamma_{\ast,\delta}^i)}\le C'(\varepsilon).
	\end{align*}
	Here, we used that $\Delta_{\brho}H_{\brho}$ depends continuously on derivatives of up to order four. This leads us now to
	\begin{align*}
	\|\Lambda(\boldsymbol{u})-\boldsymbol{\rho}_0\|_{C^{4+\alpha, 1+\frac{\alpha}{4}}_{TJ}(\Gamma_{\ast,\delta})}\le \frac{R}{2}+C'(\varepsilon).
	\end{align*}
	By choosing $R_0>2C'(\varepsilon)$ we get $\Lambda(\boldsymbol{u})\in X_{R,\delta}^{\brho_0}$ for all $R>R_0$ and so this shows that $\Lambda$ is a self-mapping. 
\end{proof}
\begin{proof}(of Proposition \ref{TheoremSTESurfaceDiffTripleJunction})
	Lemma \ref{LemmaWelldefinednessofLambdaTripleJunction} guarantees well-definedness of $\Lambda$. Then, Corollary \ref{CorollaryContractivityTripleJunction} and Lemma \ref{LemmaSelfMappingofLambdaTripleJunction} show that $\Lambda$ is a self-mapping on $X_{R,\delta}^{\brho_0}$ and a $\frac{1}{2}$-contraction. We then can apply Banach's fixed-point theorem to get the existence of a unique fixed-point of $\Lambda$ in $X_{R,\delta}^{\brho_0}$.
\end{proof}
From this we conclude immediately the existence result in Theorem \ref{TheoremSTETripleJunctions}. Observe that due to our efforts in the proof of Lemma \ref{LemmaSelfMappingofLambdaTripleJunction} the existence time and the bound for the solution by the Radius $R$ is uniformly in $\brho_0$. Also, every solution in a ball with a a radius $R>R_0$ will corresponds with the solution just found. This can pe proven with a standard argument for every approach with Banach's fixed point theorem, cf. \cite[Remark 3.14]{goesswein2019Dissertation}. So in total this finishes now our proof of Theorem \ref{TheoremSTETripleJunctions}.

\section*{Acknowledgement}
The second author was partially funded by the DFG
through the Research Training Group GRK 1692 \textit{Curvature, Cycles, and Cohomology} in Regensburg. The support is gratefully acknowledged.

\bibliography{BibFileSTESurfDiffTripJunc} 
\bibliographystyle{acm}
\end{document}